%% file: Geometry_of_Bing_Involution.tex
\documentclass[12pt]{amsart}
\usepackage[margin=1in]{geometry}
\usepackage{mathptmx}
\usepackage{mathtools}
\usepackage{amsthm}
\usepackage{amssymb}
\usepackage[alphabetic]{amsrefs}
\usepackage{graphicx}
\usepackage{tikz}
	\usetikzlibrary{decorations.pathreplacing}
	\usetikzlibrary{decorations.pathmorphing}
	\usetikzlibrary{patterns}
	\usetikzlibrary{calc}
\usepackage{caption}
\usepackage{euscript}
\usepackage[new]{old-arrows}

\newcommand{\ra}{\rightarrow}
\newcommand{\pr}{\prime}
\newcommand{\de}{\partial}
\newcommand{\te}{\theta}

\newcommand{\R}{\mathbb{R}}
\newcommand{\Z}{\mathbb{Z}}

\newcommand{\abs}[1]{\left\lvert #1 \right\rvert}
\newcommand{\tld}[1]{\widetilde{#1}}
\newcommand{\lbar}[1]{\overline{#1}}
\newcommand{\id}{\mathrm{id}}
\newcommand{\indep}{\perp \!\!\! \perp}

\renewcommand{\coprod}{\rotatebox[origin = c]{180}{$\prod$}}

\newtheorem{thm}{Theorem}
\newtheorem{innercustomthm}{Theorem}
\newenvironment{customthm}[1]
  {\renewcommand\theinnercustomthm{#1}\innercustomthm}
  {\endinnercustomthm}
\newtheorem{cor}{Corollary}[section]
\newtheorem{lemma}[cor]{Lemma}

\newtheorem{scholium}[cor]{Scholium}
\newtheorem{innercustomlm}{}
\newenvironment{customlm}[1]
  {\renewcommand\theinnercustomlm{#1}\innercustomlm}
  {\endinnercustomlm}
\newtheorem{prop}[cor]{Proposition}
\newtheorem{conj}{Conjecture}
\theoremstyle{definition}
\newtheorem{defin}{Definition}[section]
\newtheorem*{defin*}{Definition}
\theoremstyle{remark}
\newtheorem*{note}{Note}
\newtheorem{remark}{Remark}[section]

\numberwithin{equation}{section}
\numberwithin{figure}{section}

\begin{document}

\title{The Geometry of the Bing Involution}

\author{Michael Freedman}
\address{\hskip-\parindent
	Michael Freedman \\
	Microsoft Research, Station Q, and Department of Mathematics \\
	University of California, Santa Barbara \\
	Santa Barbara, CA 93106
}

\author{Michael Starbird}
\address{\hskip-\parindent
	Michael Starbird \\
	Department of Mathematics \\
	The University of Texas at Austin \\
	Austin, TX 78712
}

\begin{abstract}
	In 1952 Bing published a wild (not topologically conjugate to smooth) involution $I$ of the 3-sphere $S^3$. But exactly how wild is it, analytically? We prove that any involution $I^h$, topologically conjugate to $I$, must have a nearly exponential modulus of continuity. Specifically, given any $\alpha>0$, there exists a sequence of $\delta$'s converging to zero, $\delta > 0$, and points $x,y \in S^3$ with dist$(x,y) < \delta$, yet dist$(I^h(x), I^h(y)) > \epsilon$, where $\delta^{-1} = e^{\left(\frac{\epsilon^{-1}}{\log^{(1+\alpha)}(\epsilon^{-1})}\right)}$, and dist is the usual Riemannian distance on $S^3$. In particular, $I^h$ stretches distance much more than a Lipschitz function ($\delta^{-1} = c\epsilon^{-1}$) or a H\"{o}lder function ($\delta^{-1} = c^\pr(\epsilon^{-1})^{p}$, $1 < p < \infty$). Bing's original construction and known alternatives (see text) for $I$ have a modulus of continuity $\delta^{-1} > c \sqrt{2}^{\epsilon^{-1}}$, so the theorem is reasonably tight---we prove the modulus must be at least exponential up to a polylog, whereas the truth may be fully exponential. Actually, the functional for $\delta^{-1}$ coming out of the proof can be chosen slightly closer to exponential than stated here (see Theorem \ref{paper-thm}). Using the same technique we analyze a large class of ``ramified'' Bing involutions and show, as a scholium, that given any function $f: \R^+ \ra \R^+$, no matter how rapid its growth, we can find a corresponding involution $J$ of the 3-sphere such that any topological conjugate $J^h$ of $J$ must have a modulus of continuity $\delta^{-1}(\epsilon^{-1})$  growing faster than $f$ (near infinity). There is a literature on inherent differentiability (references in text) but as far as the authors know the subject of inherent modulus of continuity is new.

	\vspace{1em}
	\noindent
	Dedicated to R.H.\ Bing's life work on the 70th anniversary of his involution.
\end{abstract}

\maketitle

\section{Introduction}
Great examples breathe life into topology. Some, like Milnor's exotic 7-spheres, you \emph{learn}; others, like Bing's wild involution on the 3-sphere $S^3$, you \emph{see}. But what, exactly, are we looking at? As we will explain, Bing's involution $I$ is not really one specific example, but a topological conjugacy class $\{I^h = h^{-1} \circ I \circ h,\ h: S^3 \ra S^3 \text{ a homeomorphism}\}$. That is, there is no one preferred coordinate system in which to study its properties. This paper estimates a geometric feature: inherent modulus of continuity (imoc), which is conserved under change of coordinates. It describes the metrical distortion at finer and finer scales which must be a shared feature over the entire conjugacy class $\{I^h\}$.

Modern readers may not be familiar with the topological category, so some context is in order. An involution is a homeomorphism of order two: $I^2 = \id$. Bing's involution $I$ is wild in the sense that, regardless of coordinates, it cannot be made smooth, $C^1$, that is, $I^h$ can never be a diffeomorphism. Such wild examples are often the result of taking the limit of an infinite process. Newton taught us that very often, infinite processes have smooth limits---these are the subject of analysis. But often they do not; in such cases, wild topology may result. To make an analogy to condensed matter physics, lattice models often have field theories as their low energy continuum limit, but many times \cite{haah13} they do not. In both cases the subjects must proceed along two different tracks depending on the regularity of the limit. Wild topology is closely related to decomposition space theory \cite{daverman86}, the study of quotient maps in the favorable case where the quotient space, as well as the source space, is metrizable (see Section 1.1), and this is the route to the Bing involution.

Bing brought decomposition space theory to the masses with his landmark paper \cite{bing52} proving that the double of the (closed) Alexander Horned sphere complement is homeomorphic to $S^3$. The doubled structure immediately yields Bing's involution $I$. His ideas eventually contributed to the resolution of the double suspension conjecture \cites{edwards80,cannon79} and the 4-dimensional Poincar\'{e} conjecture \cites{freedman82,fq90}. But for 70 years the original construction, Bing's wild involution on $S^3$, has retained certain mysteries. On and off for 40 years the authors have wondered if the involution could be made Lipschitz. In recent years \cite{onninen19}, its compatibility with other analytic structures (Sobolev conditions) has also been studied and the Lipschitz question has appeared in a problem list \cite{heinonen97}. In this paper we prove that, in any coordinates, the Bing involution $I^h$ is much less regular than even a H\"{o}lder-continuous function.

A function between compact metric spaces $k: X \ra Y$ is continuous iff $\forall \epsilon > 0$, $\exists \delta > 0$ s.t.\ dist$(x,x^\pr) < \delta \implies \operatorname{dist}(k(x),k(x^\pr)) < \epsilon$. We can define a $\delta$-$\epsilon$ function, $\delta(\epsilon)$, that returns for each $\epsilon > 0$ the supremum of those numbers $\delta$ such that dist$(x,x^\pr) < \delta$ implies dist$(k(x),k(x^\pr)) < \epsilon$. In this paper we will be examining how far nearby points must be stretched under, in our case, involutions of $S^3$. To make ``large stretching'' correspond to a rapidly growing function, we choose to write $\delta^{-1}$ as a function of $\epsilon^{-1}$, $\delta^{-1}(\epsilon^{-1}) \coloneqq \frac{1}{\delta(1/\epsilon^{-1})} = \frac{1}{\delta(\epsilon)}$.

This $\delta^{-1}(\epsilon^{-1})$ function captures the concept of \emph{modulus of continuity of k}, moc$(k)$, and we are primarily interested in describing the behavior of this $\delta^{-1}(\epsilon^{-1})$ moc function as $\epsilon^{-1}$  approaches infinity and $\epsilon$ approaches $0$. So describing how $\delta^{-1}(\epsilon^{-1})$ behaves as $\epsilon^{-1}$ goes to infinity corresponds to investigating how small $\delta$ must be in relation to $\epsilon$ in the definition of continuity as $\epsilon$ approaches $0$. 

As an example, $h$ is Lipschitz iff there exists a constant $c$ such that $\delta^{-1}(\epsilon^{-1}) \leq c\epsilon^{-1}$. And $h$ is H\"{o}lder iff there exists a $p>1$ such that for sufficiently large $\epsilon^{-1}$, $\delta^{-1}(\epsilon^{-1}) \leq (\epsilon^{-1})^p$.

Our main result is:

\begin{thm}\label{paper-thm}
	If $I^h$ is any involution of $S^3$ in the topological conjugacy class of the Bing involution, then for any $j=0,1,2\dots$ and any $\alpha>0$, there is some sequence of $\epsilon$'s converging to zero so that the corresponding $\delta$'s must be chosen so that
	\[
		\delta^{-1}(\epsilon^{-1}) \geq e^{\left(\frac{\epsilon^{-1}}{L_j^\alpha(\epsilon^{-1})}\right)}
	\]
    where $L_0^\alpha \coloneqq \log^{(1+\alpha)}(x)$, $L_1^\alpha(x) \coloneqq \log(x)(\log \log(x))^{(1+\alpha)}$, ..., and

    \noindent
    $L_j^\alpha = \log(x)(\log \log(x)) \cdots (\underbrace{\log \cdots \log}_{(j+1)\text{-times}}(x))^{(1+\alpha)}$. All logs may be assumed base $e$. That is, $I^h$ has a nearly exponential modulus of continuity (moc). We say $I^h$ has a nearly exponential inherent modulus of continuity (imoc).
\end{thm}

We call a moc established over the entire topological conjugacy class inherent (imoc). That is, we take the imoc to be the slowest growing $\delta^{-1}(\epsilon^{-1})$ function over a conjugacy class. This is not entirely precise since two conjugates $I^h$ and $I^{h^\pr}$ could have their mocs crossing back and forth as $\epsilon \ra 0$. For this reason we usually speak of moc$(I^h) \overset{\infty}{\geq} f$, meaning for some sequence of $\epsilon \ra 0$, $\delta^{-1}_{I^h}(\epsilon^{-1}) \geq f(\epsilon^{-1})$; and if this holds for all conjugates $h$, we say imoc$(I^h) \overset{\infty}{\geq} f$.

In this paper we define imoc by minimizing over the conjugacy class but leave the ambient Riemannian metric fixed. (In our case, the ambient metric is that of the unit 3-sphere.) One could go further and also minimize over Riemannian metrics, and even minimize over the underlying differential structures on the ambient manifold $M$ as well. When dim$(M) = 4$, there may be countably many such structures (up to diffeomorphism), even for $M$ compact, and this might be interesting to do. To illustrate the simplifying effect of minimizing over the Riemannian metrics as well as the conjugacy class, we claim that for a finite group action on a 2D surface the imoc $= \id$ (simultaneously for all group elements). Dimension 2 is too low for the wildness; the action is necessarily conjugate to a smooth action. Then, averaging the Riemannian metric realizes the claim. However, in this paper we fix the standard metric on $S^3$.

\subsection*{Proof strategy for Theorem \ref{paper-thm}}
We study the geometry of the fixed 2-sphere, FIX, of $I^h$, or more precisely a smooth approximation thereof, to locate a sequence of planar disks $\Delta_i$ with the properties (1) $\de \Delta_i \subset$ FIX, (2) $a_i \coloneqq \operatorname{area}(\Delta_i)$, and (3) There is a point $y \in \operatorname{int}(\Delta_i)$ such that $\operatorname{dist}(y,I^h(y)) > d_i$ where upper bounds on $a_i$ and lower bounds on $d_i$ are obtained. In particular, we show $a_i$ decays roughly like $\frac{1}{2^i}$ while $d_i$ decreases no faster than $\frac{1}{i}$, up to log factors. For any $y$ in such a disk $\Delta$ of area $a$, there must be a point $x$ on $\de \Delta$ with $\operatorname{dist}(x,y) < \sqrt{\frac{a}{\pi}}$.

Because $x$ is fixed, we have then found pairs of points $(x_i,y_i)$ with $\operatorname{dist}(x_i,y_i) \coloneqq \delta_i \approx \frac{1}{\sqrt{2}^i}$ and $\operatorname{dist}(I^h(x_i), I^h(y_i)) \coloneqq \epsilon_i \approx \frac{1}{i}$. This yields the claimed relationship between the $\epsilon$ and $\delta$ in the definition of continuity: $\delta^{-1}(\epsilon^{-1}) \approx \sqrt{2}^{(\epsilon^{-1})}$. Some details of the proof introduce a poly log in the exponent, but we conjecture that, in truth, $\delta^{-1}$ is fully exponential in $\epsilon^{-1}$.

We use the term \emph{$(a,d)$-stretching} for the imputed distortion of distance implied by a disk $\Delta$ with the properties above.

The proof of Theorem \ref{paper-thm} has an immediate scholium. The Bing involution has a large family of relatives generated by an operation, \emph{ramification}, see Figure \ref{fig:finite_approx}(b) and surrounding text. Our analysis applies equally to ramified examples and rapidly growing ramification drives a parameter of our proof, $\text{(area)}^{-1}$, rapidly towards infinity. So ramified Bing involutions can enjoy an arbitrarily growing imoc. More details and notation are given in Section 3, here is a rough statement:

\begin{customlm}{Scholium \ref{scholium}}
	Given any monotone increasing function $f: \R^+ \ra \R^+$ there is an involution $J: S^3 \ra S^3$ of the 3-sphere such that for any topological conjugate $J^h = h^{-1} \circ J \circ h$, $h: S^3 \ra S^3$ a homeomorphism, its associated $\delta^{-1}(\epsilon^{-1})$ satisfies $\delta^{-1}(x_i) \geq f(x_i)$ for an unbounded sequence of $\{x_i\}$. Furthermore, for any growth function $f$ that is at least exponential with a sufficiently large base (that is, $f(x) > c^x$ for $x$ sufficiently large and some fixed $c > 1$),  the examples and estimates are quite tight; for such a growth function $f$, we build $J$ with moc agreeing with $f$ up to a linear factor, and then show that the moc can diminish over the conjugacy class $\{J^h\}$ only slightly, by dividing the argument by a poly log, e.g.\ $f(x) \ra f(\frac{x}{c \log^2 x})$.
\end{customlm}

\subsection*{A note on quasi-conformal category}
This paper investigates distortion of length through lower bounds on the inherent modulus of continuity imoc, over the conjugacy class of a homeomorphism. But one may also ask potentially more delicate questions about the distortion of angle, again over the conjugacy class. Instead of asking how a map may change the distance between nearby points $x$ and $y$, one may ask how that map may change the radian measure of the angle at $y$ of three nearby points $x,y,x^\pr$. In fact, the problem list \cite{heinonen97} we cited poses the question about the Bing involution in these terms.  For simplicity we have written this paper in the language of length distortion: with $(a,d)$-stretching appearing as its basic measure. But the arguments we present on distortion of length can easily be extended to show similar lower bounds on distortion of angle. In particular, no conjugate $I^h$ of the Bing involution can be a quasi-conformal map.

Let us explain. The concept of $(a,d)$-stretching is used as a tool to establish: the bag lemma (\ref{lm:bag}), its Corollary (\ref{cor:baglemma}), and the J-lemma (\ref{lm:Jlemma}). The way it works is that we identify a planar disk  $D$ of area $<a$ with $\de D \in \mathrm{FIX}$ and some point $y \in D$ transported by the map $k$ to $k(y)$, $\operatorname{dist}(y,k(y))>d$. (In the main theorem $k=I^h$, then in the scholium $k=J^h$, and in the conclusion section $k=g$.) We note that some point $x \in \de D$ must lie within $(\frac{a}{\pi})^{\frac{1}{2}}$ of $y$ implying that the length $(\frac{a}{\pi})^{\frac{1}{2}}$ has been stretched by $k$ to length $d$. But instead we could say that within any 30 degree sector around $y$ there is a point $x \in \de D$ that lies within $\sqrt{12}(\frac{a}{\pi})^{\frac{1}{2}}$ of $y$.  Now consider such an $x$, and another, similar, $x^\pr$ lying in a 30-degree sector rotated 180 degrees from the first. Now both $x$ and $x^\pr$ are fixed and $x,y,x^\pr$ will have a nearly straight angle at $y$, between 150 and 210 degrees. However, $k(x),k(y),k(x^\pr)$ has an angle at $k(y)$ of less than $\frac{2\sqrt{12}}{d}(\frac{a}{\pi})^{\frac{1}{2}}$ radians, so indeed the estimated ratio of angles agrees, up to multiplicative constants, with our estimated ratio of lengths. We will not remark on this further until the appendix, but all lower bounds that we will mention on metric distortion also correspond to lower bounds on angle distortion. After seeing an early draft of this paper, Dennis Sullivan and Matthew Romney pointed out to us that $\frac{1}{K}$-H\"{o}lder continuity is a property of all $K$-quasi-conformal maps (Corollary 6 of \cite{gehring62}), so the above deduction is not strictly necessary. However, we retain it since it implies more detailed information on distortion of angle than the formal fact alone.

It is time to describe the Bing involution $I$.

The Bing involution, $I$, is defined as the quotient of the standard reflection $\tld{I}$ in $S^2 \subset S^3$ by a certain reflection-invariant decomposition $\mathcal{D}$. The nontrivial elements of $\mathcal{D}$ are the components of $\bigcap_{i=0}^\infty \mathcal{\tld{T}}_i^\pr$, where $\mathcal{\tld{T}}_i^\pr$ consists of a disjoint union of $2^i$ solid tori, $\coprod \tld{T}_\sigma^\pr$ where $\sigma$ is a binary string of length $i$ and for each $\tld{T}_\sigma^\pr$ of $\mathcal{\tld{T}}_i^\pr$, the tori $\tld{T}_{\sigma 0}^\pr$ and $\tld{T}_{\sigma 1}^\pr$ are solid tori of $\mathcal{\tld{T}}_{i+1}^\pr$ that are interlocked in $\tld{T}_\sigma^\pr$ as illustrated in Figure \ref{fig:finite_approx}(a). The two sub-tori are called \emph{Bing doubles} or \emph{Bing pairs}. The tori that occur in $\mathcal{\tld{T}}_i^\pr$ or their images under a homeomorphism will be referred to as Bing solid tori (Bst). 

The central insight in Bing's seminal paper \cite{bing52} is that $\mathcal{D}$ is shrinkable\footnote{In a nutshell, shrinkability of the Bing decomposition means that for every $\epsilon > 0$ and $i \in \Z^+$ there is a $j \in \Z^+$, $j > i$, such that there is a homeomorphism $h$ that is the identity on $S^3 - \mathcal{\tld{T}}_i^\pr$ such that for all $k$, $1 \leq k \leq 2^j$, $\operatorname{diam}(h(\tld{T}_\sigma^\pr)) < \epsilon$. The general definition is in Section 1.1.} (see \cites{daverman86,edwards80} for definitions). The non-degenerate elements of $\mathcal{D}$ shrink to a wild Cantor set, which we refer to as Bing's Cantor Set (BCS). Given an explicit sequence of shrinking homeomorphisms, one obtains an explicit involution $I$ from $\tld{I}$. Since no one has ever taken the trouble to give fully explicit shrinking homeomorphisms (there are many small choices to be made), there really is no preferred Bing involution but rather $I$ is any element of a conjugacy class $\{I^h\}$ where $I$ is the result of a particular shrink and $h: S^3 \ra S^3$ is an arbitrary homeomorphism.

\begin{figure}[p]
	\centering
	\input{Inserts/Fig1.1}
	\caption{}\label{fig:finite_approx}
\end{figure}

In Figure \ref{fig:bingpair}(b) we see a finite approximation to what the FIX set of the Bing involution would look like given the first couple displacements of daughter Bing pairs within their mother in Bing's 1988 shrink. Similar distortions of FIX occur in all known shrinks. What we will prove in this paper seems apparent in the figure. We see these long thin tendrils such as $D_{00}$ whose boundary bounds a small disk on a flat meridional disk but whose reflected image of that small disk must go around the tendril and have quite a large diameter. Looking at known shrinks (for example \cite{bing88}---see Figure \ref{fig:bingpair} for early stages) leads to the conclusion that the diameter of the small disk decreases exponentially with the stage $k$, and the diameter of the large disk only harmonically. Immediately this gives an exponential modulus of continuity (moc). So what is the difficulty? The entire issue is that we are studying the conjugacy class of $I$, so whatever picture we draw of FIX and whatever analysis we make of that picture must survive an arbitrary topological change of coordinates. In particular, small disk intersections of shrunk Bing tori with flat meridional disks may have relatively large diameter even with arbitrarily small area.  So the whole question becomes, can this Figure \ref{fig:bingpair} be redrawn by composing with a homeomorphism $h$ of $S^3$ to spoil the conclusion that there is a disk of small diameter reflected into a disk of large diameter, the ratio being exponential in $k$, or nearly so. Figures \ref{fig:bingpair} and \ref{fig:rotated_rings} illustrate two examples in the conjugacy class we are studying. In both of these examples, the intersections of stage tori with planes have small diameters and small areas. But in other conjugates, those intersections may be grossly distorted to have large diameters even when their areas decrease, and the intersections may have strange non-convex shapes. 

In considering this problem, one must take into account the most violent possible homeomorphic transformations of this picture (Figure \ref{fig:bingpair}). It turns out that if one adds an extra assumption that the homeomorphism $h$ approximately preserves the path metric geometry of the boundaries of the Bing solid tori (Bst) (i.e.\ that there is a fixed constant $K$ such that $h$ restricted to the boundary of each Bing solid torus is $K$-biLipschitz to its image), then it is much easier to prove our theorem. In fact, the authors knew this special case in 1982. This raises the question: Had we already considered the case with the least length distortion, that is, could the modulus of continuity actually be lowered by distorting the tori? What possible good could arise from unnecessarily distorting the defining sequence of Bst and thus the geometry of the Bing Cantor set (BCS)? How could complicating the Bst make the involution \emph{more} regular? Since in the end we prove this does not happen, the best we can offer is an analogy to a geometric situation where similar distortion \emph{does} increase regularity.

In 1974, Schweitzer \cite{schweitzer74} constructed a $C^1$ counterexample to the Seifert conjecture. He found a $C^1$ vector field on $S^3$ with no compact orbits. The core of his construction was to ``bust'' isolated closed orbits by forcing them to encounter a transverse punctured torus supporting a $C^1$ Denjoy flow \cite{denjoy32}. Harrison realized \cite{harrison88} that there is a topological change of coordinates which wrinkles the punctured torus in $S^3$ so as to increase its Hausdorff dimension from two to three, such that the overall regularity of the flow containing the Denjoy example increases to $C^{3-\epsilon}$ (on careful analysis the Schweitzer example is $C^{2-\epsilon}$, so the improvement matches exactly the increase in Hausdorff dimension). In our case, the Bing Cantor set (BCS) has Hausdorff dimension $=1$ (for standard shrinks). So there is room to choose more exotic shrinks to fluff it out to three dimensions. A priori we need to show that there can be no Harrison-like improvement in regularity. To complete the Seifert conjecture story: we note that Kuperberg \cite{kuperberg94} later constructed a $C^\infty$ counterexample by different methods. Shortly thereafter, Kuperberg and Kuperberg \cite{kk96} enhanced the example to the analytic category. Readers fond of dynamics may enjoy our appendix on infinite order maps.

\begin{figure}[ht]
    \centering
    \input{Inserts/Fig1.2.tex}
    \caption{Everything on the left is part of $D_0$, a wild meridional disk that is part of the fixed point set, FIX, of $I$. In particular, $D_0$ will contain every point in the wild Cantor set that is in $T_0$.}\label{fig:bingpair}
\end{figure}

Based on the theorem and scholium, we conjecture that wild actions must have correspondingly ``wild'' moc. Notice that if a finite group action is topologically conjugate to a smooth action, then the metric can be averaged so that the action is by isometries---in which case there is no stretching at all. Thus the conjecture asserts in dimension 3 a gap between no stretching and substantial stretching---in contrast to the case in higher dimensions discussed below.

\begin{conj}[Gap Conjecture]
	Suppose a finite group $F$ acts on a 3-manifold by homeomorphisms. If the moc of each element of $F$ is subexponential, then the action is topologically conjugate to a smooth ($C^\infty$) action. Note that this conjecture implies a strengthening of our main result; it would give a fully exponential lower bound for the imoc of the Bing involution, effectively removing the square root in the exponent.
\end{conj}

In higher dimensions the situation is quite different. We next give examples of wild involutions on $S^n$, $n\geq 4$, with fixed sets of dimension $n-1$ that are Lipschitz, thus demonstrating the necessity of the three-dimensional techniques of our proof. Let us start with $n=4$.

A Mazur manifold $M$ is a compact contractible 4-manifold with boundary which can be described as: 0-handle $\cup$ 1-handle $\cup$ 2-handle, where the 2-handle is attached by a curve homotopic but not isotopic to the product factor circle $S^1 \times \ast \subset S^1 \times S^2 = \de($0-handle $\cup$ 1-handle$)$. It follows from Gabai's resolution of Property R \cite{gabai87}---i.e.\ only (0-framed) surgery on the trivial knot in $S^3$ can produce $S^1 \times S^2$---that $\de M \not\cong S^3$. (Proof: Consider the surgery on $\de M$ that reverses the 2-handle attachment to obtain $S^1 \times S^2$. If $\de M \cong S^3$, that surgery must have been on an unknot and the original 2-handle attachment made to the product circle $S^1 \times \ast$.) By the Poincar\'{e} conjecture \cites{perelman1,perelman2,perelman3}, $\pi_1(\de M) \not\cong \{e\}$, since we have just seen that $\de M \not\cong S^3$.

It is well known that the untwisted double of $M$, $M \cup_{\id} \lbar{M} \cong S^4$. Again the proof is by handlebody theory: $M \cup_{\id} \lbar{M} = \de(M \times [0,1])$, but crossing with $[0,1]$ allows the attaching region of the 2-handle to unknot proving $M \times [0,1] \cong B^5$, and, therefore, $M \cup_{\id} \lbar{M}$ is isomorphic to $S^4$. 

We can view $S^4$ as an infinite connected sum of 4-spheres, each half the diameter of the previous one (see Figure \ref{fig:geom_series}), making $S^4$ the limit of a geometric series where addition is connected sum. Do the connected sums symmetrically near the fixed sets and complete to create the metric space $X$. $X$ is biLipschitz equivalent to $S^4_\text{std}$ and $X$ supports an involution that restricts to the $M \leftrightarrow \lbar{M}$ involution on each piece.

\begin{figure}[ht]
	\centering
	\input{Inserts/Fig1.3.tex}
	\caption{Picture of the infinite connected sum.}\label{fig:geom_series}
\end{figure}
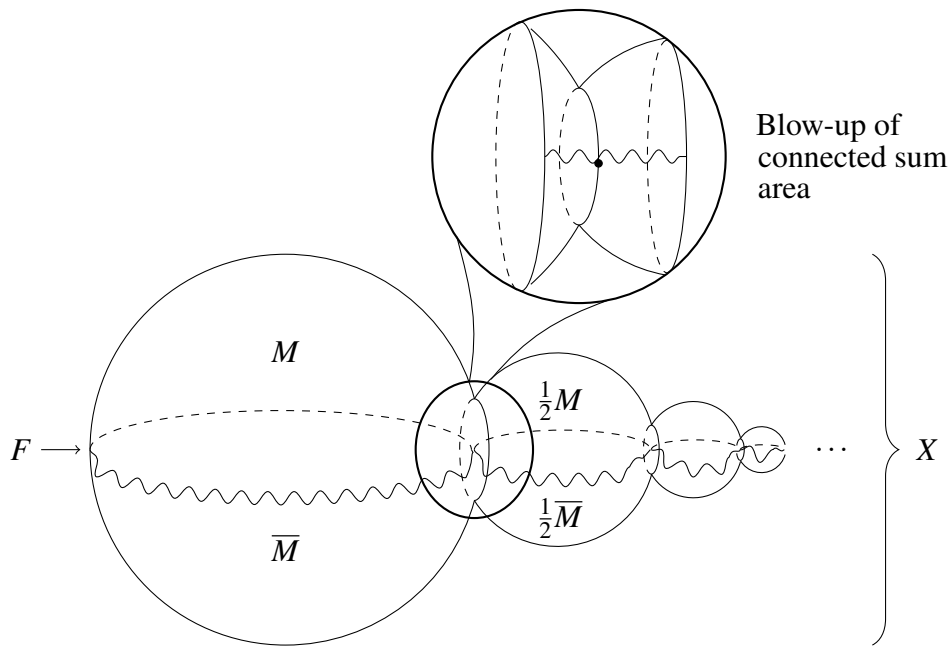

Two things may be observed: (1) the resulting involution of $X$ has the same Lipschitz constant as the original smooth involution $M \leftrightarrow \lbar{M}$ of $S^4$, and conjugating by the equivalence $X \leftrightarrow S^4$ still results in a Lipschitz map. (2) The fixed set $F$ has infinitely generated fundamental group $= \ast_{i=1}^\infty (\pi_1(\de M))$, so $F$ is not even an ANR. Hence the involution is Lipschitz but cannot be made smooth in any coordinates.

To build similar examples in all higher dimensions, $n \geq 5$, all we need is a supply of compact contractible $n$-manifolds $M^n$ with $\pi_1(\de M^n) \neq \{e\}$, with standard untwisted doubles: $M^n \cup_{\id} \lbar{M}^n \cong S^n$. These may be obtained inductively from any 4D Mazur manifold, $M^4$, as follows: let $M^n = M^{n-1} \times S^1 \cup$ 2-handle, where the 2-handle is attached (with any framing) to $\ast \times S^1$, $\ast \in \de M^{n-1}$, $n \geq 5$. On the level of homotopy type $M^n \simeq (\text{pt.} \times S^1) \cup$ 2-cell $\simeq \ast$. Finally, $\de M^n \cong \mathcal{S}(\de M^{n-1} \times S^1)$, where $\mathcal{S}$ represents a 1-surgery. Hence $\pi_1(\de M^n) \cong \pi_1(\de M^{n-1}) \times Z \slash_{t = \id}$, $t$ a generator of $Z$, and so $\pi_1(\de M^n) \cong \pi_1(\de M^{n-1}) \neq \{e\}$.

The high dimensional examples just presented contrast with our result in dimension 3. Curt McMullen pointed out to us that a classical result in two dimensions also \emph{morally} contrasts with our result. It is attributed to Beurling (between lines (2) and (3) of \cite{ahlfors63}): For every quasi-circle $C$ in $S^2$, there is a Lipschitz reflection of $S^2$ with $C$ as fixed set. Dimension two is too low to exhibit topological wildness, but in the Beurling-Alfors result we see that the involution can be more regular than its fixed set. For completeness, a quasi-circle can be defined as the image of the equator in $S^2$ under a quasi-conformal homeomorphism of $S^2$. For example, limit sets of quasi-Fuschian groups are quasi-circles. The point is that the fixed set only enjoys bounded distortion of angle, whereas the involution enjoys bounded distortion of length, a stronger condition. So, in dimension 2, as well as 4, 5, 6, $\dots$, we see examples where involutions are more regular, in various senses, than their fixed sets. Our theorem proves, and we make further conjectures about this theme in section 3, that dimension three is the exception.\footnote{After posting this paper, we became aware of \cite{grillet22}. In this paper, Grillet integrates a Lipschitz differential equation to  construct normal coordinates near the fixed set, to show that a finite order map of a 3-manifold which is nearly an isometry, i.e.\ distorts lengths by no more than factor of 1.00025, must be conjugate to a smooth action.}

\subsection{Decompositions}
We will build the fundamental examples: Alexander horned sphere (AHS), Alexander horned ball (AHB), and the Bing involution ($I$) via decompositions, so let us review the most basic definitions in the simplest context, compact metrics spaces.

Let $X$ be a compact metric space. A decomposition $\mathcal{D}$ is a collection of disjoint closed subsets $\{D_\alpha | D_\alpha \subset X \}$, called \emph{decomposition elements}, such that $X = \cup_\sigma D_\sigma$. Many of the $D_\alpha$ may be single points; the number of non-trivial elements, that is, $D_\alpha \neq \mathrm{pt}$, may be finite, countable, or uncountable. We are interested in the quotient map $\pi$ where each $D_\alpha$ is declared to be a point, thus producing a space $Y$ with the \emph{quotient topology}, i.e.\ the strongest topology making $\pi$ continuous:
\[
	\pi: X \ra Y
\]

It turns out that the following are equivalent \cite{daverman86}:
\begin{enumerate}
	\item $Y$ is metrizable.
	\item $Y$ is Hausdorff.
	\item The map $\pi$ is closed.
	\item $\mathcal{D}$ is \emph{upper-semi-continuous}, meaning every $D_\alpha \in \mathcal{D}$ has a saturated neighborhood system: that is, there are arbitrarily small open neighborhoods of each $D_\alpha$ consisting of unions of decomposition elements.
\end{enumerate}

We only consider decompositions enjoying these properties. The subject  primarily concerns the case when all decomposition elements $D_\alpha$ are \emph{cell-like}, meaning that every open neighborhood $U$ of $D_\alpha$ contains a smaller open neighborhood $V$ of $D_\alpha$ with the property that $V$ is null-homotopic within $U$. A map $\pi: X \ra Y$ is \emph{cell-like} iff for every $y \in Y$, $\pi^{-1}(y)$ is cell-like. As noted in \cite{edwards80}, this condition is equivalent to $\pi$ being a proper homotopy equivalence when restricted to the preimage of any open set of the target. The subject studies the gap between this property and the stronger condition that $\pi$ be \emph{near} a homeomorphism. The fundamental question is: ``When is a cell-like map $\pi$ approximated by homeomorphisms?'' (Since $X$ and $Y$ are metrizable and compact, the only topology to consider on functions is the norm topology.) There is an elegant answer: the Bing Shrinking Criteria, first stated in the form below by \cite{mcauley71} and with Edward's \cite{edwards80} proof.

\begin{customlm}{Bing Shrinking Criteria (BSC)}
	A surjective map $\pi: X \ra Y$ between compact metric spaces is approximated by homeomorphisms iff $\forall \epsilon > 0$, $\exists$ a homeomorphism $g: X \ra X$ such that:
	\begin{enumerate}
		\item $\operatorname{dist}(\pi g, \pi) < \epsilon$, and
		\item $\forall y \in Y$, $\operatorname{diam}\ g(\pi^{-1}(y)) < \epsilon$.
	\end{enumerate}
\end{customlm}

\begin{proof}
	The ``if'' direction is the one of interest. Let $Y^X$ be the space of maps (norm topology) from $X$ to $Y$. Let $E \subset Y^X$ be the closure of the set $\{\pi g \mid g : X \ra X \text{ a homeomorphism}\}$. And for $\epsilon > 0$, let $E_\epsilon$ be the open subset of $E$ consisting of those maps all of whose point pre-images have diameter $< \epsilon$. We next show that each $E_\epsilon$ is (open) dense in $E$: For any homeomorphism $k: X \ra X$, given $\epsilon>0$ there is an $\epsilon_0 > 0$, $\epsilon_0 < \epsilon$, such that sets of diameter $< \epsilon_0$ are carried by $k^{-1}$ to sets of diameter $< \epsilon$. Now choose $g$ as in criteria (1) and (2) for $\epsilon_0$. Then $\pi(gk)$ obeys (2) for $\epsilon$, and $\operatorname{dist}(\pi(gk),\pi k) = \operatorname{dist}((\pi g)k, \pi k) < \epsilon_0 < \epsilon$, so $\pi(gk)$ indeed approximates $\pi k$. Taking limits, $E_\epsilon$ is indeed dense in $E$. The map space $Y^X$, and thus $E$, are compact metric spaces, hence Baire spaces. So $E_0:= \cap_{\epsilon>0} E_\epsilon$ is also dense in $E$, but $E_0$ consists of homeomorphisms from $X$ to $Y$.
\end{proof}

The word \emph{shrinking} refers to the action of the $g: X \ra X$ in the BSC on the non-trivial preimages of $\pi$. Concretely, given $\{g_i\}$, $\epsilon_i$-shrinking homeomorphisms, $\epsilon_i \ra 0$, one can extract a subsequence of $\{\pi g^{-1} : X \ra Y \}$ with limiting homeomorphism $H: X \ra Y$. Given another such family $\{g_i^\pr\}$ and the corresponding $\epsilon_i^\pr \ra 0$ in the BSC with (sub-sequential) limit $H^\pr$, $H^{-1} \circ H^\pr$ is again a subsequential limit of $\{g_i \circ g_i^{\pr -1}\}$. It follows that the identification of $X$ with $Y$, when the BSC holds, is not canonical but may be precomposed with an arbitrary self-homeomorphism of $X$. Looking ahead, the Bing involution $I$ will be constructed on $S^3$ via the BSC, so this discussion explains our remark that $I$ is more naturally regarded as an entire conjugacy class $\{I^h | h: S^3 \ra S^3 \text{ a homeomorphism}\}$. Because $\mathcal{D}$ is $\tld{I}$-equivariant, $\tld{I}$ descends to an involution on $S^3 \slash \mathcal{D}$. But interpreting $I$ as an involution on $S^3$ requires choosing a shrink. A particular shrink will give a particular $I$ (see Figure \ref{fig:part_shrink}).

\begin{figure}[ht]
    \centering
    \begin{tikzpicture}[scale=1.1]
		\node at (0,0) {$S^3$};
		\draw[->] (0.35,0) -- (1.65,0);
		\node at (2,0) {$S^3$};
		\node at (1,0.3) {$\tld{I}$};
		\draw[->] (0,-0.45) -- (0,-1.15);
		\node at (-0.25,-0.75) {\footnotesize{$\pi$}};
		\node at (0,-1.5) {$S^3 \slash \mathcal{D}$};
		\draw[->] (0.7,-1.5) -- (1.3,-1.5);
		\node at (2,-1.5) {$S^3 \slash \mathcal{D}$};
		\draw[->] (2,-0.45) -- (2,-1.15);
		\node at (3.2,-0.6) {\footnotesize{$\pi$-approximate}};
		\node at (3.38,-0.9) {\footnotesize{homeomorphism}};
		\node[rotate=-90] at (0,-2.25) {$\cong$};
		\node at (0,-3) {$S^3$};
		\node[rotate=-90] at (2,-2.25) {$\cong$};
		\node at (2,-3) {$S^3$};
		\draw[->] (0.4,-3) -- (1.6,-3);
		\node at (1,-2.75) {$I$};
		\node at (-2.8,0) {\hspace{0.25em}};
	\end{tikzpicture}
    \caption{Choosing the homeomorphic approximation of $\pi$ is equivalent to selecting the conjugation $h$.}\label{fig:part_shrink}
\end{figure}

We now proceed to that icon of wild topology the Alexander Horned Sphere (AHS) [Ale1924]. Below in Figure \ref{fig:AHS} is Edward's hand sketch from \cite{edwards80}, reproduced with his permission.

\begin{figure}[ht]
	\centering
	\includegraphics{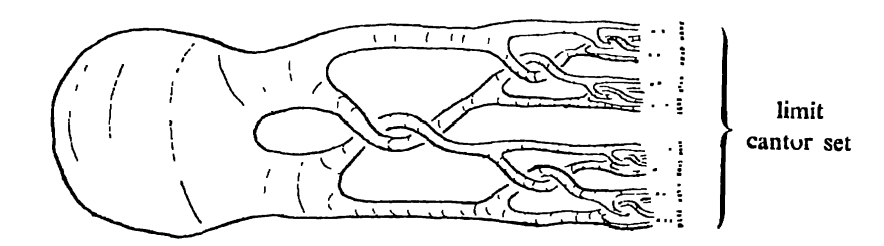}
	\caption{}\label{fig:AHS}
\end{figure}

The AHS can be understood as the result of quotienting a specific decomposition $\mathcal{D}_A$ of $S^3$ whose nontrivial elements comprise a Cantor set's worth of arcs each meeting the standard $S^2 \subset S^3$ in a single point. The decomposition is easily shrinkable by homeomorphisms $\{g_i\}$ which compress these arcs towards their endpoints off $S^2$. The reader should visualize the isotopy of the Cantor set in the plane which is effected by the braiding of the horns. The wildness of the AHS stems from the fact that this isotopy does not extend to an ambient isotopy of the plane. Thus the BSC tells us that $S^3 \slash \mathcal{D}_A \cong S^3$, so the 3-sphere has not been changed, but the 2-sphere becomes the \emph{wild} AHS. Its wildness is reflected in the fact that its exterior is no longer simply connected but instead has an infinitely generated fundamental group, which we return to momentarily.

Very often, decompositions are defined as the components of some $\bigcap_{i=0}^\infty C_i$, where the $C_i$ are smooth codimension zero submanifolds. That will be the case here. Figure \ref{fig:decompositions} below builds $\mathcal{D}_A$ as $\bigcap_{i=0}^\infty C_i$, where $C_i$ consists of $2^i$ disjoint, U-shaped solid cylinders. The components of intersection can be arranged to be a Cantor set of arcs.

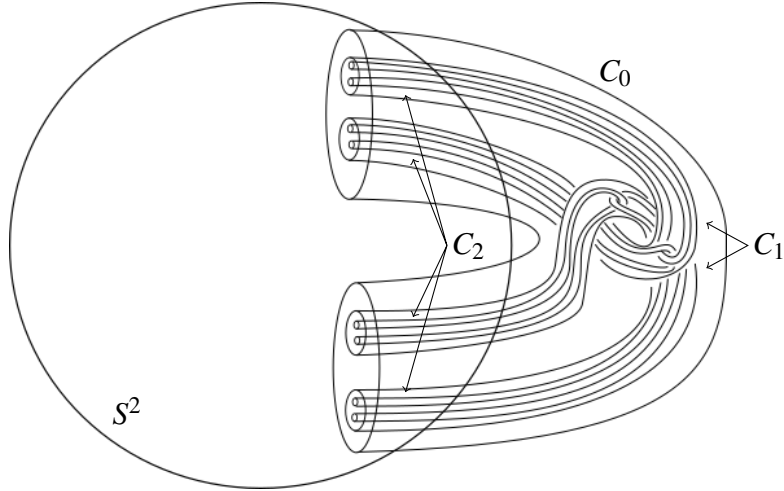
\begin{figure}[ht]
	\centering
	\begin{tikzpicture}
        \node at (0,0) {\includegraphics[scale=0.65]{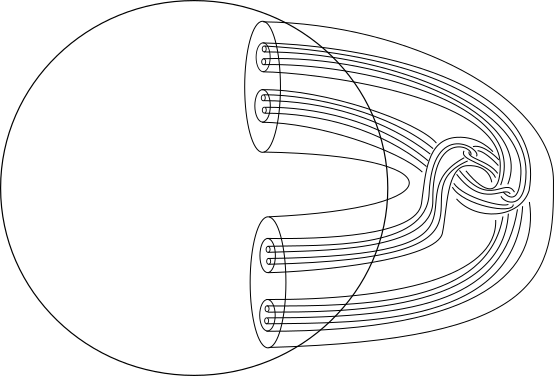}};
        \node at (-3.2,-2.2) {$S^2$};
        \node at (3.3,2.3) {$C_0$};
        \node at (5.35,0) {$C_1$};
        \draw[<->] (4.5,0.3) -- (5.05,0) -- (4.5,-0.3);
        \node at (1.35,0) {$C_2$};
        \draw[<->] (0.5,2) -- (1.05,0) -- (0.6,1.15);
        \draw[<->] (0.5,-1.95) -- (1.05,0) -- (0.6,-0.95);
    \end{tikzpicture}
	\caption{$C_2$ has 4 components, continuing the pattern as shown.}\label{fig:decompositions}
\end{figure}

In this subject, the term \emph{pseudo-isotopy} is used to mean a smooth (if the ambient manifold is smooth) ambient isotopy over $t \in [0,1)$ whose limit is not necessarily one-to-one. Technically, only discrete moments of the pseudo-isotopy figure in as the $g_i$ in the BSC, but visualizing the entire pseudo-isotopy is useful. Imagine pulling $S^2$ in Figure \ref{fig:decompositions} to the right, as if sucked by a vacuum, nearly to the $x$-direction saddle points on $\de C_i$ at time $t = 1 - \frac{1}{i}$. Visually, it should be apparent that the $t = 1$ limit restricted to $S^2$ is Figure \ref{fig:AHS}. Only a Cantor set CS of points on $S^2$ is pulled (sucked) to the right infinitely often. The tracks of CS give the non-trivial elements of $\mathcal{D}_A$, namely, a Cantor set of arcs. At the $t = 1$ limit, each arc has been crushed to its right endpoint, giving a surjection $\pi: S^3 \ra S^3$, with $\pi(S^2_{\text{std}}) = \mathrm{AHS}$. Notice that each point of $S^2$ not on the Cantor set stops moving at some time before time $t = 1$, and therefore $\pi$ is smooth on $S^2$ except on the Cantor set. One way to see that the AHS in $S^3$ is not topologically conjugate to a smooth submanifold is to realize that $\pi_1(\text{open exterior region})$ is an infinite free product-with-amalgamation of a tree of 2-generator free groups.

Consider the two \emph{closed} complementary regions of AHS. It follows from the ANR property of $S^2$ that these must be simply connected: the closed \emph{interior} region is clearly homeomorphic to the ball $B^3$, whereas the closed \emph{exterior} region, the Alexander Horned Ball AHB, is clearly not, having a non-simply connected interior. The AHB can also be described as an infinite 3D \emph{grope}, as defined, for example, in \cite{fq90}.

Given a space and subspace $(X,A)$ the double $D_A X$ is defined as: $(X \times 0 \indep X \times 1) \slash (a \times 0 \equiv a \times 1)$ for $a \in A$. Bing's 1952 paper considers the double of the AHB over its 2-sphere frontier, $D_{AHS}($AHB$)$. In decomposition language: AHB $= B^3 \slash \mathcal{D}_A$, where $B^3$ is the closed exterior component of $S^2$ in $S^3$. The obvious functoriality of decomposition maps under doubling gives the following diagram of homeomorphisms where $D_{S^2} \mathcal{D}_A$ is the reflection double of $\mathcal{D}_A$; so $D \mathcal{D}_A$ is the Bing decomposition $\mathcal{D}$ of $S^3$. Its non-trivial elements are the components of the infinite intersection $\bigcap_{i=0}^\infty D_{(C_i \cap S_2)} C_i$, as drawn (to only the second stage of doubling) in Figure \ref{fig:double}. Notice that as drawn each component is an arc meeting $S^2$ in one point:
\begin{center}
    \begin{tikzpicture}
        \node at (0,0.8) {$S^3 \slash \mathcal{D} \cong (D_{S^2}(B^3)) \slash D_{S^2} \mathcal{D}_A$};
        \node at (0,-0.8) {$D_{S^2}(B^3 \slash \mathcal{D}_A) \cong D_{\mathrm{AHS}}(\mathrm{AHB})$};
        \draw[->] (-1.6,0.4) -- (-1.6,-0.4);
        \draw[->] (1.6,0.4) -- (1.6,-0.4);
        \node at (1.8,0) {\footnotesize{$\cong$}};
        \node at (-1.8,0) {\footnotesize{$\cong$}};
    \end{tikzpicture}
\end{center}

\begin{figure}[ht]
	\centering
	\input{Inserts/Fig1.7.tex}
	\caption{}\label{fig:double}
\end{figure}

The astonishing conclusion of \cite{bing52} is that $\mathcal{D}$ shrinks (i.e.\ obeys BSC). So $D_{AHS}(\mathrm{AHB})$ is actually homeomorphic to $S^3$, and the involution $0 \leftrightarrow 1$ defined on doubles becomes an exotic involution on $S^3$. It is exotic, again, because the fixed 2-sphere is not topologically conjugate to a smooth submanifold (as evidenced by the fundamental group calculation).

Let us review the two shrinks of $\mathcal{D}$ that Bing published \cite{bing52} and \cite{bing88}. In both shrinks, Bing reduces the shrinks to an essentially 1D model where the only important measure of distance is displacement along the $x$-axis. Denoting the ``Bing tori'' dyadically, he lays them out along the $x$-axis and measures diameter discretely by choosing a large even integer $n$ and erecting $n$ parallel planes in intervals of $\frac{1}{n}$. So, the original tori are positioned as in Figure \ref{fig:plane_positions}.

\begin{figure}[ht]
	\centering
	\input{Inserts/Fig1.8.tex}
	\caption{}\label{fig:plane_positions}
\end{figure}

We will denote the image of a Bing torus after a shrinking homeomorphism by dropping the $\tld{\hphantom{m}}$. The tori are "rotated" so that each of the $k$-stage daughters meet only $n-k$ planes, Figure \ref{fig:rotated_rings}.

\begin{figure}[ht]
	\centering
	\input{Inserts/Fig1.9.tex}
	\caption{}\label{fig:rotated_rings}
\end{figure}

After $n-1$ generations, no $T_\sigma$, $\sigma$ a binary word of length $\abs{\sigma} = n-1$, meets more than one of the planes, so its $x$-axis extent is $\leq \frac{2}{n}$. Normal to the $x$-axis, we are free to have chosen a strong compression, so this procedure produces a homeomorphism $h$ that shrinks $n$th stage tori to diameter $O(\frac{1}{n})$.\footnote{Big $O$ notation.} Shrinking the tori of course shrinks the decomposition elements therein so we have met the two Bing Shrinking Criteria. (The first because these rotations of subtori, while large in the source, are small in the quotient space.)

Although the criteria have been met, visually there is a puzzle. It seems, having picked $n$, our strategy grinds to a halt at diameter $O(\frac{1}{n})$. One of us asked Bing about this feature and he said, ``It's like climbing the rope in a gym. After you pull up it does you no good to just keep pulling, you've got to let go with one hand and get a new grip.'' What he meant is that to proceed you need to choose a new $n_1 >> n$ and make $n_1$ tick marks (actually, planar meridional disks) along the partially shrunk $T_\sigma$'s. Now start over, reducing from $n_1$, to $n_1 - 1$, to $n_1 - 2$, $\dots$, the number of tick marks that successive daughters of these $T_\sigma$ cross. For a very long time \emph{nothing} happens with respect to reducing actual diameters (because of the folded nature of the $n_1$ tick marks), but as you get to $n_1 -1$ additional generations, those descendents $T_{\sigma \tau}$ are now crossing only a single tick mark and, therefore, have much smaller diameter than before. Then re-grip and continue.

Something to observe is the slowness of the shrink. At least $2^n$ descendent tori exist before all their diameters are $\leq \frac{2}{n}$. This exponential relationship (also a feature of Bing's second shrink) is the genesis of the modulus of continuity (moc) estimates in this paper.

For completeness, we summarize Bing's 1988 shrink, which he produced in answer to questions we asked him at that time. In his 1988 shrink, every other rotation of tori is \emph{greedy}, as it tries to cut diameters in half. The alternate rotations are \emph{patient}. It turns out greed does not speed the shrinking; it is only at steps indexed by $2^n - 1$, that diameters \emph{actually} are halved. So, the asymptotic rate of shrinking in Bing's 1988 shrink is the same as in the 1952 shrink. In pictures here is the idea of \cite{bing88} (Fig.\ \ref{fig:halving}).

Finally, we cite \cite{fs22}, a new, unexpected method of shrinking the Bing decomposition that the current authors constructed during the course of writing this paper. The shrink in \cite{fs22} demonstrated to us that a plausible, alternative, somewhat simpler strategy for completing the endgame of the proof ahead was doomed to failure.

\begin{figure}[ht]
	\centering
	\begin{tikzpicture}
	    \node at (0,0) {\includegraphics[scale=0.9]{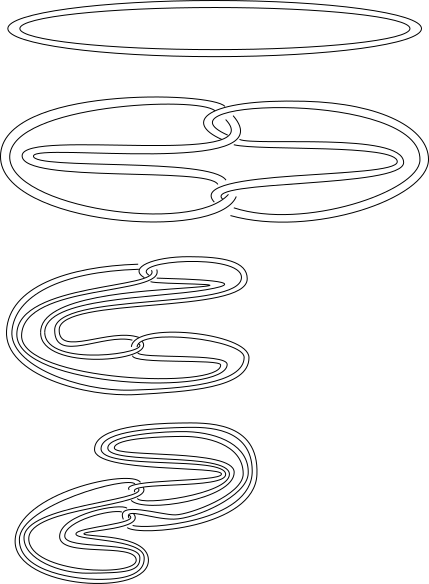}};
        \node at (1.7,5.1) {$\downarrow$ Step 1. Diam halves};
        \node at (0.5,1.2) {$\swarrow$ Step 2. Diam unchanged};
        \node at (3.2,1.2) {$\searrow$};
        \node at (3.7,0.9) {$\ddots$};
        \node at (0.1,-2.8) {$\searrow$ Step 3. Diam halves};
        \node at (-3.8,-2.8) {$\swarrow$};
        \node[rotate=90] at (-4.35,-3.3) {$\ddots$};
        \node at (0.5,-7.9) {$\downarrow$ Step $2^n-1$. Diam halves};
        \node at (-1.67,-7.3) {$\vdots$};
	\end{tikzpicture}
	\caption{}\label{fig:halving}
\end{figure}

\section{Proof of Theorem}
\subsection{Preliminary notation}
$S^3 = \{\vec{x} = (x_1,\dots,x_4) \mid \lVert \vec{x} \rVert = 1\}$, $I_0(x_1,\dots,x_4) = (x_1, x_2, x_3, -x_4)$ is a standard reflection which becomes the Bing involution $I$ after shrinking the Bing decomposition $\mathcal{D}$, whose non-trivial elements are $\cap_{i=1}^\infty \tld {\mathcal {T}^i}$, where $\tld {\mathcal {T}^i}$ is the union of the $2^i$ tori that are the columns in Figure \ref{fig:tree}. This decomposition, defined in the introduction, is $x_4 \leftrightarrow -x_4$ symmetric, but its shrink cannot be. Before shrinking, we have the unshrunk Bing tree Bt laid out in Figure  \ref{fig:tree}. It is a slight misnomer to speak of \emph{the} Bing involution since no one, certainly not Bing, endured the tedium of precisely specifying a particular pseudo-isotopy that shrinks $\mathcal{D}$. It is better to consider all pseudo-isotopies and hence the topological conjugacy class $\{I^h \coloneqq h^{-1} \circ I \circ h | \ h: S^3 \ra S^3 \text{ a homeomorphism}\}$. Going forward we simplify notation by using $I$ to denote the general conjugate. In the diagram below, the choice of the conjugator $h$ determines the involution $I$. 

\begin{figure}[ht]
	\centering
	\input{Inserts/Fig2.1.tex}
	\caption{Let $\sigma$ be a (possibly empty) binary string, then each $\tld{T}_\sigma$ is a solid torus with $\tld{T}_{\sigma0} \cup \tld{T}_{\sigma1} \hookrightarrow \tld{T}_\sigma$ the inclusion of a \emph{Bing pair}, all pre-shrunk and in standard position. Set $\tld{\mathcal{T}}^i = \cup_{\abs{\sigma} = i} \tld{T}_\sigma$, be the columns of the Bing tree. $\abs{\sigma}$ denotes the length of the binary string $\sigma$.}\label{fig:tree}
\end{figure}
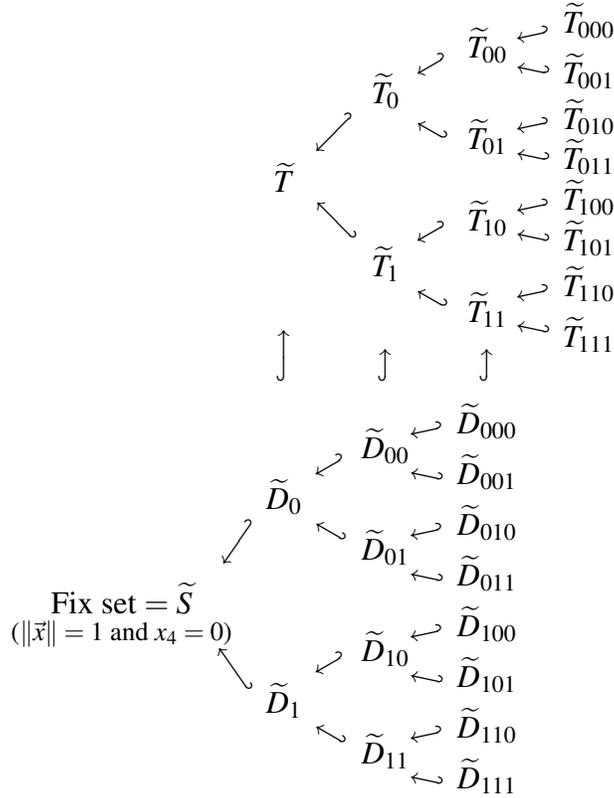

By definition, a choice of a shrink maps the Bing decomposition $\mathcal{D}$ to a Bing Cantor Set (BCS). Notice that unlike the Cantor set of wild points of the AHS, which is tame, the BCS is wild, that is, not definable as an intersection of balls. Recall $\tld{T}_\sigma$, $\sigma$ a finite binary string, denotes an unshrunk Bing solid torus. $\tld{T}_{\sigma = \varnothing}$ will denote the outermost, \emph{root}, solid torus. $\abs{\sigma}$ denotes the length of the string. An infinite $\sigma$ (indexed by the natural numbers) denotes a non-trivial decomposition element. Dropping the $\tld{\hphantom{m}}$ means the solid tori are (partially) shrunk. 

Almost our entire argument is in the context of a \emph{partial shrink}; that is, a homeomorphism $g: S^3 \ra S^3$ such that for some $i>0$, $g$ carries each stage $i$ solid torus to one with diameter $< \epsilon$, $\epsilon > 0$, that is, diam$(g(\tld{T}_\sigma)) < \epsilon$ for $\abs{\sigma} = i$. We set $T_\sigma \coloneqq g(\tld{T}_\sigma)$ and $D_\sigma \coloneqq g(\tld{D}_\sigma)$, so Figure \ref{fig:tree} with $\tld{\hphantom{m}}$ dropped becomes what we call the shrunk Bing tree (sBt). We call any $T_\sigma$ of the sBt a \emph{Bing solid torus} (Bst). WLOG we assume that all partial shrinking homeomorphisms restrict to $\id_{S^3 \setminus \tld{T}_\varnothing}$. So $T (=T_\varnothing) = \tld{T}_\varnothing$ is the Bing solid torus with $\sigma$ the empty string, the root. Occasionally, in the proof we refer to the Bing Cantor Set (BCS), which is the ends of the shrunk Bing tree (sBt) i.e.\ shrunk by a limit of partial shrinking homeomorphisms $g_i$. But even this usage of infinity is a mere convenience. We could, instead, reference a partial shrinking homeomorphism $g^\pr$ with all components of some $\mathcal{T}^i$ having diameter $<\epsilon$ and for some $l >> i$ all components of $\mathcal{T}^l$ having diameter $<\epsilon^\pr$, for $\epsilon^\pr << \epsilon$. The tiny components of $\mathcal{T}^l$ can play the same role as the BCS in our proof. Since WLOG [Moi 54] any partial shrinking homeomorphism in 3D can be approximated by a diffeomorphism or PL\ homeomorphism, it is fair to say that the entire proof is finitary and combinatorial. So, although we are discussing the (lack of) regularity of a wild map, we may apply the standard general position (which we assume throughout) and transversality arguments, familiar from the smooth category. Our proof (in detail) does not explicitly draw on the infinite complexity of wild topology. While we work on the smooth side of all limits, our guiding principle has been to follow Bing's ideas.\footnote{So much so that we often thought of Bing as a co-author.}

Let us take a moment to detail the passage to the smooth category. It is immediate from Bing's shrinking argument (either one) that some standard representative $I$ of the Bing involution is approximable by smooth involutions $\{I_i\}$, $i = 1,2,3, \dots$ (in the obvious topology, pointwise convergence). The general topological conjugate $I^h$ shares this property, since by [Moi54] there are PL or smooth approximations $\{h_i\} \ra h$, so $K_i \coloneqq I_i^{h_i}$ converges to $I^h$.

\begin{lemma}\label{lm:moc-limit}
    Given the above notation and a sufficiently small $\epsilon > 0$, let $\delta_i$ be the largest number such that $\operatorname{dist}(x,y) \leq \delta_i$ implies $\operatorname{dist}(K_i(x), K_i(y)) \leq \epsilon$. Without changing notation, pass to a subsequence such that $\delta_i$ converges to some $\delta$. Then $\delta$ is the largest number such that $\operatorname{dist}(x,y) \leq \delta$ implies $\operatorname{dist}(I^h(x), I^h(y)) \leq \epsilon$. (And, as a consequence, passing to the subsequence was actually unnecessary.)
\end{lemma}

The lemma tells us that the moc's of the approximating smooth involutions $K_i$ cannot suddenly drop in the limit, but instead converge to moc$(I^h)$.

\begin{proof}
    Passing (repeatedly) to further subsequences, and again not changing the index notation, we may locate pairs $(x_i,y_i)$ converging to $(x,y)$ with $\operatorname{dist}(x_i,y_i) = \delta_i$ and $(K_i(x_i), K_i(y_i))$ converging to $(I^h(x), I^h(y))$ where $\operatorname{dist}(K_i(x_i), K_i(y_i)) = \epsilon$. Unless $x$ and $y$ are antipodes, we need only consider the case of $\operatorname{dist}(x,y)$ arbitrarily small. Invariance of domain under $I^h$ implies that when $\epsilon$ is made a function of $\delta$ (by choosing, globally, the smallest possible $\delta$), that $\epsilon_{I^h}(\delta)$ is \emph{strictly} monotone, and of course continuous. Thus $\delta_{I^h}(\epsilon)$ and $\delta^{-1}_{I^h}(\epsilon^{-1})$ are strictly monotone continuous functions.

    Now suppose there is a $\delta^\pr > \delta$ such that $\operatorname{dist}(x,y) \leq \delta^\pr \implies \operatorname{dist}(I^h(x),I^h(y)) \leq \epsilon$. Strict monotonicity of $\delta_{I^h}(\epsilon)$ implies that for some $\epsilon_0 < \epsilon$, $\operatorname{dist}(x,y) \leq \frac{\delta + \delta^\pr}{2} \implies \operatorname{dist}(I^h(x), I^h(y)) \leq \epsilon_0$. But then for sufficiently large $i$, $\operatorname{dist}(x_i,y_i) \leq \frac{\delta + \delta^\pr}{2} \implies \operatorname{dist}(K_i(x_i),K_i(y_i)) \leq \epsilon$. But this contradicts the convergence of $\delta_i$ to $\delta$, showing no such $\delta^\pr$ exists.
\end{proof}

Although we will continue to write $I$ or $I^h$ for the involution under consideration, the reader should implicitly replace it by some smooth $K_i$, for $i$ large as justified by the preceeding lemma. Similarly, any mention of the Bing Cantor Set BCS can, if preferred, be treated as a reference to its defining solid tori of sufficiently small diameters.

Instead of taking precise Euclidean (or spherical) diameters, we exploit a cruder notion, called \emph{Bing's parallel plane method}, based on how many times an object traverses \emph{fixed length} chambers between some fixed collection of parallel planes $\mathbb{P}$, or in our case, disks $\mathbb{P} \coloneqq \{P_k | k = 1, 2, \dots,$ $n\}$. To set this up, fix the coordinates on $T$, the outermost Bing solid torus, as $D^2 \times S^1$, $D^2 = \{\vec{x} \in \R^2, \lVert x \rVert \leq 1\}$ and $S^1 = \de D^2$. Let $P_k = D^2 \times e^{\frac{2\pi i k}{n}}$. (We often call the $P_k$ ``planes'' although they are disks and may refer to them as ``parallel'' although they are not precisely parallel. We do this to be consistent with Bing's notation. Also, the proof involves \emph{many} disks and calling this special class of them ``planes'' will reduce opportunities for confusion.) So, $\mathbb{P}$ depends on a positive integer $n$ (which we think of as large; in the proof, initially $\epsilon \approx \frac{1}{n}$). An object $X$ meeting more than one $P_i$ will have diameter, diam$(X) \geq O(\epsilon)$.

Our entire proof involves analyzing how the planes $P_k$ in some family $\mathbb{P}$ intersect the various solid tori $T_\sigma$, which are shrunk Bing tori, that is, the images of the standard, symmetrically positioned families of $2^i$ Bing tori under a homeomorphism of $T$ that distorts them violently to make them small. Each component of the general position intersection of a torus $T_\sigma$ with a plane $P_k$ will be a planar 2-manifold-with-boundary. In most of our Figures, these disks-with-holes are drawn as rather neat, roundish things; however, the challenges of the proof require us to deal with the possibility that they might resemble thin neighborhoods of wiggly trees with rather large diameters.  

Among those components of $T_\sigma \cap P_k$, we will be interested in only a special set of those that contain exactly one meridional curve of $T_\sigma$ in the intersection. So in this paper, we will reserve the term 'disk-with-holes' to refer to the following restricted collection of components of sBt with planes. 

\begin{defin}
	In this paper, \emph{disks-with-holes} (dwh) are connected components of the intersection of a solid torus $T$ and a ``plane'' $P \in \mathbb{P}$ with a unique meridional boundary component. In the context of sBt, dwh also have an inductive requirement: They must be contained within a dwh of their mother, inductively starting back to ancestor $T_\varnothing = S^1 \times D^2$ in which the dwh is a $D^2$ fiber (see Figure \ref{fig:inductive_dwh}). There are two kinds of dwh: \emph{first kind}, where the outermost boundary component is the unique meridian, and \emph{second kind}, where the outermost boundary component is trivial (see Figure \ref{fig:delta} to find an example where this strange possibility can occur). All dwh encountered in the proof of the theorem are first kind unless stated otherwise.
\end{defin}

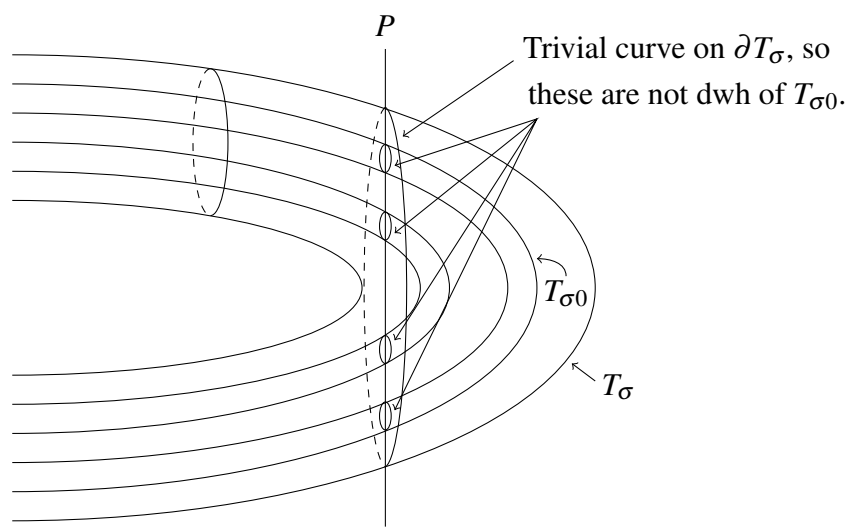
\begin{figure}[ht]
    \centering
    \input{Inserts/Fig2.2.tex}
    \caption{Inductive feature of dwh.}\label{fig:inductive_dwh}
\end{figure}

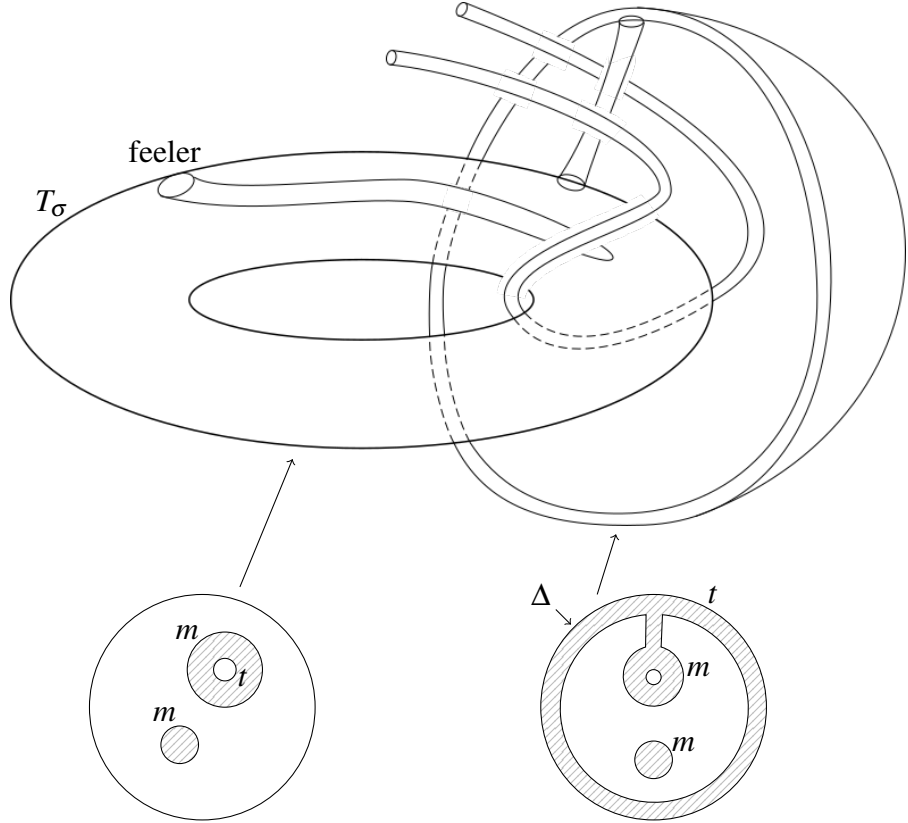
\begin{figure}[ht]
	\centering
	\input{Inserts/Fig2.3.tex}
	\caption{$T_\sigma$ with two planar cross-sections shown. On the left are dwh of the first kind. But, on the right, $\Delta$ is of the second kind. rdwh are always first kind.}\label{fig:delta}
\end{figure}

Each 'hole' of a dwh is a curve that is trivial on the boundary of its torus. In Bing's shrinks, such trivial holes never arise, but in an analysis of modulus of continuity, they must be considered. When dealing with exclusively topological issues, we can simply remove those trivial curves from consideration by the topological method of pulling back feelers, which involves making homeomorphic adjustments. Unfortunately, performing such homeomorphisms would cause us to lose control of modulus of continuity estimates. Nevertheless, for the future purpose of computing area estimates near the end of our proof, it is necessary to identify those dwh that would remain \emph{if} we did remove trivial curves by a homeomorphism, even though we do not actually perform such homeomorphisms. The following discussion and definition identify those special dwh.

\emph{Residual disks-with-holes (rdwh)} is a key concept. It will be defined for a finite family of parallel ``planes'' $\lbar {\mathbb {P}}$ and some initial subtree, $\lbar{T}$, of the sBt. $\lbar{\mathbb {P}}$ might be the $\{D^2 \times e^{2\pi ik/n} | 1 \leq k \leq n\} \subset T_\varnothing$ and $\lbar{T}$ might be the $\leq n$- stage sBt, but we will apply the notion more generally. So $\lbar{T}$ is a collection of solid tori $T_\sigma$, each $T_\sigma$ a sBt. The definition is based on the picture of each intersection of a plane $P_k \in \lbar{\mathbb {P}}$ with each torus $T_\sigma \in \lbar{T}$ \emph{after} a canonical isotopy $\mathcal{I}$ of $\lbar{T}$ repositions $\lbar{T}$ so that each intersection of a plane $P_k \in \lbar{\mathbb {P}}$ with each torus $T_\sigma \in \lbar{T}$ contains no trivial (in $\de {T_\sigma}$) scc $\gamma$. 

Notice that $\lbar{T}$ is a collection of tori, some of which are subsets of others. When we refer to $\de \lbar{T}$ in what follows, we mean the set of all boundaries of tori in $\lbar{T}$, and a scc being trivial in $\de \lbar{T}$ means it is trivial on the boundary of the torus on which it lies. 

$\mathcal{I}$ may be described as follows: Consider all trivial on $\de \lbar{T}$ scc $\{\gamma\}_t \subset \{\gamma\} = \lbar{\mathbb {P}} \cap \de \lbar{T}$. Within $\{\gamma\}_t$ select a $\lbar{\mathbb {P}}$-innermost $\gamma$: $\gamma$ bounds disks $D_\gamma$ and $E_\gamma$ on $\lbar{\mathbb {P}}$ and $\de \lbar{T}$, respectively. Ambiently isotope $E_\gamma$ across the ball $B_\gamma \subset T_\varnothing$ bounded by $D_\gamma \cup E_\gamma$ to remove $\gamma$ from $\{\gamma\}_t$, and likely removing many other scc as well. Continue inductively until $\{\gamma\}_t = \varnothing$. The composition is $\mathcal{I}$. The possible balls $B_\gamma$ are all either pairwise disjoint or nested, so the order of selection of innermost $\gamma \in \{\gamma\}_t$ is immaterial. In fact, the entire isotopy (by \cite{hatcher83}) is canonical in that it is determined from a contractible space of choices. We do not give details here, since our logic does not require that $\mathcal{I}$ is canonically defined, only the fairly obvious property that if $\lbar{T} \subset \lbar{\lbar{T}}$ is extended to a larger tree, e.g.\ by adding the daughters of its leaves, then the isotopy $\mathcal{I}(\lbar{\mathbb {P}},\lbar{\lbar{T}})$ may be taken to agree with $\mathcal{I}(\lbar{\mathbb {P}},\lbar{T})$ outside a neighborhood of the new sBt, $\lbar{\lbar{T}} \setminus \lbar{T}$. Similarly, if additional parallel planes are added to $\lbar{\mathbb{P}}$ to make $\lbar{\lbar{\mathbb {P}}}$, the isotopy $\mathcal{I}(\lbar{\lbar{\mathbb {P}}},\lbar{\lbar{T}})$ will yield the same intersection pattern with $\lbar{\mathbb{P}}$ as the isotopy $\mathcal{I}(\lbar{\mathbb {P}}, \lbar{\lbar{T}})$. $\mathcal{I}$ removes all trivial curves from $\{\gamma\}$ (leaving only meridians and possibly longitudes; see Figure \ref{fig:meridional_disk}). The curves that remain have not moved. Among these, consider only the meridians $\{\gamma\}_m$. Each such meridian, which is innermost w.r.t.\ meridians of its torus $\de T_\sigma$, coincides with an outer boundary of some dwh $\Delta$ of $\lbar{\mathbb {P}} \cap \lbar{T}$, prior to applying $\mathcal{I}$.

\begin{defin}\label{def:rdwh}
	Such $\Delta$ are called \emph{residual disks-with-holes} rdwh. By the remark below, they are disks with holes of the first kind: the outer boundary is a meridian. $\Delta^+$ is defined to be the disk in $\lbar{\mathbb {P}}$ bounding $\de_m\Delta$, the meridional component of the boundary, i.e.\ $\Delta$ with trivial holes filled. Finally, define $\operatorname{area}^+(\Delta) \coloneqq \operatorname{area}(\Delta^+)$. Note that $\Delta \subset T_\sigma$ does not imply $\Delta^+ \subset T_\sigma$.
\end{defin}

\begin{note}
	The isotopy $\mathcal{I}$ is useful because it allows us to focus combinatorially on only residual dwh and avoid an area overcounting issue at the end of the proof. It is tempting to assign $\mathcal{I}$ a greater role and analyze the modulus of continuity not of $I$, but the simplified $I^{\mathcal{I}_1} \coloneqq \mathcal{I}_1^{-1} I \mathcal{I}_1$, but we do not know how to carry modulus estimates across this conjugation.
\end{note}

\begin{remark}
	rdwh are first kind. To see this, observe (Figure \ref{fig:delta}) that for a second kind dwh $\Delta$, the \emph{inward} normal from the meridional boundary component $\de_m \Delta$ must point \emph{outward} toward infinity in the plane $P$ containing $\Delta$. On the other hand, after the isotopy $\mathcal{I}$, a meridional scc $m$ of intersection $m \subset \de T \cap P$, if innermost among meridians of $\de T$, must, by linking number, have the \emph{inward} direction into $T$ align with the \emph{inward} direction in $P$. Thus, the rdwh $\Delta \subset T$ bounded by $m$ has agreement between inward into $T$ and inward w.r.t.\ $P$, showing that it is first kind.
\end{remark}

\begin{defin}
	For a large even integer $n$ let $\mathbb{P} = \{D^2 \times e^{2\pi i k/n}, 1 \leq k \leq n\}$. With this plane family fixed, we say $T_\varnothing$ has \emph{length} $n$. The \emph{length} $\ell(T_0)$ and $\ell(T_1)$ of its daughters is defined, say for $T_0$, by looking at all dwh of $\mathbb {P} \cap T_0$ and giving each dwh an index $k$ indicating the plane $P_k \in \mathbb{P}$ in which it lies; then $\ell(T_0) \coloneqq$ the number of transitions between distinct indices. Similarly for $T_\sigma$, $\ell(T_\sigma)$ is defined by taking the dwh of $T_\sigma \cap \mathbb{P}$, indexing each again according to $P_k \in \mathbb{P}$, and counting index transitions.
\end{defin}

\begin{defin}
	\emph{Residual length} $r\ell(T_\sigma)$ is defined by now taking \emph{only} the rdwh in $T_\sigma$, again indexing each by its plane $P_i \in \mathbb{P}$. $r\ell(T_\sigma) =$ number of rdwh index transitions. From the discussion of $\mathcal{I}$ above Definition \ref{def:rdwh}, $r\ell(T_\sigma)$ is unchanged if we add descendents of $T_\sigma$ to the sBt $\lbar{T}$ under consideration.
\end{defin}

In this proof we speak of certain $T_\sigma$ \emph{meeting} certain planes $P$, initially $P \in \mathbb{P}$, but we later add additional planes. When we say ``does not meet'' or ``no essential meeting'' with $P$, we mean \emph{does not meet in any rdwh}; these are all that count. Usually, in figures, only residual dwh are indicated. Intuitatively, $r\ell$ edits out some of the length a given torus spends \emph{apparently} (not topologically) inside another non-ancestor torus as viewed in its projection to $S^1$. $r\ell$ allows us to focus on the important structure without obfuscating detail.

\begin{note}
	Both $\ell$ and $r\ell$ are even integers and $r\ell(T_\sigma) \leq \ell(T_\sigma$).
\end{note}

It is an observation of Bing's that:
\[
	2\ell(T_\sigma) - 4 \leq \ell(T_{\sigma 0}) + \ell(T_{\sigma 1}).
\]
Our version of this observation involves residual dwhs.

\begin{lemma}\label{lm:el_count}
	$2r\ell(T_\sigma) - 4 \leq r\ell(T_{\sigma 0}) + r\ell(T_{\sigma 1})$.
\end{lemma}

\begin{proof}
	Perform the canonical isotopy $\mathcal{I}$ so that $T_\sigma$ meets $\mathbb{P}$ only in essential curves, then extend the isotopy to $\lbar{\mathcal{I}}$, as discussed, so the daughters $T_{\sigma 0}$ and $T_{\sigma 1}$ also have only essential intersection. Now the Bing observation above gives the lower bound to the number of planar meridional disks of the daughters. But, the boundaries of these disks are precisely the meridional boundaries of the daughters rdwh.

	For completeness Bing's argument is: $T_\sigma \setminus (T_{\sigma 0} \cup T_{\sigma 1})$ is the Borromean link complement. Since this link is nonsplit, each meridional disk $\Delta$ of $T_\sigma$, and in fact each dwh (just cap the holes on $\de T_\sigma$) must meet the core circle of either $T_{\sigma 0}$ or $T_{\sigma 1}$. By linking number, any such meeting must come in $\pm$ pairs. This implies that $\Delta \cap T_{\sigma 0}$ or $\Delta \cap T_{\sigma 1}$ must contain two components essential in $H_2(T_{\sigma i}; \de T_{\sigma i}; \Z)$ ($i=0$ or $1$), i.e.\ two dwh. So, passing to daughters at least doubles the total number of dwh, but since we are counting index transition, one can be lost at each turning point (the min/max of a lift of the daughters to the infinite cyclic cover $\tld{T}_\sigma$) of each daughter, accounting for the -4.
\end{proof}

\emph{Orientation}. Fixing the orientation on $D^2$ (the factor disk of $T_\varnothing$), and choosing arbitrary orientations on each $T_\sigma$, each of the dwh of $T_\sigma$ itself becomes oriented, and, as a set, the dwh are cyclically ordered in $T_\sigma$. 

\begin{defin}
	A Bing $T_\sigma$ is called a final $k$, $k_f$, if $r\ell(T_\sigma) = k$ and all descendants of $T_\sigma$ have smaller $r\ell$ values. In particular, $r\ell(T_{\sigma 0}) = k-2 = r\ell(T_{\sigma 1})$. Note that $k$ is necessarily even.
\end{defin}

\begin{lemma}\label{lm:final_k}
	For $\mathbb{P}_{n+1}$, having $n+1$ planes, and $0 \leq 2k \leq 2n$, there are at least $2^{(n-k)}$ $2k_f$'s.
\end{lemma}

\begin{proof}
	By induction on $n-k$. The proof is by comparison to the lengths induced by Bing's 1952 shrink: at each fork, lengths decrease by two. By Lemma \ref{lm:el_count}, at any fork, either the daughter lengths are both two less than the mother's, or at least one daughter has a length greater than or equal to the mother's. In the second case, choose such a large length sibling and simply discard the other sibling's tree branch. The retained branch must have as many $k_f$'s in it as the original $T_\sigma$ has among its descendents.
\end{proof}

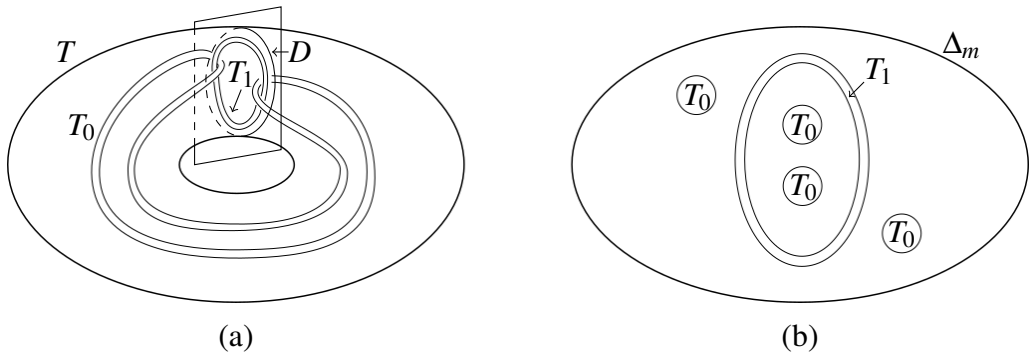
\begin{figure}[ht]
	\centering
	\input{Inserts/Fig2.4.tex}
	\caption{A meridional disk $\Delta_m$ of $T$ having intersection with $T_0$ and $T_1$ where $T_0$ has four (residual) meridional disks and $T_1$ has none.}\label{fig:meridional_disk}
\end{figure}

An important role is played by the exponentially many final $n$'s $\{n$-$f\}$ produced by Lemma \ref{lm:final_k} (in the case $n$ is even). To support intuition for the upcoming argument, we draw two distinct cases of 4-$f$'s in Figure \ref{fig:planes_alt}. Simple combinatorics tell us that a 4-$f$ must meet 2 or 3 parallel planes. Up to irrelevant \emph{wiggles}, the two basic pictures are Figure \ref{fig:planes}(a) and \ref{fig:planes}(b).

\begin{figure}[ht]
	\centering
	\input{Inserts/Fig2.5.tex}
	\caption{}\label{fig:planes}
\end{figure}

``Wiggles,'' as in Figure \ref{fig:planes_alt}, mean the introduction of additional dwh, or even rdwh, without increasing the number of plane-index transitions which determine $\ell$ and $r\ell$.

\begin{figure}[ht]
	\centering
	\input{Inserts/Fig2.6.tex}
	\caption{The picture above is not importantly different from $T_\sigma$ in Figure \ref{fig:planes}(a), both have $r\ell = 4$.}\label{fig:planes_alt}
\end{figure}
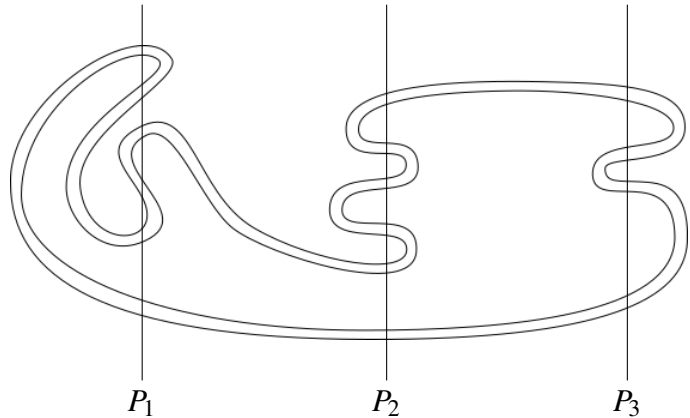

Next we come to the Bag and long arc lemmas. Both can produce ``$(a,d)$-\emph{stretching},'' defined as finding two points $x,y$; $x \in \mathrm{FIX}$, so $I(x) = x$, and $y$ with $\operatorname{dist}(x,y) \leq \sqrt{\frac{a}{\pi}}$ and $\operatorname{dist}(I(x),I(y)) > d$. ``$a$'' stands for planar area and $\sqrt{\frac{a}{\pi}}$ is a corresponding length scale produced by the isoparametric inequality, that is, if $D$ is any planar disk of area less than $a$, then any point in $D$ is within distance $\sqrt{\frac{a}{\pi}}$ of a point on $\de D$. Our lower bound on modulus will be obtained by describing 4 circumstances under which we can deduce $(a,d)$-stretching, and proving that at least one must occur. The first two are consequences of Lemmas \ref{lm:bag} and \ref{lm:longarc}, the next two are in Propositions \ref{prop:final_4} and \ref{prop:final_4}.

\begin{lemma}[Bag Lemma]\label{lm:bag}
	Let $I$ be a smooth orientation-reversing involution of $\R^3$ (it will be smoothly conjugate to $\tld{I}(x,y,z) = (-x,y,z)$, the standard orientation-reversing involution). Let $P$ and $P^\pr$ be parallel planes perpendicular to the $x$-axis in $\R^3$, distance $d$ apart. Assume $\Delta \subset P$ is a disk with area$(\Delta) \leq a$. Assume there is a scc $J \subset \operatorname{FIX}(I) \cap \Delta$ and that the subdisk of $\mathrm{FIX} \coloneqq \operatorname{FIX}(I)$ bounded by $J$, $D_J$, meets $P^\pr$, then there exist $x \in J$ and $y \in \Delta$ such that $\operatorname{dist}(x,y) \leq \sqrt{\frac{a}{\pi}}$ and $\operatorname{dist}(I(x),I(y)) \geq d$.
\end{lemma}

\begin{proof}
	Let $R$ be the obvious (rel $J$)-homology between $D_J$ and $\Delta_J$, where $\Delta_J \subset P$ is the subdisk of $\Delta$ bounded by $J$. $R$ is the closure of those points $z$ such that an arc from $z$ to infinity crosses $D_J \cup \Delta_J$ algebraically once. The \emph{bag} is the frontier of $R \cup I(R)$, which equals $\Delta_J \cup I(\Delta_J)$. Suppose $I(\Delta_J) \cap P^\pr = \varnothing$, then any point $p$ of $D_J \cap P^\pr$ has a ray parallel to the positive $x$-axis heading towards $\infty$ and disjoint from the bag. But this is a homological contradiction: $D_J$ is ``in the bag,'' $D_J \subset R \cup I(R)$. In the easiest case to envision, $R \cup I(R)$ is a 3-ball $B$ and $p \in Int(B)$, so the ray must intersect $\de B$. In general, $R \cup I(R)$ may not be a manifold; however, its non-manifold singularities lie on $P$ and $D_J \cap P^\pr \subset R \cup I(R) \setminus \operatorname{Frontier}(R \cup I(R))$; so, $D_J \cap P^\pr$ lies in a bounded component of the complement of the singular 2-sphere $\operatorname{Frontier}(R \cup I(R))$.

	Consequentially, $I(\Delta_J)$ meets $P^\pr$ at some point we designate as $I(y)$. Thus $y \in \Delta_J$. But $\Delta_J \subset P$ is a planar disk with area $\leq \operatorname{area}(\Delta) \leq a$. By the area hypothesis, every point $y$ of $\Delta_J$ must lie within $\sqrt{\frac{a}{\pi}}$ of a point $x \in \de \Delta_J \subset \mathrm{FIX}$. Thus we have a point $y \in P$ within $\sqrt{\frac{a}{\pi}}$ of a point $x$ in FIX such that $I(y)$ lies on $P^\pr$, hence $\operatorname{dist}(I(x) (=x),I(y)) \geq d$.

	The bag lemma's conclusion is the necessity of $(a,d)$-stretching. All our measures of necessary stretching will follow from the bag lemma.
\end{proof}

In 1982 the authors understood the bag lemma and its immediate corollary below. As we will see next, it implies either $(a,d)$-stretching, or the existence of \emph{long arcs} within the meridional scc $m$ of FIX $\cap \de($sBt$)$---the subject of the following long arc lemma. The phrase \emph{long arc} refers to the fact that these meridions $m$, in order to avoid an immediate conclusion of $(a,d)$-stretching from the bag lemma, cannot be small round circles but instead must run all about $\de($sBt$)$ and cross many parallel planes. In 1982, we considered these long arcs to be a disaster, not a lemma, as their presence thwarts the most obvious application of the bag lemma. The key insight---40 years on---was that these long arcs are \emph{useful} input to what we call the J-lemma, which leads to Propositions \ref{prop:final_4}, which provide an alternative venue for the bag lemma to find stretching. The road has many twists, turns, and forks, but in the end all lead to quantifiable $(a,d)$-stretching.

\begin{cor}[Corollary of Bag Lemma]\label{cor:baglemma}
	To establish notation, let $T$ be an $I$-invariant shrunk Bing torus (sBt) with $\operatorname{FIX} \cap T$ being two disjoint disks $D_0, D_1$ with meridional curve boundaries. Let $P, P^{\pr\pr}$ be parallel planes distance $2d$ apart. Let $\Delta, \Delta^{\pr\pr}$ be dwh of the first kind, where $\de_m \Delta$ is the meridional component of the boundary of $\Delta$. $\Delta \subset (T \cap P)$, $\Delta^{\pr\pr} \subset T \cap P^{\pr\pr}$ with $\operatorname{area}(\Delta^+)$, $\operatorname{area}(\Delta^{\pr\pr+}) < a$. Suppose $A$ is an annulus on $\de T$ between $\de_m \Delta$ and $\de_m \Delta^{\pr \pr}$ such that $A \cap (\de D_0 \cup \de D_1) = \varnothing$. Then there exist points $x,y$ with $x \in \mathrm{FIX}$, $\operatorname{dist}(x,y) \leq \sqrt{\frac{a}{\pi}}$ and $\operatorname{dist}(I(x),I(y)) \geq d$. That is, $I$ has $(a,d)$-stretching. In this corollary, as in Lemma \ref{lm:bag}, $I$ is still a smooth involution.
\end{cor}

\begin{figure}[ht]
	\centering
	\input{Inserts/Fig2.7.tex}
	\caption{(a) Depicts the simplest case of the bag lemma where $\Delta_J \cup I(\Delta_J)$ is a 2-sphere bounding $R \cup I(R)$. (b) Shows the case where $D_J$ has a feeler that pierces $\Delta_J$. Any ray from a point on $D_J \cap P^\pr$ to the right must intersect $I(\Delta_J)$.}\label{fig:bag_lm}
\end{figure}
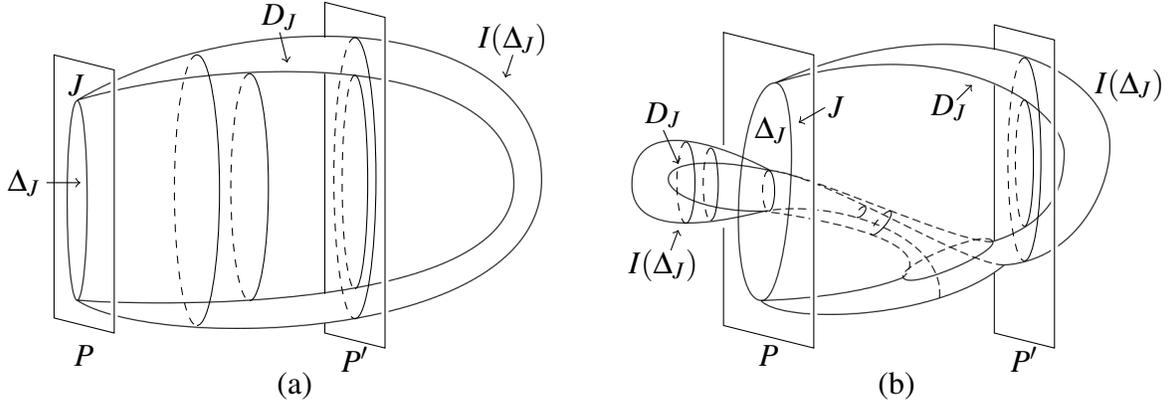

\begin{proof}
	Consider a plane $P^\pr$ parallel to $P$ and $P^{\pr\pr}$ and halfway between. There is a dwh $\Delta^\pr$ in $P^\pr \cap T$ with $\de_m \Delta^\pr \subset A$. $\Delta^\pr$ must contain a point of the Bing Cantor set. That point $p$ lies on $D_0$ or $D_1$, say $D_0$. But $D_0 \cap A = \varnothing$, so $D_0 \cap (\Delta \cup \Delta^{\pr\pr}) \subset \operatorname{Int}(D_0)$. So, there exists a scc $J \subset (D_0 \cap (\Delta \cup \Delta^{\pr\pr}))$ such that the subdisk $D_J$ of $D_0$ bounded by $J$ contains $p$.
\end{proof}

In accordance with Lemma \ref{lm:moc-limit}, we stated the bag lemma in the context of smooth involutions. We showed there that it is sufficient to work with smooth approximates $K_i$ of $I^h$, and Lemma \ref{lm:bag} and its corollary technically should be applied to these with $I = K_i$. It is possible, alternatively, to prove a similar bag lemma for wild involutions. For completeness we sketch how this works. In the wild case, FIX$(I_{\text{wild}}) \cap P$ may fail to have any scc $J$ even if $P$ is varied. So, the scc $J$ must be replaced by a separating continuum $C$ in FIX$(I_{\text{wild}}) \cap P$ such that $C \subset \operatorname{Int}(\Delta)$, where $\Delta$ is a disk in $P$ of area $< a$. Next, find a scc $J^\pr$ in FIX$(I_{\text{wild}})$ ``close'' to $C$ and containing $C$ in the disk $J^\pr$ bounds in $D_0 \subset \operatorname{FIX}(I_{\text{wild}})$. $J^\pr$ bounds a singular disk $D_{J^\pr}$ consisting of a ``collar'' on $J^\pr$ that projects $J^\pr$ perpendicularly into $\Delta$ union a null homotopy in $\Delta$. Area estimates for $D_{J^\pr}$ are not available but one can still check that every point of $D_{J^\pr}$ is within $\sqrt{\frac{a}{\pi}}$ of $J^\pr$, since $a > \operatorname{area}(\Delta)$. The ``bag'' is then $D_{J^\pr} \cup I(D_{J^\pr})$ and the rest of the argument proceeds as in the smooth case.

The following lemma states a useful consequence of the preceding corollary.

\begin{lemma}[Long Arc Lemma]\label{lm:longarc}
	Suppose $I$, $\mathbb{P}$, and a sBt $T$ whose fixed point meridional disks are $D_0$, and $D_1$ are given with the consecutive planes of $\mathbb{P}$ distance $d$ apart and transition rdwh's $\{\Delta_1, \Delta_2, \dots, \Delta_n\}$ in cyclic order around $T$, with $\Delta_i \subset P_i$ and $P_1 \neq P_n$. If $I$ does not have $(a,d)$-stretching, then $\de D_0 \cup \de D_1$ must meet all the $\Delta_i$'s except for a maximum of two in a row when the areas of the $\Delta_i^+$'s not hit by $\de D_0 \cup \de D_1$ are small ($< a$). Since $\de D_0$ and $\de D_1$ are connected sets, under the conditions above, there exist arcs $\alpha_0 \subset \de D_0$ and $\alpha_1 \subset \de D_1$ such that $\alpha_0 \cup \alpha_1$ meet all the $\de \Delta_i$'s except for a maximum of two in a row, which can occur not more than twice. \qed 
\end{lemma}

\begin{defin}
	A \emph{shrunk Bing fork} (sBfork) refers to a mother sBt and its daughters $(T_\sigma; T_{\sigma 0}, T_{\sigma 1})$.
\end{defin}

We will omit the subscript $\sigma$ of the $T$'s, $D$'s, and $E$'s in what follows for readability.

\begin{lemma}[J-Lemma]\label{lm:Jlemma}
	Let $(T; T_0, T_1)$ be a sBfork essentially meeting parallel planes $Q$ and $Q^\pr$ that are at least distance $d$ apart. That is, these three sBt meet $Q,Q^\pr$ in rdwh with $T \cap Q$ containing rdwh $\Delta$, and $T \cap Q^\pr$ containing rdwh $\Delta^\pr$. Let $\alpha$ be an arc of $\de D_0 \subset \mathrm{FIX} \cap \de T_0$ meeting $\Delta^\pr$ exactly at its endpoints, crossing $Q$, and having linking number 1 with $\beta = \operatorname{longitude}(T_1)$. Assume $\beta \cap \Delta = \varnothing$ and $\beta \cap \Delta^\pr = \varnothing$. This linking number is (well-)defined by closing $\alpha$ by any arc in $\Delta^\pr$. Suppose $\operatorname{area}^+(\Delta) < a$. In this circumstance, $I$ has $(a,d)$-stretching.

	\begin{figure}[ht]
		\centering
		\begin{tikzpicture}
			\node at (0,0) {\includegraphics[scale=0.8]{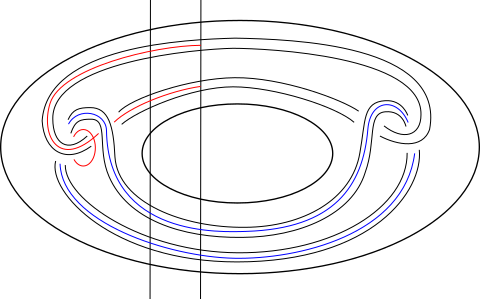}};
			\node at (-0.8,-3.4) {$Q^\pr$};
			\node at (-1.9,-3.4) {$Q$};
			\node at (-1.35,-2.9) {$d$};
			\draw[->] (-1.15,-2.9) -- (-0.85,-2.9);
			\draw[->] (-1.55,-2.9) -- (-1.85,-2.9);
			\draw[line width=0.35ex] (-1.9,0.35) -- (-1.9,2.55);

			\node at (-1.3,3.3) {$\Delta^\pr$};
			\draw[->] (-1.1,3) -- (-0.9,2.35);
			\node at (-2.6,2.9) {$\Delta$};
			\draw[->] (-2.4,2.8) -- (-2,2.3);
			\node[red] at (-4.3,2) {$\alpha$};
			\draw[red,->] (-4.1,1.9) -- (-3.7,1.6);

			\node at (4.5,1.8) {$T$};
			\node at (2.3,3) {$T_0$};
			\draw[->] (2.1,2.8) -- (1.8,2.4);
			\node at (3.2,-2.7) {$T_1$};
			\draw[->] (3,-2.5) -- (2.5,-2);

            \draw[fill=black] (-3.1,0.25) circle (0.2ex);
			\node[red] at (-3.2,-0.6) {$J$};
			\node at (-2.7,-1.2) {$E$};
			\node[blue] at (-4.2,-2) {$\beta$};
			\draw[blue,->] (-4,-2) -- (-3.15,-1.6);
		\end{tikzpicture}
		\caption{}\label{fig:Jlemma}
	\end{figure}
\end{lemma}

\begin{proof}
	Recall that $D_0$ is the component of FIX $\cap T$ that meets $T_0$. Coming from the pre-shrunk Bt there is a disk $E_0 \subset T$ with $\de E_0 = \beta$ and $E_0 \cap D_0 = \varnothing$ (see Figure \ref{fig:finite_approx}c). Now using innermost circles on $\Delta^+$, cut $E_0$ off to obtain a disk $E$ with $\de E = \de E_0 = \beta$ and $E$ disjoint from $\Delta^+$. That is, cut off $E_0$ near the dark line in Figure \ref{fig:Jlemma} to obtain $E$ with the same boundary $\beta$. Of course, $E$ is not necessarily contained in $T$, because of the holes in the rdwh $\Delta$.

	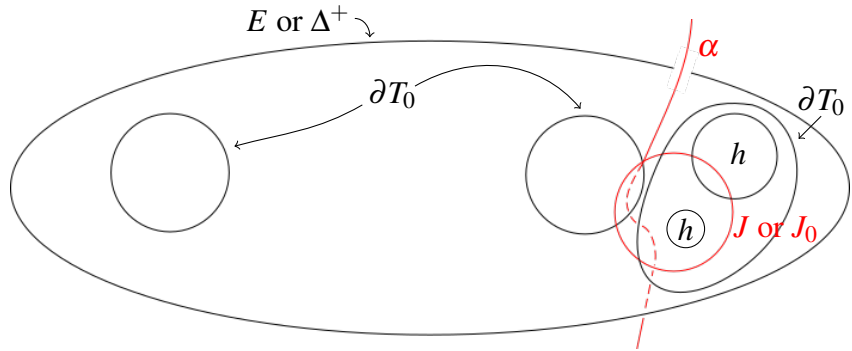
\begin{figure}[ht]
		\centering
		\input{Inserts/Fig2.9.tex}
		\caption{The pictures in $E$ and $\Delta^+$ are similar and thus rendered as one in this figure. The \emph{holes} labeled by $h$ arise from cut and paste along $\Delta^+$. The holes do not lie in $T$; they arise from ``feelers'' penetrating $T$. Intersections with $D_0 \subset$ FIX are rendered in red.}\label{fig:fix_intersect}
	\end{figure}

	The intersections of $E$ with $\de T$ and FIX are similar to the intersections of $\Delta^+$ with $\de T$ and FIX, since all intersections of $E$ with $\de T$ and $D_0$ arise from cutting off $E_0$ on $\Delta^+$, so those intersections are drawn together in Figure \ref{fig:fix_intersect}. After cutting off, notice that $E \cap \de \Delta^+ = \varnothing$. The linking hypothesis forces a scc $J$ of $E \cap \mathrm{FIX}$ to have odd intersection with $\alpha$, since $\alpha$ (closed off by an arc in $\Delta^\pr$) has linking number one with $\beta = \de E$, and, therefore, intersection number one with $E$. By the same token, there must be a scc $J_0$ of $\Delta^+ \cap \mathrm{FIX}$ having odd intersection with $\alpha$. Assuming $J_0$ innermost in $\Delta^+$ w.r.t.\ this property, let $D_{J_0}$ be the subdisk of $\mathrm{FIX}$ bounded by $J_0$, an innermost scc of $\mathrm{FIX} \cap \Delta^+$ meeting $\alpha$ in an odd number of points.

	The geometry of $D_{J_0}$ feeds into the $(a,d)$-stretching bound given by the bag lemma. The planar area in $Q$ spanned by $\de D_{J_0}$ determines $a$, and the $x$-coordinate separation between $Q$ and $Q^\pr$, which $D_{J_0}$ spans, determines $d$.
\end{proof}

\subsection*{General Strategy}
Establishing large $(a,d)$-stretching involves finding a disk $D$ on FIX such that $\de D$ lies on a plane $P$ where $\de D$ bounds a disk on $P$ of relatively small area $a$ and $D$ itself spans across to a plane parallel to $P$ relatively far, namely $d$, away. The smallness of $a$ will come from analyzing the exponentially large number of disjoint rdwh's of sBt's at stage $n$ lying on $P$. Finding an appropriately large $d$ will come from seeking opportunities to apply the J-lemma. Specifically, the arc $\alpha$ in the J-lemma will arise by analyzing how far away the turning points of daughter sBt's are from their mother's turning points. If there is a significant rotation from mother to daughter, then the long arc lemma will provide us with the arc $\alpha$ of the J-lemma. If, further, the granddaughters rotate significantly within a daughter, then Proposition \ref{prop:final_4} below locates $(a,d)$-stretching. But if there is a ``delay'' in rotation, so that $\alpha$ is not ``sufficiently folded,'' we must continue with a second strategy for applying the J-lemma called the \emph{endgame}.

Proposition \ref{prop:final_4} will get us through the trickiest bit of the proof. The context of Proposition \ref{prop:final_4} is a final $n$, $n$-$f$. We call the $n$-$f$ $T_m$ for ``mother.'' In it are daughters $T_{m0}$ and $T_{m1}$, we rename $T_{m0}$ as $T$ and mostly work inside $T$ with $T$'s daughters $T_0$ and $T_1$, granddaughters of $T_m$; we also rename $T_{m1}$ as $T_s$ for the sister of $T_{m0}$. The strategy of the proof is to seek the hypothesis of the J-lemma using the long arc lemma.

To establish notation, see Figure \ref{fig:6_strings} in which $n=4$ and $T_m$ is drawn spanning three of the original planes, $P$, $Q$, and $R$, of $\mathbb{P}_{n+1}$, so $T_m$ is drawn unfolded. To make the proof of Proposition \ref{prop:final_4} easier to compare to Figure \ref{fig:6_strings}, we have transferred the notation from the figure directly into the proof, sacrificing some unimportant generality. After initially reading the proof in this restricted context, $n_i = 4$ and the mother torus $T_m$ unfolded, i.e.\ spanning three planes $P$, $Q$, and $R$ from the original set of planes $\mathbb{P}_{n_i+1}$, the general case requires only slight modifications. In general, $T_m$ may span many original planes or be ``highly folded'' and span back and forth among only two planes. We address the required modification at the end of the proof.

\subsection*{Setup for Proposition \ref{prop:final_4}}
Proposition \ref{prop:final_4} rests on the arc $\alpha$ from the J-lemma and redrawn in Figure \ref{fig:6_strings}. The pattern in which such an $\alpha$ passes through its turning point plane, $P$ in Figure \ref{fig:6_strings}, and its parallels at distances $d$ and $2d$ will be the same in all cases. There are ten rdwh's labelled---two, $\vphantom{\Delta}_m\Delta_f^P$ and $\vphantom{\Delta}_m\Delta_i^Q$, are rdwh of $P \cap T_m$ and $Q \cap T_m$. The remaining eight labelled rdwh's are rdwh's of $T$ (a daughter of $T_m$). The subscripts $i$ and $f$ indicate initial and final rdwh's with respect to the cyclic order on $T$.

Let $(T_m; T, T_s)$ be a final $n$ fork, i.e.\ \begin{tikzpicture}[baseline={($(current bounding box.center) - (0,5pt)$)}]
	\node at (0,0) {$n$};
	\draw (0.5,-0.25) -- (0.2,0) -- (0.5,0.25);
	\node at (1,0.25) {$n-2$};
	\node at (1,-0.25) {$n-2$};
\end{tikzpicture}
w.r.t.\ $r\ell$. Let $(P,P^\pr,P^{\pr\pr},Q^{\pr\pr},$ $Q^\pr,Q)$ be 6 consecutive parallel planes with a fixed spacing $d > 0$ between $P \leftrightarrow P^\pr$, $P^\pr \leftrightarrow P^{\pr\pr}$, $Q^{\pr\pr} \leftrightarrow Q^\pr$, and $Q^\pr \leftrightarrow Q$. Let $a$ be the maximum value of $\operatorname{area}^+$ for 6 initial or final rdwh: $\vphantom{\Delta}_m\Delta_f^P$, $\Delta_i^{P^\pr}$, $\Delta_f^{P^\pr}$, $\Delta_f^{Q^\pr}$, $\Delta_i^{Q^\pr}$, $\vphantom{\Delta}_m\Delta_i^Q$, as drawn in Figure \ref{fig:6_strings}, where we have provided an example illustration of both the daughters and granddaughters inside one of the two possible configurations of $n$-$f$'s for $n=4$. The superscript labels the plane, the subscripts mean initial/final w.r.t.\ a chosen orientation on $T$, and the back-script $m$ indicates an rdwh on $T_m$; the other 8 are rdwh on $T$. $T_m$ being a final 4 implies an absence of certain rdwh, indicated according to our shorthand notation $R \cap T = \varnothing = P \cap T_S$ (which follows from $r\ell(T) = 2 = r\ell(T_S)$). To fix labels, assume $T_0 \cap P \neq \varnothing \neq T_1 \cap Q$. Again, we write $\cap = \varnothing$ to indicate no rdwh.

\begin{figure}[ht]
	\centering
	\input{Inserts/Fig2.10.tex}
	\caption{}\label{fig:6_strings}
\end{figure}

\begin{prop}\label{prop:final_4}
    If either ($T_0 \cap \Delta_i^{P^{\pr\pr}} \neq \varnothing$ and $T_0 \cap \Delta_f^{P^{\pr\pr}} \neq \varnothing$), or ($T_1 \cap \Delta_f^{Q^{\pr\pr}} \neq \varnothing$ and $T_1 \cap \Delta_i^{Q^{\pr\pr}} \neq \varnothing$), then $I$ induces $(a,d)$-stretching as defined above.
\end{prop}

\subsection*{Turning points and displacement}
We referred in the general strategy above to the idea of 'turning points', which are intuitively the places where the two daughters, $T_0$ and $T_1$, of a sBt $T$ clasp one another in their mother $T$. The phrase \emph{tightly clasped} is applicable when the upper \emph{turning} point of one daughter is adjacent to the \emph{lower} turning point of the other, an important special case, but by no means general. Let's make that notion of \emph{turning point} more formal. Suppose $T$ is a sBt with rdwh's $\{\Delta_1, \Delta_2, \dots, \Delta_n\}$ in cyclic order around $T$. Let $T_0$ and $T_1$ be the daughters of $T$. Then the \emph{lowest turning point} of $T_0$ in $T$ is obtained by looking at the universal cover of $T$ and a lift of $T_0$ in it and finding the lowest $\Delta_i$ that contains an rdwh of the lift of $T_0$. Similarly, we can define \emph{highest turning point}; and, of course, the same definition applies to $T_1$. Notice that this definition allows us to describe the relative locations of turning points from generation to generation when the rdwh's are arising from intersections with some set of planes $\mathbb{P}$. We use the term ``displacement'' to refer to the relative position of the daughters' turning points relative to their mother's turning points. These displacements will always be measured by rdwh's rather than continuously. To initialize the induction, the turning points of the outer $T_\varnothing$ are determined by the initial and final planes of $\mathbb{P}_{n+1}$.

Let us express these notions in a diagram. Consider a sBt $T_\sigma$ with its cyclically ordered family $\{\Delta_i\}$ of rdwh. Let $r: T_\sigma \ra S^1$ be a retraction carrying $\{\Delta_i\}$ to points $\{p_i\}\subset S^1$ with any consecutive string of $\Delta_i$'s that lie on the same plane all mapping to the same point on $S^1$, but otherwise retaining the original cyclic order. Now let $\tld{r}: \tld{T}_\sigma \ra \R^1$ be the lift. These maps constitute a discrete approximation to a projection of $T_\sigma$ onto $S^1$, with a real-valued $\tld{r}$ from the universal cover:
\[
\begin{tikzpicture}
	\node at (0,0) {$\tld{T}_\sigma$};
	\draw[->] (0.4,0) -- (1.6,0);
	\draw[->] (0,-0.4) -- (0,-1.1);
	\node at (-0.2,-0.75) {\footnotesize{$\tld{r}$}};
	\node at (2,0) {$T_\sigma$};
	\node at (0,-1.5) {$\R^1$};
	\draw[->] (0.4,-1.5) -- (1.6,-1.5);
	\draw[->] (2,-0.4) -- (2,-1.1);
	\node at (1.8,-0.75) {\footnotesize{$r$}};
	\node at (2,-1.5) {$S^1$};
\end{tikzpicture}
\]

Consider lifts of $T_{\sigma 0}$ and $T_{\sigma 1}$ into $\tld{T}_\sigma$. The \emph{turning points} of $T_{\sigma 0}$ and $T_{\sigma 1}$ in $T_\sigma$ correspond to lowest and highest rdwh in the lifts. Then \emph{displacement} $d$ is the distance between a turning point of a mother and a turning point of one of her daughters. This distance is simply measured discretely along the real axis $\R^1$ according to the $\tld{r}$ coordinate, introduced above, on $\tld{T}$ the universal cover. The periodic nature of the lifted pattern of rdwh in $\tld{T}_\sigma$ shows that turning points and displacement are defined independently of the lift. If $T_\sigma$ wiggles back and forth between two planes of $\mathbb{P}_{n_i+1}$ there may be several lowest/highest turning points. In the important case of the daughters of $n$-$f$'s, the ``final'' condition implies that the daughters are tightly clasping which determines unique initial and final upper and lower turning points. This uniqueness propagates through later generations provided the daughters are \emph{rather precise}, a condition defined after completing the proof of Proposition \ref{prop:final_4}

\begin{proof}
	The turning points of $T_0$ and $T_1$ in Figure \ref{fig:6_strings} as drawn do \emph{not} satisfy the hypotheses of Proposition \ref{prop:final_4}, but if any of those turning points were shifted one plane toward the center (that is, between $P^{\pr \pr}$ and $Q^{\pr \pr}$), then the hypotheses of Proposition\ \ref{prop:final_4} \emph{would} be satisfied. Note that in Figure \ref{fig:6_strings}, $T_0$ and $T_1$ each has $rl=0$; however, that is not a hypothesis of Proposition \ref{prop:final_4}.
	
	The two small arrows (1 and 2) near the center of Figure \ref{fig:6_strings} are included in the figure to point out the need for us to deal with the possibility that the clasps are not necessarily tight, that is, it may not be the case that the lowest turning point of $T_0$ and the highest turning point of $T_1$ (or vice versa) occur in consecutive planes, as they appear in Figure \ref{fig:6_strings}.  The turning point of $T_1$ near arrow 1 can be pushed counterclockwise (and the turning point of $T_0$ near arrow 2 can be pushed clockwise), still without satisfying the hypotheses of Proposition \ \ref{prop:final_4}. But there is a definite limit in both  cases or the hypotheses of Proposition \ \ref{prop:final_4} will be satisfied.

	In the case of the turning point near arrow 1, we will see that $T_1$ may be pushed left through $P$ up and around the horn, back through $P$ and $P^\pr$, without satisfying the hypotheses. However, if $T_1$ is stretched to also go through $P^{\pr \pr}$ in $\Delta_f^{P^{\pr\pr}}$, then the hypotheses will be satisfied where $T_1$ is playing the role of $T_0$. Similarly, in the case of the turning point near arrow 2, $T_0$ may be pushed through $Q$, around the bend, back through $Q$ and $Q^\pr$, but again if that end of $T_0$ is pushed through $Q^{\pr \pr}$ in $\vphantom{\Delta}_m\Delta_f^{Q^{\pr \pr}}$, the hypotheses will be satisfied. 
	
	We seek to prove $(a,d)$-stretching. The overall strategy of the proof is that we first attempt to use the bag lemma to imply $(a,d)$-stretching. If that fails, we conclude that the long arc lemma implies that the hypotheses of the J-lemma accrue, and so the J-lemma will imply $(a,d)$-stretching. So the failure of one attempt implies the existence of the hypotheses for a different strategy. Here are the details.
	
	By symmetry, it suffices to show the claimed stretching assuming $T_0$ has rdwh intersections with $\Delta_i^{P^{\pr\pr}}$ and $\Delta_f^{P^{\pr\pr}}$. In addition we will assume that $T_0$ goes through $P$ as pictured, rather than around the other direction through $Q$. If $T_0$ went the other way, then it would intersect the corresponding $Q^{\pr \pr}$'s and the plane $Q$ and we would do the analysis on the $Q$ side. 
	
	Only initial and final rdwh $\Delta$ are relevant, so let us draw a straight, or \emph{monotone}, picture of $T$ (see Figure \ref{fig:PQ}), which omits some recurrent planar intersections, and focus on the stretching caused by the long ``hook'' in $T_0$. This monotone picture is a simplification already encountered when deleting the ``extra wiggles'' in Figure \ref{fig:planes_alt}. If we see $(a,d$)-stretching in the monotone picture a fortiori it occurs in the literal picture.

	Recall that $D_0$ denotes the meridional disk of $\mathrm{FIX} \cap T$ that intersects $T_0$. For the actual $I^h$, $D_0$ is wild, but continuing our approximation of $I^h$ by $I^h \coloneqq K_i$, $D_0$ will be smooth. Suppose $\de D_0$ fails to intersect either $\Delta_i^{P^\pr}$ or $\Delta_f^{P^\pr}$, say $\de D_0 \cap \Delta_f^{P^\pr} = \varnothing$. Some point of the $(\mathrm{BCS} \cap T_0) \subset D_0$ lies on the rdwh $\vphantom{\Delta}_m\Delta_f^P$ and some other such point lies on the rdwh $\Delta_i^{P^{\pr \pr}}$. So if $\de D_0 \cap \Delta_i^{P^\pr} = \varnothing$, then there exists a scc $K$ on $D_0 \cap \Delta_i^{P^\pr}$ that bounds a subdisk $D_K$ of $D_0$ containing one of those BCS points. Therefore, such a $D_K$ intersects a plane distance $d$ from its boundary. The assumption that the $area^+$ of $\Delta_f^{P^\pr} \leq a$ then completes the hypotheses of the bag lemma. So if $\de D_0 \cap \Delta_f^{P^\pr} = \varnothing$, then we can conclude $(a,d)$-stretching.
	
	So suppose $\de D_0$ intersects both $\Delta_i^{P^\pr}$ and $\Delta_f^{P^\pr}$. In this case, we will have the hypotheses of the J-lemma, namely: 
	
	\begin{enumerate}
	    \item There exists an arc $\alpha \subset \de D_0$ with endpoints on $\Delta_i^{P^\pr}$ and $\Delta_f^{P^\pr}$ respectively as pictured (monotonely) in Figure \ref{fig:PQ}. 
	    \item Since $T_s$ has $rl =2$, we can select a longitude of $T_s$, $\beta$, such that $\beta \cap P = \varnothing$, $\beta \cap \vphantom{\Delta}_m\Delta_i^Q = \varnothing$, and $\beta = \de E_0$, where $E_0$ is a disk in $T_m$ (see Figure \ref{fig:finite_approx}(c)). Notice that  $\alpha$ has $\pm 1$ ``linking'' with $T_s$ (``linking'' is defined when $\alpha$ is closed with a (any) arc within the dwh of $P^\pr \cap T_m$ that contains the endpoints of $\alpha$).
	    \item The rdwh $\vphantom{\Delta}_m\Delta_f^P$ of $T_m$ has $area^+ < a$. 
	    \item The distance between $P$ and $P^\pr$ is greater than $d$.
	\end{enumerate}
	
	Therefore, the J-lemma concludes that $I^h$ has $(a,d)$-stretching.  The final point is to restore the generality lost by referencing the proof to the special case of a 4-$f$, spanning 3 planes ($P$, $Q$, and $R$) as in Figure \ref{fig:6_strings}. The only point worthy of attention for the general $n_i$-$f$ case is how to cut off on planes the ``primordial'' $E$ to an $E_0$ dual to $\alpha$ and hence suitable for an application of the J-lemma. The issue is already visible in Figure \ref{fig:planes}(b), where $T_\sigma$, the daughter of a folded 4-$f$ is not disjoint from $P_1$, but meets it in the \emph{lower left} portion of the figure. However, the $\alpha$-arc, which would lie in the \emph{upper left} portion, is separated from $P_1 \cap E$ by some $\text{rdwh}^+$ of $T_\sigma \cap P_1$, so the previous innermost circles argument still constructs the dual $E_0$ from $E$. This case is now general; the $E_0$ dual to $\alpha$ can always be obtained by cutting off $E$ on $\text{rdwh}^+$ on $T_\sigma$.
\end{proof}

\begin{figure}[ht]
	\centering
	\input{Inserts/Fig2.11.tex}
	\caption{}\label{fig:PQ}
\end{figure}

Having completed the basic lemmas, here is the logic the proof will follow. Fix a conjugate $I^h$ of the Bing involution to study. Consider a sequence $\{n_i\}$ of positive integers approaching infinity, $i = 1,2,3,\dots$, each with its family $\mathbb{P}_{n_i+1}$ of $n_i+1$ \emph{parallel planes}. For each $n_i$ pass to the induced ${n_i}$-$f$ sBt, and then to the $k$th descendants sBt inside them, $k = 0,1,2,3,\dots$. Denote by $\{T_{\sigma^j_{n_i}}\}$ the collection of final $n_i$'s with respect to $\mathbb{P}_{n_i+1}$. $\sigma_{n_i}$ is a bit string of a final $n_i$ and the index $j$ picks out one of these. For each $j$, adding dots $\{T_{\sigma^j_{n_i} \dots} \}$ denotes the descendants, as well, of the final $n_i$ $T_{\sigma^j_{n_i}}$. In terms of $\sigma$, the dots just mean add arbitrary further bits to the right of $\sigma$. We have shown (Lemma \ref{lm:final_k}) that these families are exponentially large in $n_i$ (and also $k$ for the number of $k$th-descendents). We use these huge families to select for each $n_i$ and $k$ a final $n_i$ (${n_i}$-$f$), $T_{\sigma^j_{n_i}}$, if $k = 0$, or one of its descendants if $k \geq 1$, with a favorably small cross sectional area $\leq a$, $a$ exponentially small in $(n_i + k)$ ($a$ as in Proposition \ref{prop:final_4}). The notion of \emph{residual} in rdwh is crucial in the coming area estimate: rdwh of various ${n_i}$-$f$'s cannot overlap, so they must divvy up the total area, $\pi$, of any given plane.

Having gained control of the $a$ variable, Proposition \ref{prop:final_4} presents an alternative: either the $\alpha$-arc has \emph{deep feet}, i.e.\ both endpoints lie on the distant plane, $P^\pr$ in the proposition, one in $\Delta_i^{P^\pr}$ and the other in $\Delta_f^{P^\pr}$ or Proposition \ref{prop:final_4} tells us $T_\sigma$ (written $T$ in the proposition) has its daughters' $(T_{\sigma 0}$ and $T_{\sigma 1}$) turning points minimally displaced from its own.

Let us look ahead, perhaps cryptically, to the \emph{endgame} and then explain in detail. In the initial alternative, an exponential relation between $a$ and $d$ is observed at scale $\epsilon \approx \frac{1}{n_i}$. In the second case, we proceed to the \emph{endgame} where we see that overly-small, called \emph{timid} in the text, turning point displacements of the descendants $\{T_{\sigma^j_{n_i} \dots}\}$ are incompatible with the presumed shrinking of the Bing decomposition $\mathcal{D}$. In eliminating the possibility of all timid turning points, the J-lemma and much of Proposition \ref{prop:final_4}'s reasoning is recycled to find a new lower bound for $d$ (denoted $d_k$) in $(a_k, d_k)$-stretching at scale $\epsilon \approx d_k = \frac{1}{(n_i/2)L_j(n_i/2)}$ for some $j=0,1,2,\dots$, $L_j \coloneqq L_j^1$, and $L_j$ the function defined in Theorem \ref{paper-thm}. Because of the logs in the denominator, the endgame lower bound on $d$ is slightly weaker than our initial one and results only in a nearly exponential lower bound to the inherent modulus of continuity (imoc). The bound is weaker because while the first alternative (the $\alpha$ in Proposition \ref{prop:final_4} has deep feet) produces $d = \frac{1}{n_i}$, the endgame must be satisfied with this slightly smaller $d_k$. New ``planes'' added to $\mathbb{P}_{n_i+1}$ in the endgame all must fit into a final $n_i$ of length $\approx 1$ (see Figure \ref{fig:inf-planes}), dictating the more rapid convergence of $\{d_k\}$ when viewed as a series. Thus, for some sequence of $\epsilon \ra 0$ we locate some case, either through Proposition \ref{prop:final_4} or the endgame, where $(a,d)$-stretching yields an exponential or nearly exponential relationship between $\delta \approx \sqrt{a}$ and $\epsilon \approx d$. Of course we express all this through the function $\delta^{-1}(\epsilon^{-1})$.

Having introduced the idea of \emph{turning points} and the \emph{displacement} it is convenient to summarize Propositions \ref{prop:final_4} as taking a ${n_i}$-$f$, $T_m$, as input and yielding at least one of the following three outputs:

\subsubsection{Trichotomy}\label{sec:tri}

\begin{enumerate}
    \item $(a,d)$-stretching established via (the corollary to) the bag lemma \emph{near} one of its daughters, say $T$, of $T_m$.
    \item $(a,d)$-stretching established via the J-lemma \emph{near} $T_m$, or
    \item A \emph{rather precise} monotone picture (see Figure \ref{fig:precise_pic}(a)) for the turning points of the granddaughters $(T_0, T_1)$ of $T_m$ inside $T$.
\end{enumerate}

Recall $(a,d)$-stretching involves finding a point $y \in \Delta^+$, where $\Delta^+$ is the full disk ``completion'' of a dwh $\Delta$ in $T$ (case 1) or $T_m$ (case 2). Because $y$ may lie in $\Delta^+ \setminus \Delta$ we use the word ``near'' to record the relationship $y$ has to the sBt ($T$ or $T_m$). It is not really a metrical notion; it is used to convey the above relationship.

The words \emph{near} and \emph{rather precise} have technical meanings. We work in the monotone picture for $T$, in which extra wiggles across planes are suppressed. Figure \ref{fig:precise_pic}(a) shows no displacement of turning points and is called \emph{precise}. Figure \ref{fig:precise_pic}(b) shows the ``rather precise'' case defined below. By definition, \emph{rather precise} turning points will mean that when $T_0$ (and $T_1$) is drawn in the monotone picture w.r.t.\ planes $(P, P^\pr, P^{\pr\pr}, Q^{\pr\pr}, Q^\pr, Q)$, we see that $T_0$, up to local wiggles, essentially meets $P^{\pr\pr}$ and $Q^{\pr\pr}$ once with each orientation (and is allowed to meet $P$, $P^\pr$, $Q^\pr$, and $Q$ arbitrarily).

\begin{figure}[p]
    \centering
    \input{Inserts/Fig2.12.tex}
    \caption{}\label{fig:precise_pic}
\end{figure}

\subsection*{The endgame idea}
\begin{figure}[ht]
	\centering
	\input{Inserts/Fig2.13.tex}
	\caption{}\label{fig:useful_arc}
\end{figure}

Any large displacement of turning points creates an $\alpha$-arc with two \emph{deep feet} on a plane $P_k^\pr$ separated by a plane $P_k^{\pr \pr}$  from $\beta$, the longitude of the sister sBt. This $\alpha$-arc sets up an application of the J-lemma to locate $(a_k,d_k)$-stretching.

The long arc lemma might be better called the ``hooked arc lemma,'' in its applications, since 0-displacement produces a long arc but one that will \emph{not} ``hook''\footnote{The term ``hook'' describes non-monotone progress of a daughter's rdwh through the cyclically ordered family of its mothers rdwh.} at its own scale (it will hook at previous scales) in that it does \emph{not} proceed to clasp its sister and then return in the opposite direction, as measured along the $S^1$-coordinate of the mother. An undisplaced $\alpha$ merely proceeds between the turning points of the previous stage. Displacement lengthens one end of one daughter and shrinks the corresponding end of the other daughter. The one now lengthened by some amount $d_k$ is retained because it implies the useful $\alpha$ arc (see Figure \ref{fig:useful_arc}); it allows an application of the J-lemma at the mother's level.

We will show how to extend Proposition \ref{prop:final_4} to the daughters of a final $n_i$, ${n_i}$-$f$, and beyond that yield either $(a_k,d_k)$-stretching or displacement restrictions on the successive daughters.

Let us now explain both the geometry and calculus sketched above. A final $n_i$ has, in particular, $n_i$ index transitions representing passages between packets of planes of $\mathbb{P}_{n_i+1}$ which were organized to have total thickness 1, and so are at spacing $\frac{1}{n_i}$. Thus, the total \emph{combinatorial length}, consisting only of $x$-distance between transitional rdwh is $n_i(\frac{1}{n_i}) = 1$. A final $n_i$ might be \emph{monotone} and pass through approximately $\frac{n_i}{2}$ planes in one direction, and then return. On the other hand, it might wiggle $n_i$ times back and forth between two adjacent planes, or do something in between these two extremes.

Regardless of whether $T_m$ a final $n_i$ is monotone or folded, we can use its cyclically ordered $n_i$ index transitions between rdwh to extract one representative rdwh from each equal-index batch to play exactly the same role as the original elements of $\mathbb{P}_{n_i+1}$ in defining $r\ell$ for descendents of $T_m$. Since $r\ell(T_m) = n_i$ and the consecutive $x$-separation of the distinct indexed rdwh is $\frac{1}{n_i}$, $T_m$, has combinatorial length 1. We now need to introduce an infinite family of ``planes'' inside $T_m$ for use in measuring the displacement of turning points as we pass from each generation to the next inside $T_m$. These ``planes'' are in actuality pairs of rdwh, paired outward from the two turning points of $T_m$ in its mother. Using combinatorial length$(T_m) = 1$ and a conventional distance of $\frac{1}{n_i}$ between representative rdwh of locally constant index, we supply an infinite family of \emph{new planes} starting from turning points$(T_m)$ inward so that the total length of the intervening chambers converge to less than $\frac{1}{2}$, leaving at least some small gap in the middle. To retain the precisely exponential relationship between $\epsilon$ and $\delta$, we would need the spacing of these new planes to decay harmonically: the $k$th new plane separated from its predecessor by constant $\frac{1}{k}$. Unfortunately, this won't work, the harmonic series diverges so these new planes won't fit around a circle of length, or as we view it, across a span of $\frac{1}{2}$. This is where the poly logs come in. They arise when looking for the minimal ``compression'' of the harmonic sequence to one that converges to a constant. Let us do a small calculus exercise to see how this plays out. Setting $u = \log x$ below:
\begin{gather*}
    \int_n^\infty \frac{1}{x\log^{(1+\alpha)}x}\ dx = \int_{\log n}^\infty u^{-1-\alpha}\ du = -\frac{1}{\alpha}u^{-\alpha}\Big\vert_{\log n}^\infty = \frac{1}{\alpha}(\log n)^{-\alpha} \searrow 0, \text{ as } n \ra \infty \\[0.25em]
    \int_n^\infty \frac{1}{x \log x (\log \log x)^{(1+\alpha)}}\ dx = \int_{\log n}^\infty u^{-1}(\log u)^{-1-\alpha} du = \int_{\log \log n}^\infty v^{-1-\alpha} dv = \frac{1}{\alpha} (\log \log n)^{-\alpha} \searrow 0 \\
    \vdots \\
    \int_{e^j}^\infty \frac{1}{x \log x \log \log x \cdots (\underbrace{\log \cdots \log (x)}_{(j+1)\text{-times}})^{(1+\alpha)}}\ dx \searrow 0
\end{gather*}

These identities motivate the definition of our function $L_j^\alpha(x)$. The integrals dominate corresponding $k$-indexed sums offset by 1. A factor of $\lceil e^j \rceil$ in $L_j^\alpha$ allows us to start all sums at $k=0$, and absorbs the difference between using $e$, the natural base for exponentials, and the base $\sqrt{2}$ associated with lengths seen in cross sections of daughters versus their mothers.

For all $j \geq 0, \alpha > 0$ we can begin our sequence $\{n_i\}$, $n_i > n$, large enough so that the corresponding integral above converges to a value, say $< \frac{1}{10}$. This will allow the now infinite ``plane'' family $\mathbb{P}^\pr$, about to be constructed to have two disjoint copies sitting in the final $n_i$ of length $< \frac{1}{10}$. The climax of the endgame is that if after each generation within $T_m$ we never see an opportunity to apply the J-lemma to conclude nearly\footnote{``Nearly'' refers to the insertion of the function $L_j^\alpha$.} exponential moc, then there is a branch through the tree of descendents where the integrated displacement of the two turning points is inadequate for shrinking. The ``circle'' $T_m$ has combinatorial length one, with initial end points distance $\frac{1}{2}$ apart and neither can move more than $\frac{1}{10}$, without creating $(a,d)$-stretching at some step. This could contradict the shrinking of the decomposition $\mathcal{D}$ to points.

What remains is to translate over to the endgame the topological apparatus of Proposition \ref{prop:final_4} and also to give a careful account of cross-sectional $\text{areas}^+$ as we descend the tree of daughters. Exponential decay of cross-sectional $\text{areas}^+$ as a function of generation follows if the $\text{areas}^+$ divide more or less equally among daughters. We must show that such symmetry is not needed---there are enough descendants to sort through to find many with the required condition on $\operatorname{area}^+$.

We now explain the countable family of ``new planes,'' actually pairs of rdwh to be added to $T_0$ (or $T_1$ assuming $T_0$ was rather precise). Recall if $T_0$ and $T_1$ were both imprecise (the opposite of rather precise), then $(a,d)$-stretching is produced by one of the three alternatives listed in the trichotomy (\ref{sec:tri}). In this construction the log-like function $L_j^\alpha$, $L$ for short, is fixed and $n_i$ is chosen large enough so that $\sum_{k=1}^\infty \frac{1}{(n_i+k)L(n_i+k)} < \frac{1}{10}$, so there will be plenty of room inside any of the $n_i$-$f$ for our construction. From here, to avoid clutter, we drop the subscript, writing $n$ for $n_i$.

$T_0$ being the daughter of a $n-f$ has $r\ell(T_0) = n-2$. Because we may assume both $T_0$ and $T_1$ are rather precise, there will be $\frac{n-2}{2}$ $\frac{1}{n}$-sized inter-plane chambers in $T_0$, each of the two directions leading from the lower to the upper turning point of $T_0$ in $T$, where $T_0$ clasps tightly with $T_1$. To establish notation let $\Delta_i^-$ and $\Delta_f^-$ be initial and final rdwh at the lower turning point of $T_0$ where clasping with $T_1$ occurs and $\Delta_i^+$ and $\Delta_f^+$ be initial and final rdwh at the upper clasping with $T_1$. The ``new planes'' are rdwh located on the continuum of planes parallel to $\mathbb{P}_{n+1}$ located between the elements of $\mathbb{P}_{n+1}$ that $T_0$ meets. The first 4 ``new planes,'' that is, 8 new rdwh, are defined by starting at $\Delta_i^-$ ($\Delta_f^-$) and selecting any rdwh in the two planes at successive distances $\frac{1}{(n+1)L(n+1)}$ into the packet of planes from $\Delta_i^-$ ($\Delta_f^-$), and similarly spaced from the other end, $\Delta_i^+$ ($\Delta_f^+$). See Figure \ref{fig:new-planes}. This is similar to our earlier construction of $P^\pr$ and $P^{\pr\pr}$ near $P$. Next, add 8 more rdwh on parallel planes at additional spacings of $\frac{1}{(n+2)L(n+2)}$ and $\frac{2}{(n+2)L(n+2)}$. Continue in this way until the next group of 8 rdwh will not fit within the $\frac{1}{n}$-thick slab between adjacent planes of $\mathbb{P}_{n+1}$. Since $n$ is assumed very large, we are free to assume that $L(n) > 500$, so that we can repeat the procedure $> 100$ times before we are too close to the end of our $\frac{1}{n}$-thick slab.

When we \emph{do} come close to the end of this $\frac{1}{n}$-thick slab, we leave off adding ``new planes'' but continue around $T_0$ in the same orientation until we arrive at another plane $P$ of $\mathbb{P}_{n+1}$ and into the initial portion of a new segment of $T_0$, essentially spanning between $P$ and some adjacent plane $Q$ of $\mathbb{P}_{n+1}$. This segment leaves $P$ on some final rdwh which is added to our growing collection of ``new planes.'' From this fresh starting point, we add 100 or more groups of two new planes (at this point 4 separate groups will be growing two from the bottom clasp up and two from the top clasp down) until encountering the next plane of $\mathbb{P}_{n+1}$. At this time we pause, follow $T_0$ until we recommence at the next essential segment of $T_0$ joining planes of $\mathbb{P}_{n+1}$ and there also add another 100 or more planes to the growing strings. Continue constructing all four strings of ``new planes'' in this fashion.

Recall the total length (measured in the $x$-coordinate only) of essential segments in both paths through $T_0$ between the lower and upper clasps is $\approx \frac{1}{5}$. We have chosen $n$ so large that less than 1\% of such length is ``wasted'' at unused ends of segments. Since we also chose $n$ large enough that the sum $\sum_{k=1}^\infty \frac{1}{(n+k)L(n+k)} < \frac{1}{10}$ we have ample room for requisite two copies of that sum, from both ends leading to 4 total copies, even allowing for a 1\% inefficiency, to fit into $T_0$ and still leave a gap in the middle since $4(\frac{1}{10}) < (0.99)\frac{1}{2}$. Actually, since there are two strings of new planes between two lower and upper clasps, there are a total of 8 copies of this sum in $T_0$: 2 corresponding to 2-strings, 2 corresponding to the two ends of each string, and 2 corresponding to the double-chamber structure essential to Proposition \ref{prop:final_4}, which is similarly crucial in the endgame.

Figure \ref{fig:new-planes} summarizes the above construction of new planes pictorially whereas Figure \ref{fig:inf-planes} summarizes the numerics along one (either) of the two directions.

\begin{figure}[htp]
    \centering
    \input{Inserts/Fig2.14.tex}
    \caption{$r\ell(T_0) = n-2=8$. All ovals and dots indicate a ``new plane'' rdwh.}
    \label{fig:new-planes}
\end{figure}
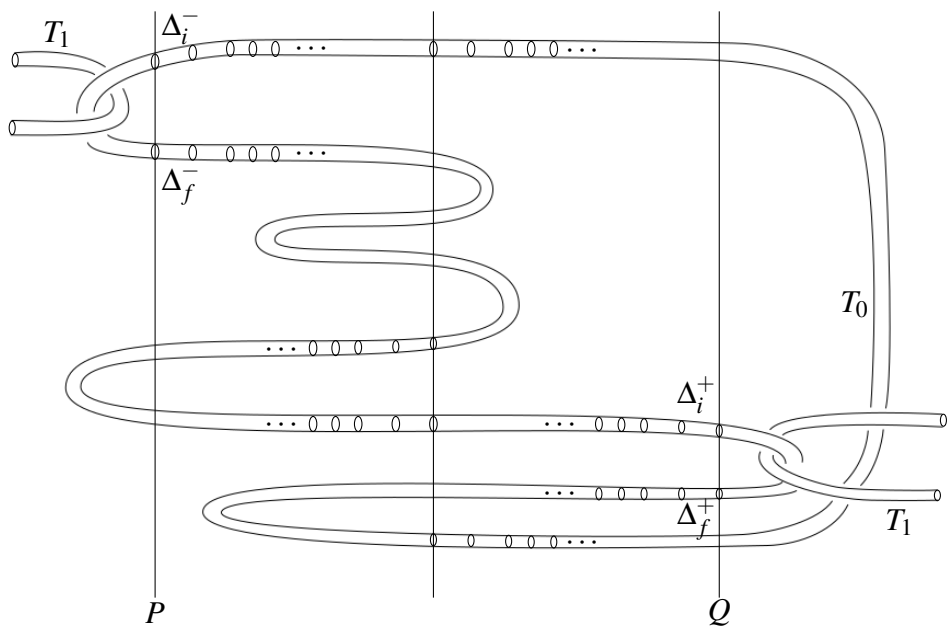

\begin{figure}[ht]
	\centering
	\input{Inserts/Fig2.15.tex}
	\caption{$L \coloneqq L_j^\alpha$, some $\alpha > 0$, $j = 0,1,2,\dots$, and $n_i$ sufficiently large that the slabs all fit in length $\frac{1}{2}$.}\label{fig:inf-planes}
\end{figure}
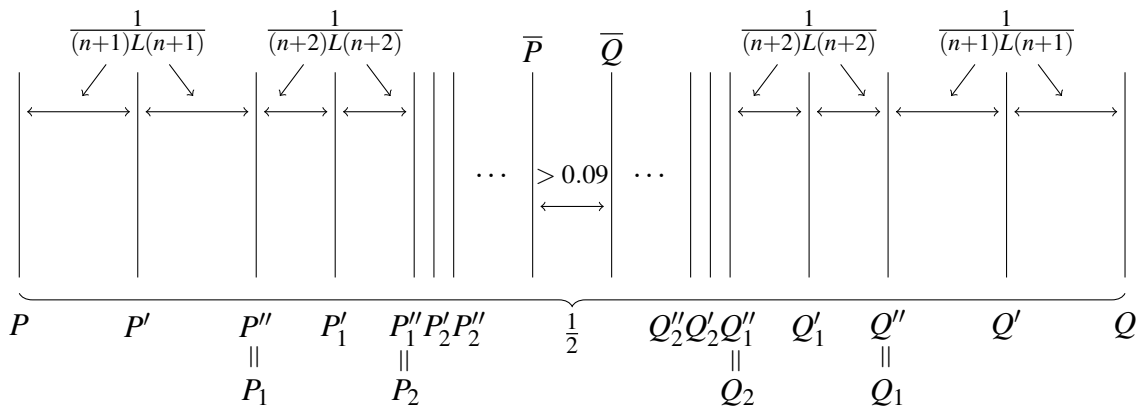

The new planes just located allow the trichotomy of Proposition \ref{prop:final_4} to be iterated: (1) $(a,d)$-stretching near $T$ (or $T_s)$ from the bag lemma, (2) $(a,d)$-stretching near $T_m$ by J-lemma, or (3) the turning points of $(T_0,T_1)$ in $T$ are rather precise. So, $(a,d)$-stretching is located with $\epsilon \approx \frac{1}{n_i}$ or $(T_0,T_1) \subset T$ is rather precise. The beauty of the rather precise condition is that the $E$ disk boundary longitude $T_1 = \beta$ has its boundary $\beta$ well-positioned to apply the J-lemma, now one step deeper in the tree of daughters than was used for case (2) of Proposition \ref{prop:final_4}. Recall that interior intersections of $E$ with a plane it should not cross ($P^\pr$ in Figure \ref{fig:6_strings}) can be removed via an innermost circle argument, but intersections of $\de E = \beta$ with that plane would be fatal for a J-lemma application. The rather precise positioning $(T_0,T_1) \subset T$ implies $\beta \cap P^\pr = \varnothing$, allowing a J-lemma application one generation down (and similarly reversing the roles of $T_0$ and $T_1$). So, we see that in case (3) where we have not yet seen $(a,d)$-stretching, we at least are set up to look for it again, by the machinery of Proposition \ref{prop:final_4}, in the next generation. Doing so requires inspecting how $(T_{00},T_{01}) \subset T_0$ and $(T_{10},T_{11}) \subset T_1$, sit w.r.t.\ the first 4 new planes (8 rdwh added in our construction). Repeating the trichotomy we will either find $(a_1,d_1)$-stretching near, say, $T_0$ using the bag lemma, $(a_1,d_1)$-stretching near $T$ using the J-lemma, or learn that $(T_{00},T_{01}) \subset T_0$ and $(T_{10},T_{11}) \subset T_1$ are both rather precise. Again, if we are forced into case three, the rather precise condition implies well-located $E$ disks, now for doubly indexed $T$'s and we may again apply the J-lemma in this generation to realize the same trichotomy, using the next 8 new planes. The calculus exercise permitted us to fit in an unlimited number of these groups of 8 along $T$'s winding course (Figure \ref{fig:new-planes}) and still leave a gap between new planes coming up from the bottom of the slab and new planes coming down from the top. But such a gap between turning points as stage depth approaches infinity contradicts the shrinking of the Bing decomposition $\mathcal{D}$ to points.

Let us supply a little more notation and a few diagrams to help follow this argument in detail. We already referred to our attempt to locate $(a_1,d_1)$-stretching using the first batch of 4 new planes. In general, $(a_k,d_k)$-stretching is what would be found if alternatives (1) or (2) apply at generation $k$ below $T$ using the $k$th group of 4 planes. Clearly, the relevant length is $d_k = \frac{1}{(n+k)L(n+k)}$, the inter-plane spacing at that stage. After clarifying our notation, the final point will be to estimate $a_k$ in terms of $n$ and $k$.

In Proposition \ref{prop:final_4}, 6 planes measure $x$-axis extent: $P$, $P^\pr$, $P^{\pr\pr}$, $Q^{\pr\pr}$, $Q^\pr$, and $Q$. As we proceed to deeper stages we add 4 new planes (or occasionally 6 at the start of a new slab, in which case the formal equalities written below do not hold) at each step, according to the pattern:
\[
	P, P^\pr, P^{\pr\pr}=P_1, P_1^\pr, P_1^{\pr\pr}=P_2, P_2^\pr, P_2^{\pr\pr}, \dots, Q_2^{\pr\pr}, Q_2^\pr, Q_2=Q_1^{\pr\pr}, Q_1^\pr, Q_1=Q^{\pr\pr}, Q^\pr, Q
\]

For symmetry, use the notation $P^{\pr\pr} = P_1$, $P_1^{\pr\pr} = P_2$, ... , $Q_1^{\pr\pr} = Q_2$, $Q^{\pr\pr} = Q_1$ when applicable. Looked at this way, at each stage we add 6 planes but the outer two usually agree with the inner two of the previous stage. So usually only 4 of the 6 planes are new. We will now measure $x$-axis extent. Similarly to the definition of $r\ell$ for the original plane family $\mathbb{P}_{n_i+1}$, by at each stage looking, at any time, only at rdwh for the sBt and a finite initial segment of $\{\text{new planes}\}$. Recall the discussion associated with the definition of rdwh, which assures us that the definition of rdwh's is stable with respect to adding planes and descendant tori. Again, we use the shorthand $T_\sigma \cap \text{plane} = \varnothing$ to mean the intersection contains no rdwh.

To summarize, the final bit of the proof leads us down the tree of daughters, stopping successfully, if case (1) or case (2) of the trichotomy (\ref{sec:tri}) occurs, otherwise continuing always to case (3). Inductively, the structure of sister tori allows us to continue the method of Proposition \ref{prop:final_4} to find case (1) or (2) stretching, or if we witness a very limited displacement, case (3).

Assuming there is no $(a,d)$-stretching at the previous $(T_m; T_\sigma)$ level, Figure \ref{fig:literal_pic}, below, is a monotone picture, i.e.\ certain palindromic repetitions have been edited out for visual convenience. We see in $\tld{T}_\sigma$ (written $\tld{T}$ going forward) with numerous rdwh of $T$ on $P$, $P^\pr$, and $P^{\pr\pr}$ edited out. Refer to Figure \ref{fig:6_strings} to decode the labeling of Figures \ref{fig:literal_pic} and \ref{fig:monotone_pic}. Since edited length $\leq$ length, the simpler monotone picture is still useful for real-axis-length lower bounds.

\begin{figure}[ht]
	\centering
	\input{Inserts/Fig2.16.tex}
	\caption{}\label{fig:literal_pic}
\end{figure}

Proposition \ref{prop:final_4} says no $(a,d)$-stretching $\implies$ lift$(T_0)$ meets at most one $P^\pr$ lift, that is, no deep feet, and the lift$(T_1)$ meets at most one $Q^\pr$ lift. By our conventions, $T_0$ meets $P$ (in $\vphantom{\Delta}_m\Delta_f^P$) and $T_1$ meets $Q$, and similarly, by convention, we will assume $T_{00}$ meets $\Delta_f^{P^{\pr\pr}}$. Figure \ref{fig:monotone_pic} is the monotone picture one stage deeper.

\begin{figure}[ht]
	\centering
	\input{Inserts/Fig2.17.tex}
	\caption{New planes inserted at next stage if $(a,d)$-stretching is not discovered by Proposition \ref{prop:final_4}, nor $(a_1, d_1)$-stretching detected in Figure \ref{fig:literal_pic} with the immediate daughters of the 4-$f$.}\label{fig:monotone_pic}
\end{figure}
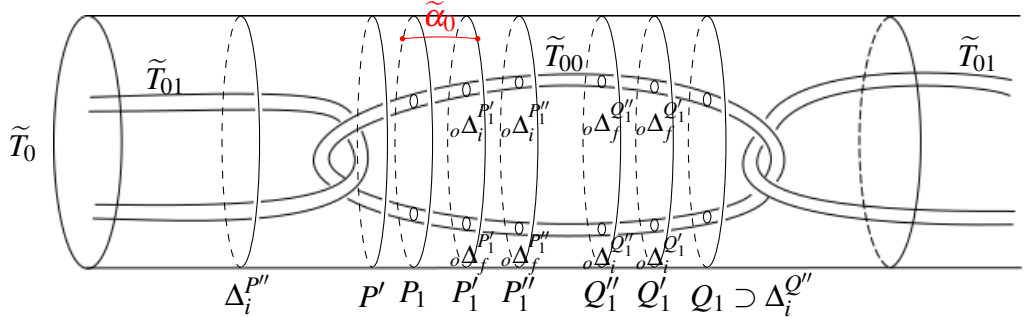

In the first paragraph of the proof of Proposition \ref{prop:final_4} we discussed how loose the clasps of $T_0$ and $T_1$ can occur without that looseness producing the hypotheses of Proposition \ref{prop:final_4}, and, therefore, producing $(a,d)$-stretching. Our conclusion there, translated to the monotone picture above is that the clasps there are completely tight and rather precise (as drawn): the left lift of $T_1$ has rdwh on $P^\pr$ but not $P^{\pr \pr}$ and the right lift of $T_1$ has rdwh on $Q^\pr$ but not $Q^{\pr \pr}$. Otherwise $(a,d)$-stretching would already occur at the $n$-$f$ level. This tightness means that the primordial $E$ (from the unshrunk Bt) with $\de E = \beta = \operatorname{longitude}(T_s)$, $T_s$ being the sibling of $T$, has the requisite disjointness from $\Delta_f^{P^{\pr \pr}}$ and from $\Delta_i^{Q^{\pr \pr}}$ to serve as input for an application of the J-lemma using $\alpha$ with deep feet on $P_1^\pr$. Notice that as drawn in Figure \ref{fig:literal_pic} (and consistent with Figure \ref{fig:6_strings}), $\tld{\alpha}$ does not reach as far right in Figure \ref{fig:6_strings}, or as far left in Figure \ref{fig:literal_pic}, as $\Delta_i^{P^\pr}$. If it did, using $\beta$ we could successfully apply the J-lemma and conclude $(a_1,d_1)$-stretching. If it does not, we must proceed to a daughter $T_{00}$ or $T_{01}$ to continue the search for $(a_2,d_2)$-stretching.

Associated to $T_{00}$ is a shadowing $\alpha_0$-arc (as in the long arc lemma) on $\de T_0$. Exactly as in the proof of Proposition\ \ref{prop:final_4}, but now, to avoid the J-lemma producing $(a_1,d_1)$-stretching, the $\alpha_0$-arc must \emph{not} have both endpoints (deep feet) on $P_1^\pr$ (i.e.\ should not reach that far left in Figure \ref{fig:monotone_pic}). $a_1$ is taken to be the maximum of $\operatorname{area}^+$ of the next arising 10 rdwh, 8 of which are new rdwh's of $T_0$ denoted wth a back script 0: $\Delta_f^{P^{\pr\pr}} = \Delta_f^{P_1}$, $\vphantom{\Delta}_0\Delta_i^{P_1^\pr}$, $\vphantom{\Delta}_0\Delta_i^{P_1^{\pr\pr}}$, $\vphantom{\Delta}_0\Delta_i^{Q_1^{\pr\pr}}$, $\vphantom{\Delta}_0\Delta_i^{Q_1^\pr}$, $\Delta_i^{Q^{\pr\pr}} = \Delta_i^{Q_1}$, $\vphantom{\Delta}_0\Delta_f^{Q_1^\pr}$, $\vphantom{\Delta}_0\Delta_f^{Q_1^{\pr\pr}}$, $\vphantom{\Delta}_0\Delta_f^{P_1^{\pr\pr}}$, and $\vphantom{\Delta}_0\Delta_f^{P_1^\pr}$. $d_1$ is the spacing between the next tranche of planes $P_1 \leftrightarrow P_1^\pr$, $P_1^\pr \leftrightarrow P_1^{\pr\pr}$, $Q_1^{\pr\pr} \leftrightarrow Q_1^\pr$, and $Q_1^\pr \leftrightarrow Q_1$, soon to be set to $d_1 = \frac{1}{(n+1)L(n+1)}$. The key observation is that $\beta = \de E$, $E$ the natural disk spanning $T_1$ (see Figure \ref{fig:finite_approx}(c)) has the requisite disjointness to serve as input to the J-lemma. 

Similarly, there is an $\alpha_1$-arc shadowing $T_{01}$ also on $\de T_0$ and to avoid $(a_1,d_1)$-stretching (from the J-lemma) it cannot have both its endpoints as far left as $Q_1^{\pr\pr}$. The J-lemma will accept as input either $\alpha_0 \subset \de T_0$ and the disk $E$ with $\de E = \beta = \operatorname{longitude}(T_1)$, or the reverse: $\alpha_1 \subset \de T_1$ and the primordial $F$ with $\de F = \beta = \operatorname{longitude}(T_0)$.

The shadowing property (long arc lemma) of the $\alpha$-arcs now implies that none of the four width $d_1$ slabs $\lbar{P_1 P_1^\pr}$, $\lbar{P_1^\pr P_1^{\pr\pr}}$, $\lbar{Q_1^{\pr\pr}Q_1^\pr}$, nor $\lbar{Q_1^\pr Q}$ can be disjoint from both $\alpha$ arcs, ($\alpha_0 \cup \alpha_1$). There are only two combinatorial possibilities for $\alpha_0 \cup \alpha_1$ to lie in $T_0$ and these are pictured in $\tld{T}_0$ in Figure \ref{fig:redblue} below.

\begin{figure}[ht]
	\centering
	\input{Inserts/Fig2.18.tex}
	\caption{}\label{fig:redblue}
\end{figure}

The two cases illustrated in red and blue (resp.) correspond in Figure \ref{fig:6_strings} (there at one deeper level) to shuttling the two clasps of $T_0$ and $T_1$ either to the NW and SE of the central region, as drawn there, or alternatively to the SW and NE. The dotted colored lines in Figure \ref{fig:redblue} correspond to the fact that daughters may not clasp tightly; a loose clasp can extend the range (in one direction) of the overlap of $T_{00}$ and $T_{01}$ and hence of $\alpha_0$ and $\alpha_1$. But, usefully, such loosening has no effect when viewed in the monotone picture, until the clasps become so loose as to violate the rather precise condition. Then a lift of $T_{00}$ or $T_{01}$ crosses an entire 6-plane packet of either $P$ or $Q$-planes (above) leading immediately to $(a_1,d_1)$-stretching by an application of the J-lemma.

Comparing Figures \ref{fig:monotone_pic} and \ref{fig:redblue}, observe that while $T_0$ spans across planes $P^{\pr\pr}$ and $Q^{\pr\pr}$, $T_{00}$ spans across planes $P_1^{\pr\pr}$ and $Q_1^{\pr\pr}$, that is, the data used to compare $(a,d)$-stretching and the turning points of daughters has been replicated at the next stage. At this next stage, the longitude for $T_{01}$ will bound a disk which we might call $E_{01}$ to set up a J-lemma application on $T_{00}$ (and vice versa, a longitude for $T_{00}$ bounds a disk $E_{00}$ setting up the J-lemma on $T_{01}$). The constants $(a,d)$ have updated to $(a_1,d_1)$. Inductively, we may now analyze stretching down the entire tree of descendants of $T_\sigma$ of a final $n_i$-$f$. We will see that avoiding $(a_k, d_k)$-stretching at $k$-stages below the final $n_i$, constrains the successive turning point displacement so much that the descendants of $T_\sigma$ cannot shrink (as they must!) unless the $(a_k, d_k)$-stretching is sufficient to establish a nearly exponential lower bound to the modulus of continuity. Some fraction of descendants must be discarded at each stage for lack of a suitable area estimate. We address this final issue next.

To obtain our modulus of continuity lower bound on $I^h$, for each $\mathbb{P}_{n_i+1}$ (with an initial torus $T_\varnothing$ of $r\ell(T_\varnothing)=2n_i$), we must locate a suitable final $n_i$-$f$ with small $\text{area}^+$ rdwh, and also at each subsequent stage select one of its descendants exhibiting an exponential decay in $\operatorname{area}^+$ of rdwh on the relevant planes. Once we are in the endgame, at any stage $k$, any imprecise torus may be used to find $(a_k,d_k)$-stretching, so we may adopt the policy: whenever we see a pair of rather precise daughters we follow (only) the descendants of the smaller volume one. This guarantees that when the lineage finally produces an imprecise daughter, on which shrinking of \emph{each} component of $\mathcal{D}$ depends, that its volume is $< \frac{1}{2^k} \operatorname{vol}(n_i$-$f)$. Then our presumed ability to vary new plane positions by 10\% implies, via Fubini's theorem, that $\text{area}^+$ of the relevant rdwh also decay exponentially in $k$. So our remaining concern is finding an $n_i$-$f$ with exponentially small cross-sectional $\text{areas}^+$. Logically, given $n_i$ we need only a single small sBt $n_i$-$f$, but to find \emph{one} with small $\text{area}^+$ we need to produce exponentially \emph{many}. Going forward we drop the subscript to $n_i$.

Lemma \ref{lm:final_k} supplies about $2^{\frac{n}{2}}$ $n$-$f$'s. So the average volume of an $n$-$f$ is $\leq 2^{-\frac{n}{2}}$. Recall the rdwh in the planes $P$, $P^\pr$, $P^{\pr\pr}$, $Q^{\pr\pr}$, $Q^\pr$, and $Q$ in the proof of Proposition \ref{prop:final_4}. To find a $n$-$f$ playing the role of $T_m$ whose relevant rdwh have small $\text{area}^+$, first note:

\begin{customlm}{Key Fact}
    Rdwh on any given plane associated to disjoint solid tori are \emph{unnested} so must divvy up the area, call it 1, of each plane. $\sum \operatorname{area}^+ < 1$, the sum being over all rdwh in a given plane. (This fact motivated the definition of rdwh.)
\end{customlm}

\begin{proof}
    Each ``hole'' in a dwh bounds a disk in its 2D torus boundary. The canonical isotopy $\mathcal{I}$ in the definition of rdwh removes all such circles of intersection together with any torus meridians lying in the planar disks they bound. Thus, after $\mathcal{I}$, a residual meridian bounds a disk with no other meridians inside it.
\end{proof}

By Fubini's theorem, every rdwh occurring in the proof of Proposition \ref{prop:final_4} can be assumed to have $\operatorname{area}^+ < 10n\operatorname{vol}(T)$ using the freedom to vary plane locations by 10\%. Since the $\approx 2^{\frac{n}{2}}$ disjoint $n$-$f$'s, $T$ must have average volume $\leq 2^{-\frac{n}{2}}$, some must have volume that small and therefore the $\text{areas}^+$ of the rdwh in the proof must satisfy $\operatorname{area}^+(\mathrm{rdwh}) \leq 10n2^{-\frac{n}{2}}$. The linear factor is irrelevant; we see that it is safe to assume we may input a $n$-$f$ into Proposition \ref{prop:final_4} so that the $\operatorname{area}^+$ in the theorem decays exponentially, like $c^{-n}$, $c>1$.

Now we have shown $(a,d)$-stretching---if it occurs before the endgame---with $a \approx c^{-n}_0$, $c_0 > 1$, and $d \approx \frac{1}{n}$, that is, $\delta < \frac{\sqrt{a}}{\pi}$ and $\epsilon \geq d$. This pre-endgame stretching would be fully exponential, and detected by Proposition \ref{prop:final_4} in outcomes (1) or (2). In the endgame, $a_k \leq c_0^{-n}(\frac{1}{2})^k$ and $d \geq \frac{0.8}{(n+k)L(n+k)}$. So, there is a constant $c>1$ so that for $n$ sufficiently large, and setting $m \coloneqq n+k$, there are points separated by no more than $c^{-m}$ whose $I^h$ images are at least $\frac{0.8}{m L(m)}$ apart. Thus, writing $\delta^{-1}$, as in the statement of Theorem \ref{paper-thm}, as a function of $\epsilon^{-1}$ we have:
\[
    \delta^{-1}\left(\frac{mL(m)}{0.8}\right) \geq c^m = c^{\left[m \frac{L(m)}{0.8} \slash \frac{L(m)}{0.8}\right]} \geq c^{\left[m \frac{L(m)}{0.8} \slash L\left(m \frac{L(m)}{0.8}\right)\right]}
\]
where the final inequality follows from monotonicity of $L$ and $L(m\frac{L(m)}{0.8}) > \frac{L(m)}{0.8}$, for $n$ sufficiently large. Thus, for any function $L = L_j^\alpha$, and $\epsilon > 0$ sufficiently small, $\delta^{-1}(\epsilon^{-1}) \geq c^{\epsilon^{-1}/L(\epsilon^{-1})}$. By the form of the exponent, i.e.\ since $L(\epsilon^{-1})$ approaches infinity as $\epsilon \ra 0$, and the ratios $L_j^\alpha / L_{j+1}^\alpha$ also approach infinity, there is no loss of generality in taking the base of the exponent to be $e$, yielding
\[
    \delta^{-1}(\epsilon^{-1}) \geq e^{\frac{\epsilon^{-1}}{L(\epsilon^{-1}}}
\]
This completes the proof of Theorem \ref{paper-thm}.
\qed

\section{Scholium and Conjectures}
In the usual definition of continuity of a function $h$, $\delta > 0$ is a function of $\epsilon > 0$, but it is convenient to speak of $\delta^{-1}(\epsilon^{-1})$. Lipschitz can then be identified as a linear relationship and H\"{o}lder as a power $>1$, $\delta^{-1}(\epsilon^{-1}) \leq (\epsilon^{-1})^p$, $p>1$. Recall that the functional relationship $\delta^{-1}(\epsilon^{-1})$ is the \emph{modulus of contintuity}, moc$(h)$. In this paper, we studied lower bounds for moc$(h)$, i.e.\ for $\delta^{-1}(\epsilon^{-1})$. We are only interested in the behavior of the moc near infinity, and only obtain information on a selected sequence of $\epsilon^{-1}$ approaching infinity. We use the notation $f(x) \overset{\infty}{\leq} g(x)$, to mean $f(x_i) < g(x_i)$ for a divergent sequence $\{x_i\}$. Note that $f(x) \overset{\infty}{\leq} g(x)$ is not a transitive relationship since different divergent sequences could be used. 

The ``Bing Involution'' $I$ is usually thought of as a small family of very similar involutions depending on exact choices of the shrinking homeomorphisms \cites{bing52,bing88}. Bing's two shrinks and all related shrinks share a feature that implies the inequalities in (\ref{eq:moc1}) and (\ref{eq:moc2}) below, namely, all sBt in those examples have path metrics quasi-isometric to the product of $S^1$ with a round disk. In these specific examples of shrinks, $\delta^{-1}$ must be chosen large enough so that
\begin{equation}
	\delta^{-1}(\epsilon^{-1}) \overset{\infty}{\geq} \sqrt{2}^{\epsilon^{-1}} \label{eq:moc1}
\end{equation}

Also in these examples
\begin{equation}
	\delta^{-1}(\epsilon^{-1}) \overset{\infty}{\leq} c^{\epsilon^{-1}} \text{ for some } c \label{eq:moc2}
\end{equation}
Probably $c$ around 4 is sufficient but we have not tried to optimize the \emph{clasping geometry} of Bing pairs in any detail, which is what the optimal choice of $c$ would involve.

(\ref{eq:moc1}) above is known for Bing's original shrinks and for quasi-isometric ones since at each stage Bing doubling reduces cross-sectional areas by a factor of at least 2, cross-sectional length by at least $\sqrt{2}$, and Bing's original shrink (approximately) rotates points through a distance of $\frac{1}{n}$ at the $n$th stage. Thus, there are points $(x,y)$ at the $n$th stage with dist$(x,y) \leq (\frac{1}{\sqrt{2}})^n$ and dist$(I(x), I(y)) = \frac{1}{n}$.

The main theorem of this paper is that if $I^h$ is any topological conjugate of $I$, then:
\begin{equation}\label{eq:moc3}
	\operatorname{moc}(I^h) \text{ obeys } \delta^{-1}(\epsilon^{-1}) \overset{\infty}{\geq} e^{\left(\frac{\epsilon^{-1}}{L(\epsilon^{-1})}\right)}
\end{equation}
$L$ any of the polylogs of Theorem \ref{paper-thm}. (Of course one should not expect \emph{upper} bounds on moc over the conjugacy class.)

Comparing lines (\ref{eq:moc1}), (\ref{eq:moc2}),  and (\ref{eq:moc3}), we see that our analysis is optimal up to a power $0 < p < 1$ applied to $\epsilon^{-1}$. This motivates the following definition of \emph{growth classes}. (Readers familiar with computer science will notice an analogy with the way distinct complexity classes are defined. There too some ``slop'' in the argument, logarithmic in the CS application, is introduced to account for time wasted as one Turing machine simulates another.)

For monotone increasing functions $f,g: [1,\infty) \ra [1,\infty)$, called \emph{growth functions}, we write $f \overset{\infty}{\prec} g$ iff there is a power $0 < p \leq 1$ such that for all $x$'s, $f(x^p) \leq g(x)$ (also written $g \overset{\infty}{\succ} f$). If $f(x^p) \leq g(x)$ and $g(x^{p^\pr}) \leq h(x)$ then $f(x^{pp^\pr}) \leq h(x)$, so $\overset{\infty}{\prec}$ is a transitive relation. We say $f \equiv g$ iff $f \overset{\infty}{\prec} g$ and $g \overset{\infty}{\prec} f$.

Equivalence classes of growth functions are partially ordered: $[f] \overset{\infty}{\prec} [g]$ by definition iff $f \overset{\infty}{\prec} g$, which does \emph{not} depend on the choice of the representative. Note that there is no maximal equivalence class: given $f$, define $\lbar{f}(x) = f(e^x)$; clearly $f \overset{\infty}{\prec} \lbar{f}$ and $\lbar{f} \not\overset{\infty}{\prec} f$.

\begin{scholium}\label{scholium}
	Given any growth function $f$, there is an orientation reversing involution $J$ of $S^3$ such that the modulus of continuity of any topological conjugate $J^h$ of $J$ grows faster than $f$, i.e.\ $\delta^{-1}_{J^h}(\epsilon^{-1}) \overset{\infty}{\succ} f(\epsilon^{-1})$, and if $f \overset{\infty}{\succ} \exp$, then $J$ itself can be constructed such that $\delta^{-1}_J(\epsilon^{-1})$ is within a constant multiple of $f(\epsilon^{-1})$ for a divergent sequence of values of $\epsilon^{-1}$.
\end{scholium}

\begin{proof}
	The necessary class of involutions $\{J\}$ is built by \emph{ramifying} the Bing doubling process and then shrinking the corresponding ramified decomposition $\mathcal{D}_r$. Ramification was depicted in Figure \ref{fig:bingpair}(b); at the $n_i$th stage, ramification is specified by a positive integer $r_{n_i}$ which counts the number of Bing pairs (twins) placed inside the mother. If $\{r_{n_i}\}$ is chosen to grow rapidly (think of the Acerman or busy beaver functions) then the typical volumes and cross-sectional areas $a_{n_i+k}$ of $n_i$th stage tori will decay rapidly in the plane families constructed in the proof of the main theorem. Taking square root $\sqrt{a_{n_i+k}}$ we obtain a length scale $\delta$ which the proof shows to be stretched out to at least a linear ($x$-axis) separation of $\frac{1}{(n_i+k)L(n_i+k)} \approx \epsilon$. Given any growth function which is \emph{at least} exponential, a ramification function $\{r_{n_i}\}$ can be chosen to closely keep pace (say never deviating by more than a factor of 2 at integer values). From $\{r_{n_i}\}$, build a ramified Bing decomposition $\mathcal{D}_{\{r_{n_i}\}}$ invariant under the standard involution $\tld{I}(x_1,x_2,x_3,x_4) = (x_1,x_2,x_3,-x_4)$ and let $J$ be the induced involution on the shrunk quotient:

	\begin{figure}[!ht]
		\centering
		\begin{tikzpicture}[scale=1.1]
			\node at (0,0) {$S^3$};
			\draw[->] (0.35,0) -- (1.65,0);
			\node at (2,0) {$S^3$};
			\node at (1,0.3) {$\tld{I}$};
			\draw[->] (0,-0.45) -- (0,-1.15);
			\node at (-0.25,-0.75) {\footnotesize{$\pi$}};
			\node at (0,-1.5) {$S^3 \slash \mathcal{D}_{\{r_n\}}$};
			\draw[->] (0.8,-1.5) -- (1.2,-1.5);
			\node at (2,-1.5) {$S^3 \slash \mathcal{D}_{\{r_n\}}$};
			\draw[->] (2,-0.45) -- (2,-1.15);
			\node at (3.2,-0.6) {\footnotesize{$\pi$-approximate}};
			\node at (3.38,-0.9) {\footnotesize{homeomorphism}};
			\node[rotate=-90] at (0,-2.25) {$\cong$};
			\node at (0,-3) {$S^3$};
			\node[rotate=-90] at (2,-2.25) {$\cong$};
			\node at (2,-3) {$S^3$};
			\draw[->] (0.4,-3) -- (1.6,-3);
			\node at (1,-2.75) {$J$};
			\node at (-2.8,0) {\hspace{0.25em}};
		\end{tikzpicture}
	\end{figure}

	Using Bing's shrink (now with ramification) to construct the final isomorphism, moc$(J)$ will grow proportionately to $f$ at least at a divergence sequence of values for $\epsilon^{-1}$. Furthermore, the proof of the main theorem shows that no other shrink, one producing a conjugate $J^h$, can reduce the moc growth from Bing's shrink by more than a polylog in the exponent. Consequentially, all conjugates $J^h$ will have equivalent moc$(J^h)$, modulus of continuity growth functions.
\end{proof}

We turn now to some conjectures and questions.

\setcounter{conj}{0}
\begin{conj}[Gap Conjecture]
	Let $J: S^3 \ra S^3$ be any orientation reversing involution on the 3-sphere. If $\operatorname{moc}(J) \overset{\infty}{\leq} \exp$, then $J$ is conjugate to the standard orientation-reversing involution.
\end{conj}

While we can build examples to force the modulus of continuity to be very large, the conjecture says that there is a gap or desert: that is, no examples exist between isometry (no stretching) and exponential stretching in $S^3$. In the introduction we produced examples of wild Lipschitz involutions on $S^n$, $n \geq 4$, so there is no such gap in higher dimensions.

The next conjectures primarily serve to show how little we know.

\begin{conj}[High-dimensional Bing Conjecture]
	The theorem and scholium hold if we replace $S^3$ by $S^n$, $n \geq 4$, and replace the Bing involution by its spun counterparts.
\end{conj}

It is known \cite{edwards06} that spinning the (possibly ramified) Bing decomposition of $S^3$ results in similar, still shrinkable, decompositions of $S^n$, $n \geq 4$. These lead to analogous families of wild involutions on $S^n$. It seems reasonable to hope that the methods of this paper can be generalized to give a similar analysis of the moc for topological conjugates of these involutions, and verify that they too have at least weakly exponential imoc. Clearly, the innermost circle arguments in the paper will need to be replaced with more homological alternatives to extend the proof.

As far as we have been able to determine, moc over a topological conjugacy class, \emph{inherent modulus of continuity}, has not previously been studied, although it is the low regularity cousin of inherent differentiability \cites{harrison75,harrison79,fh88}. In that study, the foundational examples\footnote{There are examples in all dimensions $\geq 2$ of $C^r$-diffeomorphisms, any $r>0$ that cannot be increased to $C^s$ for any $s>r$ by a topological change of coordinates.} exhibit considerable dynamic \emph{complexity}. Involutions, on the other hand, represent the extreme of dynamical simplicity. In an appendix we describe dynamically complex homeomorphisms with arbitrarily large imoc on all manifolds of dimension $\geq 2$.

We return now to the case of 3-dimensional symmetries. The fact that the fixed set of $I$ has codimension 1 ($I$ reverses orientation) is crucial to our proof of large imoc, but we do not know if this feature is essential to the phenomenon, thus leading to the following conjecture.

\begin{conj}[Finite group actions]
	If a finite group $F$ acts on a 3-manifold $M$ with subexponential moc (for each element of $F$), then $F$ is topologically conjugate to a smooth action.
\end{conj}

The simplest example of this conjecture is the ``other Bing involution,'' $I_\pi$, apparently first described by Montgomery and Zippin \cite{alford66}. This involution comes from a $180^\circ$-rotational symmetry of the Bing decomposition $\mathcal{D}$ (which happens to commute with the reflection we have studied), and has a wild circle as a fixed point set. Again, although $\mathcal{D}$ enjoys this rotational symmetry, the symmetry is not preserved by any shrinking homeomorphisms, so the involution becomes wild in the quotient.

If one wished to adapt the methods of this paper to study the imoc of $I_\pi$, a good warmup is to ask how much metric distortion is required to rotate $\R^3$ by an involution $I^\pr$ fixing a \emph{bent} $z$-axis with acute angle $\delta$ as indicated in Figure \ref{fig:I_pi}.

\begin{figure}[ht]
    \centering
    \begin{tikzpicture}[scale=1.6]
        \draw (0,0.5) to[out=225,in=135] (0,-0.25);
        \draw[dashed] (0,0.5) arc (150:-150:2 and 0.75);
        \draw (0,0.5) -- (3,0.1) -- (0,-0.25);
        \draw[->] (0,-0.25) -- (0,-1.7);
        \draw (0,0.5) -- (0,1.4);
        \draw[->] (-0.25,1.7) arc (-180:0:0.25 and 0.2);
        \draw[->] (0,1.6) -- (0,2.1);
        
        \node at (4.1,0.1) {$\beta$};
        \node at (0.9,0.1) {$\delta$};
        \node at (-0.4,0.1) {$\alpha$};
        \node at (0.5,1.7) {$\pi$};
    \end{tikzpicture}
    \caption{$\pi$-rotation $I^\pr$ about $z$-axis with an angle $\delta$ crease.}
    \label{fig:I_pi}
\end{figure}
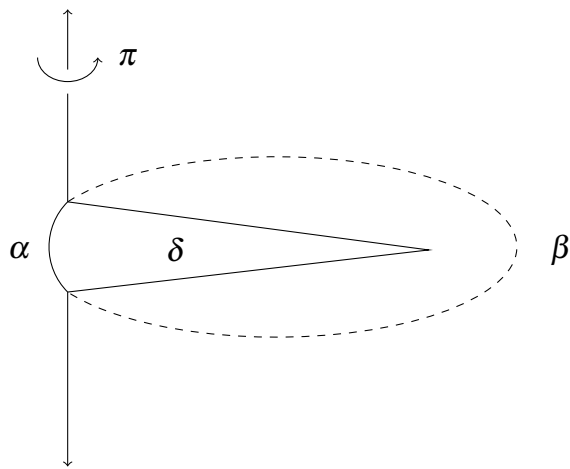

The answer to this exercise is that the topological conjugacy class of $I^\pr$, fixing this bent $z$-axis, contains involutions that are approximately $\delta^{-1/2}$-biLipschitz but not better.

To do this exercise, we assume $\delta$ is small and approximate $2 \tan \frac{\delta}{2} \approx \delta$. One's first thought is that (scaling the ``beak'' to have length one) an arc $\alpha$ of length $\delta$ on the left must be thrown by $I^\pr$ to the right (dotted in Figure \ref{fig:I_pi}) to an arc $\beta \coloneqq I^\pr(\alpha)$ of length $> 2$, implying a Lipschitz constant $O(\delta^{-1})$.

But one can do a bit better. There is implicitly an $S^1$ family of arcs $\alpha_\omega$, $\omega = e^{2 \pi i \te}$, together forming a 2-sphere, all $\alpha_\omega$ having one end point at its north and the other endpoint at its south pole, and so that $I^\pr(\alpha_\omega) = \alpha_{-\omega}$. It is \emph{not} necessary that if $\alpha = \alpha_\omega$ then $\beta = \alpha_{-\omega}$. We can reduce stretching by avoiding this case. Let $f(\te) = \log(\operatorname{length} \alpha_\omega)$ be a function on $[0,1]$ with a single local maximum and single local minimum. Indeed, $f(\te)$ must vary between $\log \delta$ and $\log 2$, let us choose $f$ to (approximately) minimize
\begin{equation}
    \abs{f(\te) - f(\te + \frac{1}{2})}
\end{equation}
over $\te \in [0,1]$, and hence (approximately) minimize stretching

\begin{figure}[ht]
    \centering
    \begin{tikzpicture}[scale=1.6]
        \draw[->] (0,0) -- (0,3.5);
        \draw[->] (0,0) -- (6,0);
        \draw (0,0) -- (0.5,3) -- (5.5,0);
        \draw (0,1.5) -- (2.5,0) -- (3,3) -- (5.5,1.5);
        \draw (5.5,0) -- (5.5,3);
        
        \node at (-0.6,0) {$\log \delta$};
        \node at (-0.6,3) {$\log 2$};
        \node at (0,-0.3) {0};
        \node at (2.5,-0.3) {$\te \ra$};
        \node at (5.5,-0.3) {1};
        \node at (1.8,3.7) {graph $f(\te)$};
        \draw[->] (1.8,3.4) -- (1.5,2.6);
        \node at (4.5,3.7) {graph $f(\te + \frac{1}{2})$};
        \draw[->] (4.5,3.4) -- (4.2,2.6);
    \end{tikzpicture}
    \caption{}\label{fig:min_stretch}
\end{figure}

Figure \ref{fig:min_stretch} suggests (and it is so) that $\max_\te \abs{f(\te) - f(\te+\frac{1}{2})}$ can be made arbitrarily close to $\frac{\abs{\log 2 - \log \delta}}{2}$, but not smaller when $f$ has unique local max and local min. This construction leads to the $\approx \delta^{-1/2}$-biLipschitz estimate for $I^\pr$.

This observation may be thought of as the codimension two analog of our bag lemma (\ref{lm:bag}), where instead of a single 2-dimensional bag, we are led to a $\te$-family, $\te \in [0, \frac{1}{2}]$, of 1-dimensional bags, with a stretching lower bound $\approx \delta^{-\frac{1}{2}}$ holding for some $\te$. We hope it may serve as a starting point for a researcher interested in bounding the imoc of $I_\pi$, and more generally tackling the above conjecture.

\subsection*{Questions}
There is a fixed point free involution of $S^4$ with quotient $\ast\mathrm{RP}^4$, a manifold with nontrivial Kirby-Siebenmann invariant. What is its inherent modulus of continuity? What about involutions on smooth manifolds which are exotic due to more subtle Donaldson/Seiberg-Witten invariants?

This paper plays on the boundary between wild topology and analysis. In dimension 3, the Hauptvermutung holds [Mo54]; that is, all topological manifolds are smoothable. So the first questions involving distinctions between the topological and smooth categories in dimension 3 involve symmetries. This paper demonstrates this analytical difference in the simplest symmetry possible, an involution.  The most compelling venue for the analytical implications of infinite construction and shrinking is four dimensions. There what is known about the boundary between topological and smooth categories comes from a study of the minimal analytic context required to formulate the anti-self-dual Yang-Mills equations \cite{ds89}. If a 4-manifold has a quasi-conformal (QC) structure, most, or perhaps all, of Donaldson theory applies. In particular, a Casson handle cannot be biLipschitz equivalent, or even QC-equivalent, to a standard open handle. However, there is no internal argument for this fact within geometric topology. But perhaps lower bounds on modulus of continuity, or even lower bounds on conformal distortion, of homeomorphisms between smooth manifolds, are not wholly inaccessible by geometric methods.

The authors would like to thank Michael Wang for manuscript preparation, particularly creation of the illustrations.

\bibliography{references}

\section*{Appendix: imoc for infinite order maps}
Given any growth function $g:[1,\infty) \ra [1,\infty)$, it is possible to construct an infinite order map $F: D^n \times S^1 \ra D^n \times S^1$ for each $n \geq 1$ so that imoc$(F) \overset{\infty}{\geq} g$. The function $F(x,t) = (x, t + e^{2 \pi i f(x)})$ for some function $f: D^n \ra [0,1]$ which we now describe.

$D^n$ is the unit $n$-disk. Let $\{a_i\}$ be a rapidly growing sequence of positive integers; how rapidly will depend on the choice of $g$. We next describe an embedding in $D^n$ of infinitely many small, disjoint, round $n$-disks. The embedding is made in infinitely many tranches of $n$-disks. The first tranche (which could be a single disk) covers $\frac{1}{2} \operatorname{vol}(D^n)$. Inductively, the $i+1$th tranche of disks is placed in the complement of the first $i$ tranches, has total volume $\frac{1}{2^{i+1}}$, and the property that the largest radius among $i$th-tranche disks must be smaller than $\frac{1}{a_i}$. So, these disks shrink rapidly in size. Let $G$ be the complement of the interiors of all the embedded disks. $G$ has zero $n$-dimensional Lebesque measure and might look something like a Sierpi\'{n}ski gasket, when $n = 2$. $G$ will be a Cantor set when $n = 1$. We now define $f$ such that $f\big\vert_G \equiv 0$, and on each $i$th-tranche disk $f(\ast) = \frac{1}{i}$, for $\ast$ the center point of that disk. Before changing coordinates, we see points $x \times 1$ and $\ast \times 1$ with $x \times 1 \in G \times S^1 \subset \operatorname{Fix}(F)$, $\operatorname{dist}(x \times 1, \ast \times 1) \leq \frac{1}{a_i}$, and $\operatorname{dist}(F(x \times 1), F(\ast \times 1)) = \frac{1}{i}$. Clearly, $\{a_i\}$ can be chosen so that moc$(F) \overset{\infty}{\geq} g$.

Now consider the effect of changing coordinates on $D^n \times S^1$, that is, replacing $F$ and $F^h$ for any homeomorphism $h: D^n \times S^1 \ra D^n \times S^1$. Let $S_i$ denote the union of the $i$th tranche of disks cross $S^1$, $\operatorname{vol}_{n+1}(S_i) = \frac{1}{2^i}$. Choosing the convention that $D^n$ has $n$-vol $= 1$ and $S^1$ has 1-vol = 1, then clearly $\operatorname{vol}_{n+1}(h(S_i)) < 1$. It follows that one of the $i$th-tranche solid tori $T_i$, i.e.\ some component of $S_i$ has its $n+1$-volume expanded by $h$ less than a factor of $2^i$. Thus, by volume comparison, every point in $h(T_i)$ must be within $2^{\frac{i}{n}} a_i$ of $\de(h T_i)$. On the other hand, by continuity, some point $\ast \times D_0$ must proceed exactly $\frac{1}{i}$th of the way around the $S^1$ coordinate when $F^h$ is applied. Let $z$ be a fixed point with dist$(z, \ast \times \te_0) \leq 2^{\frac{i}{n}} a_i$. Then dist$(F(z) = z, F(\ast \times \te_0)) \geq \frac{1}{i}$. Since we have complete control of $\{a_i\}$, this enables us to force, simultaneously for all $h$, an arbitrarily growing moc for $F^h$, that is an arbitrary imoc for $F$.

Finally, it is worth noting that the Cartesian product structure on $D^n \times S^1$ was merely a convenience. The dynamic $F: D^n \times S^1 \ra D^n \times S^1$ can be implanted by a codimension-0 embedding, $k: D^n \times S^1 \hookrightarrow M$, inside a chart of any manifold $M$ of dimension $\geq 2$ to obtain:

\begin{customthm}{A}
    Let $M^{n+1}$ be a Riemannian manifold (compact or noncompact) of dimension $\geq 2$, then for any growth function $g$, $M$ admits a self-homeomorphism $F_g: M \ra M$ with $\operatorname{imoc}(F) \overset{\infty}{\geq} g$.
\end{customthm}

\begin{proof}[Proof sketch]
    When $M$ is noncompact, e.g.\ $\R^{n+1}$, a new issue arises: a radical expansion $h$, of $M$, can reduce the moc of $F^h$. But as $\epsilon \ra 0$, the effect on the moc approaches composition with at most linear functions.
    
    Setting moc$(F) =: \delta_F^{-1}$ and moc$(F^h) =: \delta_{F^h}^{-1}$, let $c_1 =$ inf$($length of a degree 1 loop in $h \circ k(D^n \times S^1)$ and $c_{n+1} > \operatorname{vol}(h \circ k(D^n \times S^1)))$. Then for sufficiently large $\epsilon^{-1}$,
    \[
    \delta_{F^h}^{-1}(\epsilon^{-1}) > \mathrm{const.}\ 2^{-\frac{\epsilon^{-1}}{n}} c_{n+1}^{\frac{1}{n+1}} \delta_F^{-1}(c_1^{-1}\epsilon^{-1}) \tag{A1}
    \]
    for some const.$(n) > 0$, but this may be compensated by a still faster growing $\{a_i\}$.
\end{proof}

\begin{note}
    Theorem A does not hold for dim$(M) = 1$. Danny Calegari has shown the authors an elementary proof that every homeomorphism $F: S^1 \ra S^1$ is conjugate, at least to a biLipschitz homeomorphism. We do not know what is the maximal regularity that can be achieved for the general homeomorphism $F: S^1 \ra S^1$. It is certainly less than $C^2$ by \cite{denjoy32}.
\end{note}

\begin{note}
    Similar to our remark in the introduction, the $F$ constructed above are inherently non-quasiconformal. The reason is that the points $(\ast \times \te_0)$ located above do not just see one nearby fixed $z$, but another nearby $z^\pr$ in an opposite, or nearly opposite, direction. Thus $F$ bends a straight-angle or near a straight-angle into an order $\frac{1}{a_i}$-sharp point.
\end{note}

\end{document}

%% file: Inserts/Fig1.1.tex
\begin{tikzpicture}
    \node at (0,8.5) {\includegraphics[scale=0.7]{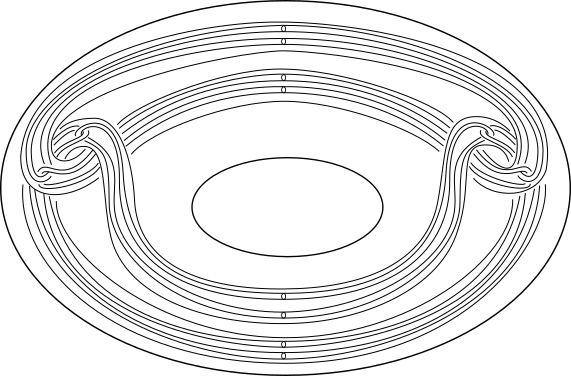}};
    \draw (1,5.1) -- (1,4.9) -- (-1,4.5) -- (-1,12.1) -- (1,12.5) -- (1,11.9);
    \draw[dashed] (1,11.9) -- (1,8.9);
    \draw (1,8.9) -- (1,7.35);
    \draw[dashed] (1,7.35) -- (1,4.9);
    \node at (0.75,12.2) {$S^2$};
    
    \draw (-0.05,11.97) arc (90:270:0.35 and 1.45);
    \draw[dashed] (-0.05,11.97) arc (90:-90:0.35 and 1.45);
    \draw (-0.05,11.55) arc (90:270:0.07 and 0.25);
    \draw[dashed] (-0.05,11.55) arc (90:-90:0.07 and 0.25);
    \draw (-0.05,10.7) arc (90:270:0.07 and 0.3);
    \draw[dashed] (-0.05,10.7) arc (90:-90:0.07 and 0.3);
    \draw (-0.05,7.23) arc (90:270:0.3 and 1.1);
    \draw[dashed] (-0.05,7.23) arc (90:-90:0.3 and 1.1);
    \draw (-0.05,6.64) arc (90:270:0.1 and 0.32);
    \draw[dashed] (-0.05,6.64) arc (90:-90:0.1 and 0.32);
    \draw (-0.05, 5.73) arc (90:270:0.1 and 0.223);
    \draw[dashed] (-0.05, 5.73) arc (90:-90:0.1 and 0.223);
    
    \node at (-0.65,9.4) {$\tld{D}_{\sigma 0}$};
    \draw[->] (-0.65,9.6) to[out=90,in=180] (-0.4,9.8);
    \node at (0.2, 9.7) {\footnotesize{$\tld{D}_{\sigma 01}$}};
    \draw[->] (0.2,9.8) -- (0,10);
    \node at (-0.5, 11.75) {\footnotesize{$\tld{D}_{\sigma 00}$}};
    \draw[->] (-0.4,11.55) -- (-0.2,11.35);
    \node at (-0.65,6.95) {$\tld{D}_{\sigma 1}$};
    \draw[->] (-0.65,6.7) to[out=-90,in=180] (-0.4,6.45);
    \node at (0.45,7.05) {\footnotesize{$\tld{D}_{\sigma 10}$}};
    \draw[->] (0.15,6.87) -- (0,6.67);
    \node at (-0.6,5) {\footnotesize{$\tld{D}_{\sigma 11}$}};
    \draw[->] (-0.5,5.1) -- (-0.17,5.4);
    \node at (-2,9.3) {$\tld{T}_{\sigma 0}$};
    \draw[->] (-1.8,9.4) to[out=0,in=-90] (-1.5,9.7);
    \node at (-2,7.4) {$\tld{T}_{\sigma 1}$};
    \draw[->] (-1.7,7.4) to[out=0,in=90] (-1.5,7);
    \node at (3,10.2) {\footnotesize{$\tld{T}_{\sigma 00}$}};
    \draw[<->] (2.35,11) -- (2.9,10.4) -- (2.15,10.85);
    \node at (2,9.2) {\footnotesize{$\tld{T}_{\sigma 01}$}};
    \draw[<->] (1.9,10.1) -- (2,9.3) -- (1.7,9.9);
    \node at (2,7.5) {\footnotesize{$\tld{T}_{\sigma 10}$}};
    \draw[<->] (2.1,6.82) -- (2,7.3) -- (1.8,6.9);
    \node at (3.4,7.5) {\footnotesize{$\tld{T}_{\sigma 11}$}};
    \draw[<->] (3.7,6.8) -- (3.4,7.3) -- (3.3,6.7);
    \node at (-3.5,11.7) {$\tld{T}_\sigma$};
    \draw[->] (-3.3,11.7) to[out=0,in=90] (-3.1,11.4);
    
    \node at (0,4.2) {(a) Picture of $\tld{T}_{\sigma 0} \cup \tld{T}_{\sigma 1} \subset \tld{T}_\sigma$, $\tld{T}_{\sigma 00} \cup \tld{T}_{\sigma 01} \subset \tld{T}_{\sigma 0}$, and $\tld{T}_{\sigma 10} \cup \tld{T}_{\sigma 11} \subset \tld{T}_{\sigma 1}$. This picture illustrates};
    \node at (0,(3.65) {the nesting pattern in the coordinates of $\tld{T}_\sigma$. In reality $\tld{T}_\sigma$ would have a bent shape like $\tld{T}_{\sigma 0}$.};
    
    \node at (0,0.2) {\includegraphics[scale=0.6]{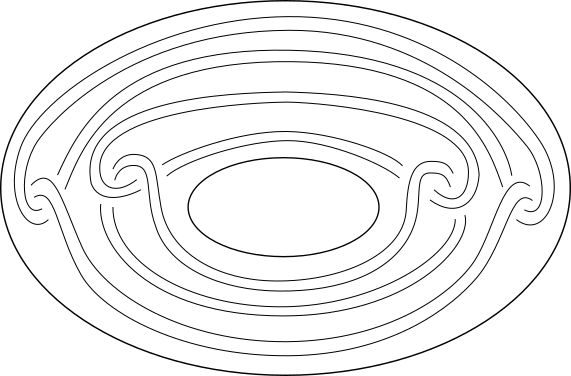}};
    \node at (0,-3.15) {(b) $r$ Bing pairs $= r$-fold ramification};
    \node at (-3.2,0) {$\dots$};
    \node at (3.2,0) {$\dots$};
    \node at (3.13,-0.23) {$r$};
    
    \node at (0,-6.2) {\includegraphics[scale=0.7]{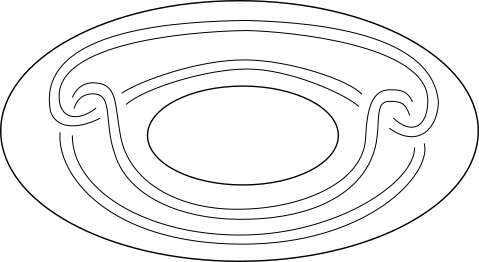}};
    \node at (0,-4.65) {$E_{\sigma 0}$};
    \node at (2.7,-6.9) {$E_{\sigma 1}$};
    \node at (-3.4,-8.1) {$T_\sigma$};
    \node at (-3.5,-7.15) {$T_{\sigma 1}$};
    \draw[->] (-3.6,-6.9) to[out=90,in=200] (-3.35,-6.6);
    \node at (-3.85,-5.75) {$T_{\sigma 0}$};
    \draw[->] (-3.9,-5.95) to[out=-70,in=225] (-3.5,-6);
    
    \path[pattern=north east lines, pattern color = gray] (0,-4.35) .. controls (1.8,-4.35) and (3.5,-4.9) .. (3.5,-5.7) to[out=140,in=50] (2.35,-5.55) to[out=148,in=32] (-2.25,-5.55) to[out=105,in=65] (-3.3,-5.6) .. controls (-3.3,-4.8) and (-1.8,-4.35) .. cycle;
    \path[pattern = north west lines, pattern color = gray] (0,-7.85) .. controls (2,-7.85) and (2.6,-6.8) .. (2.6,-6.1) -- (3.2,-6.4) .. controls (3.2,-7) and (2.2,-8.1) .. (0,-8.1) .. controls (-2.3,-8.1) and (-3.05,-6.9) .. (-3.05,-6.15) -- (-2.45,-6) .. controls (-2.45,-7.4) and (-1.3,-7.85) .. cycle;
    \node at (0,-9.04) {(c) The shaded regions $E_{\sigma0}$ and $E_{\sigma 1}$ are disks in $T_\sigma$ bounded by the longitudes of $T_{\sigma0}$ and $T_{\sigma1}$};
    
    \node at (-5.2,0) {\hspace{0.5em}};
\end{tikzpicture}

%% file: Inserts/Fig1.2.tex
\begin{tikzpicture}
    \node at (0,4.5) {\includegraphics[scale=0.33]{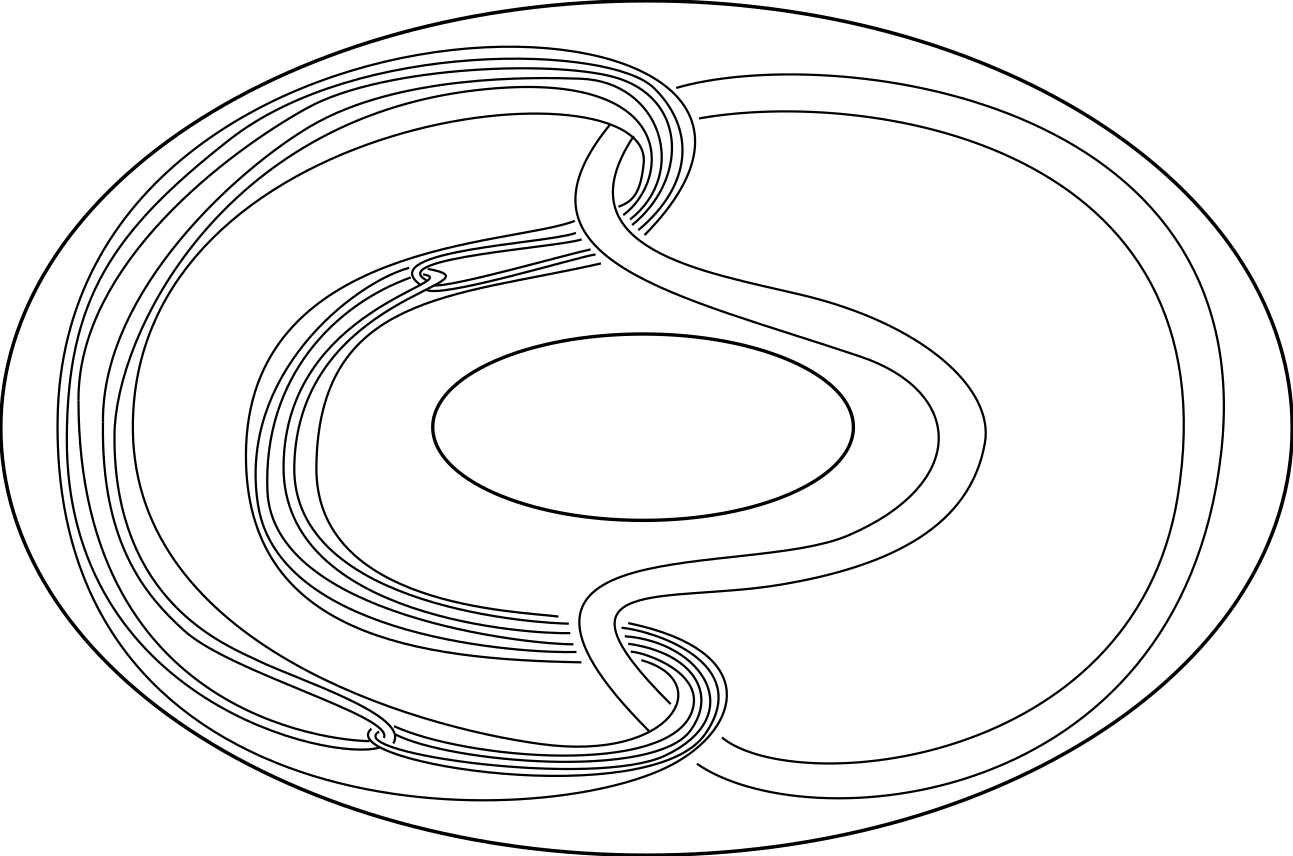}};
    \node at (0,8.5) {$\partial D_0$};
    \draw (0,8.22) arc (90:-90:0.3 and 1.45);
    \draw[dashed] (0,8.22) arc (90:270:0.3 and 1.45);
    \node at (0,3.95) {$\partial D_1$};
    \draw (0,3.7) arc (90:-90:0.3 and 1.46);
    \draw[dashed] (0,3.7) arc (90:270:0.3 and 1.46);
    \node at (0,0.3) {(a) Bing's 1988 shrink};
    
    \node at (-0.7,7.97) {$T_0$};
    \node at (1.2,7.8) {$T_1$};
    \node at (2.2,8.15) {$T$};
    
    \draw (-5.13,4.35) arc (180:0:0.33 and 0.1);
    \draw[dashed] (-5.13,4.35) arc (180:360:0.33 and 0.1);
    \node at (-5.6,4.5) {\footnotesize{$\de D_{00}$}};
    
    \draw (-3.5,4.35) arc (180:0:0.31 and 0.1);
    \draw[dashed] (-3.5,4.35) arc (180:360:0.31 and 0.1);
    \node at (-2.4,4.4) {\footnotesize{$\de D_{01}$}};
    
    \node at (-3.6,7.8) {\footnotesize{$\de D_{000}$}};
    \draw[->] (-3.55,7.6) -- (-3.35,7.23);
    \draw (-3.27,7.18) ellipse (0.035 and 0.07);
    \node at (-2.2,6.7) {\footnotesize{$\de D_{001}$}};
    \draw[->] (-2.4,6.9) -- (-2.6,7.1);
    \draw (-2.72,7.17) ellipse (0.035 and 0.07);
    
    \node at (-1.75,3.5) {\footnotesize{$\de D_{010}$}};
    \draw[->] (-1.95,3.35) -- (-2.15,3.1);
    \draw (-2.2,3.04) ellipse (0.03 and 0.05);
    \node at (-2.05,2.3) {\footnotesize{$\de D_{011}$}};
    \draw[->] (-2.05,2.45) -- (-1.85,2.65);
    \draw (-1.8,2.73) ellipse (0.03 and 0.05);
    
    \node at (-3.5,5.55) {\footnotesize{$T_{01}$}};
    \draw[->] (-3.3,5.5) -- (-3.05,5.25);
    \node at (-3,6.2) {\footnotesize{$T_{00}$}};
    \draw[->] (-3.25,6.3) -- (-3.45,6.55);

    \path[pattern = north east lines, pattern color = lightgray] (-2.18,-2.03) to[out=205,in=70] (-4.18,-3.9) arc (110:296:1.43 and 0.65) to[out=-67,in=160] (-1.75,-6.34) to[out=-20,in=-20] (-1.59,-5.88) to[out=160,in=-70] (-2.65,-4.97) arc (-44:88:1.43 and 0.65) to[out=60,in=-155] (-1.97,-2.45) to[out=20,in=20] cycle;
    \path[fill=white] (-4.03,-4.46) ellipse (0.25 and 0.12);
    \path[fill=white] (-2.95,-4.52) ellipse (0.25 and 0.12);
    
    \path[pattern = north east lines, pattern color = lightgray] (1.94,-3.71) to[out=-10,in=110] (3.16,-4.58) arc (120:283:1.43 and 0.65) to[out=-120,in=22] (2.81,-6.92) to[out=225,in=200] (3,-7.36) to[out=24,in=-115] (4.85,-5.6) arc (-45:97:1.43 and 0.65) to[out=100,in=-10] (2.1,-3.25) to[out=170,in=170] cycle;
    \path[fill=white] (3.5,-5.11) ellipse (0.25 and 0.12);
    \path[fill=white] (4.7,-5.1) ellipse (0.25 and 0.12);
    
    \node at (0,-4.7) {\includegraphics[scale=0.33]{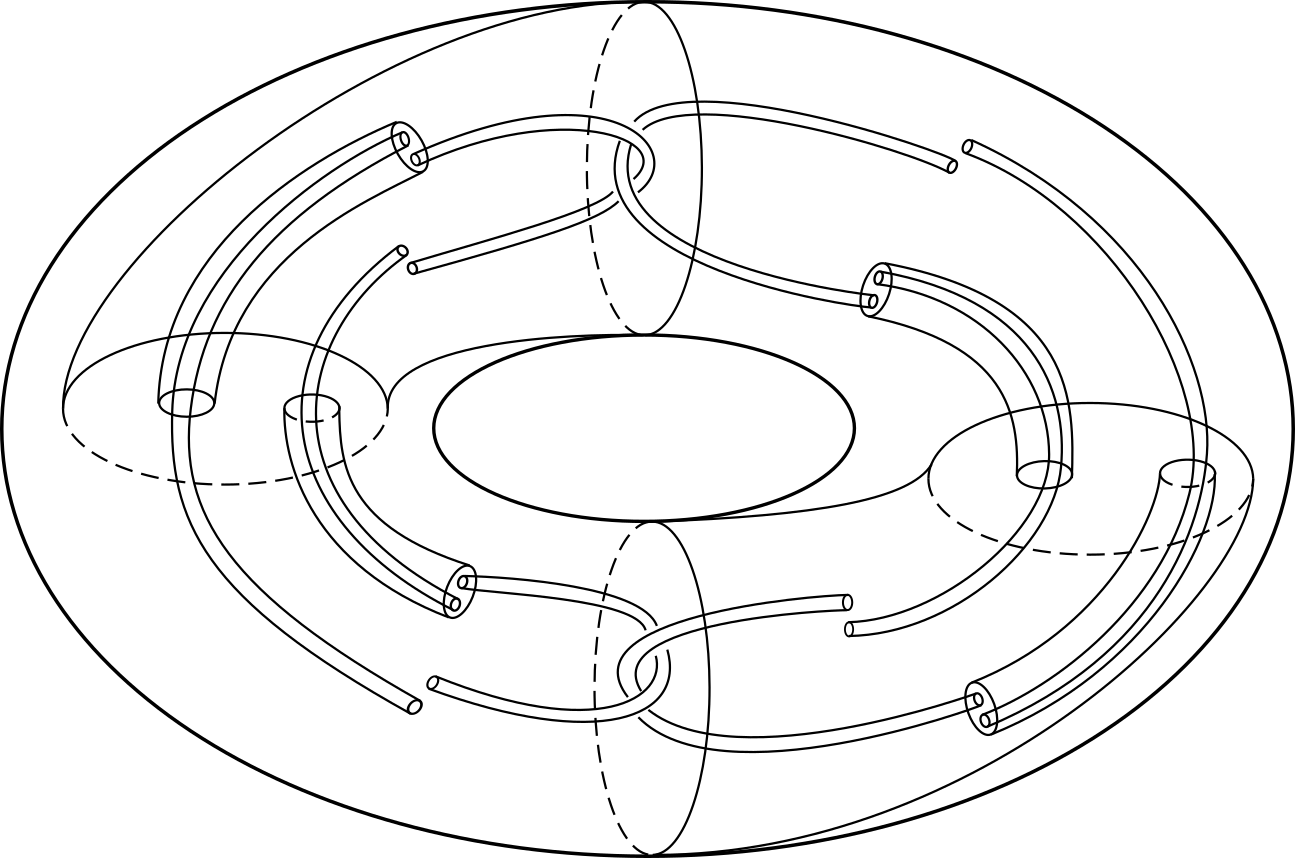}};
    \node at (0,-8.9) {(b) The wild fixed point set of $I$ in Bing's 1988 shrink};
    
    \node at (-4.25,-1.7) {$D_0$};
    \draw[->] (-4,-1.9) -- (-3.5,-2.2);
    \node at (4.5,-7.6) {$D_1$};
    \draw[->] (4.3,-7.4) -- (4,-7.1);
    
    \node at (0,-0.7) {$\partial D_0$};
    \node at (0,-5.2) {$\partial D_1$};
    \node at (-4.5,-4.8) {\footnotesize{$\partial D_{00}$}};
    \draw[->] (-4.5,-4.65) to[out=90,in=180] (-4.3,-4.5);
    \node at (-3.5,-4.85) {\footnotesize{$\partial D_{01}$}};
    \draw[->] (-3.45,-4.7) to[out=90,in=180] (-3.25,-4.5);
    \node at (-1.8,-1.75) {\footnotesize{$\partial D_{000}$}};
    \draw[->] (-1.8,-1.9) -- (-2,-2.1);
    \node at (-1.75,-2.8) {\footnotesize{$\partial D_{001}$}};
    \draw[->] (-1.8,-2.6) -- (-1.95,-2.45);
    \node at (-1.3,-5.6) {\footnotesize{$\partial D_{011}$}};
    \draw[->] (-1.35,-5.75) -- (-1.5,-5.9);
    \node at (-1.85,-6.7) {\footnotesize{$\partial D_{010}$}};
    \draw[->] (-1.85,-6.57) -- (-1.75,-6.4);
\end{tikzpicture}

%% file: Inserts/Fig1.3.tex
\begin{tikzpicture}[scale=1.3]
    \draw (-2,0) arc (180:15:2);
    \draw (-2,0) arc (-180:-15:2);
    \draw (1.93,0.52) arc (148:15:1);
    \draw (1.93,-0.52) arc (-148:-15:1);
    \draw (3.74,0.25) arc (149:15:0.5);
    \draw (3.74,-0.25) arc (-149:-15:0.5);
    \draw (4.65,0.11) arc (150:15:0.25);
    \draw (4.65,-0.11) arc (-150:-15:0.25);
    \node at (-2.7,0) {$F$};
    \draw[->] (-2.5,0) -- (-2.1,0);
    \node at (5.6,0) {$\dots$};
    
    \draw[dashed] (-2,0) arc (180:0:1.95 and 0.4);
    \draw[dashed] (1.93,0) arc (180:0:0.9 and 0.2);
    \draw[dashed] (3.74,0) arc (180:0:0.45 and 0.1);
    \draw[dashed] (4.65,0) arc (180:0:0.23 and 0.05);
    
    \draw[decorate, decoration=snake] (-2,0) arc (-180:0:1.95 and 0.4);
    \draw[decorate, decoration=snake] (1.93,0) arc (-180:0:0.9 and 0.2);
    \draw[decorate, decoration=snake] (3.74,0) to[out=-40,in=220] (4.65,0);
    \draw[decorate, decoration=snake] (4.65,0) to[out=-20,in=200] (5.1,0);
    
    \draw (1.93,0.52) arc (90:-90:0.15 and 0.52);
    \draw[dashed] (1.93,0.52) arc (90:270:0.15 and 0.52);
    \draw (3.74,0.25) arc (90:-90:0.08 and 0.25);
    \draw[dashed] (3.74,0.25) arc (90:270:0.08 and 0.25);
    \draw (4.65,0.11) arc (90:-90:0.04 and 0.11);
    \draw[dashed] (4.65,0.11) arc (90:270:0.04 and 0.11);
    
    \node at (0,1) {$M$};
    \node at (0,-1) {$\lbar{M}$};
    \node at (2.8,0.5) {$\frac{1}{2}M$};
    \node at (2.8,-0.7) {$\frac{1}{2}\lbar{M}$};

    \draw[decorate, decoration={brace,amplitude=10pt}] (6,2) -- (6,-2) node[midway, xshift=1.75em] {$X$};
    
    \draw[thick] (1.93,0) ellipse (0.6 and 0.7);
    
    \draw[thick] (3,3) circle (1.5);
    \draw (1.88,0.7) to[out=80,in=-75] (1.75,2.18);
    \draw (2.1,0.72) to[out=50,in=200] (3.3,1.53);
    \draw (3,3.7) arc (90:-90:0.2 and 0.7);
    \draw[dashed] (3,3.7) arc (90:270:0.2 and 0.7);
    \draw (3,3.7) arc (140:100:1.5);
    \draw (3,3.7) arc (30:48:2.5);
    \draw (3,2.3) arc (-140:-100:1.5);
    \draw (3,2.3) arc (-30:-48:2.5);
    \draw[decorate, decoration=snake] (2.65,3) -- (4.1,3);
    \draw[fill=black] (3.2,2.93) circle (0.2ex);
    \draw (2.4,4.38) arc (90:-90:0.25 and 1.38);
    \draw[dashed] (2.4,4.38) arc (90:270:0.25 and 1.38);
    \draw (3.9,4.19) arc (90:-90:0.2 and 1.19);
    \draw[dashed] (3.9,4.19) arc (90:270:0.2 and 1.19);
    
    \node at (5.55,3.3) {Blow-up of};
    \node at (5.8,3) {connected sum};
    \node at (5.1,2.65) {area};
\end{tikzpicture}

%% file: Inserts/Fig1.7.tex
\begin{tikzpicture}
	\node at (0,8.5) {\includegraphics[scale=0.7]{Figures/Fig1.1a.png}};
    \draw (1,5.1) -- (1,4.9) -- (-1,4.5) -- (-1,12.3) -- (1,12.7) -- (1,11.9);
    \draw[dashed] (1,11.9) -- (1,8.9);
    \draw (1,8.9) -- (1,7.35);
    \draw[dashed] (1,7.35) -- (1,4.9);
    \node at (0.7,12.3) {$S^2$};

    \draw (-0.05,11.97) arc (90:270:0.35 and 1.45);
    \draw[dashed] (-0.05,11.97) arc (90:-90:0.35 and 1.45);
    \draw (-0.05,11.55) arc (90:270:0.07 and 0.25);
    \draw[dashed] (-0.05,11.55) arc (90:-90:0.07 and 0.25);
    \draw (-0.05,10.7) arc (90:270:0.07 and 0.3);
    \draw[dashed] (-0.05,10.7) arc (90:-90:0.07 and 0.3);
    \draw (-0.05,7.23) arc (90:270:0.3 and 1.1);
    \draw[dashed] (-0.05,7.23) arc (90:-90:0.3 and 1.1);
    \draw (-0.05,6.64) arc (90:270:0.1 and 0.32);
    \draw[dashed] (-0.05,6.64) arc (90:-90:0.1 and 0.32);
    \draw (-0.05, 5.73) arc (90:270:0.1 and 0.223);
    \draw[dashed] (-0.05, 5.73) arc (90:-90:0.1 and 0.223);

    \draw[<->] (0.1,11.1) -- (0.65,8.4) -- (0.1,10.2);
    \draw[<->] (0.1,5.6) -- (0.65,7.8) -- (0.1,6.6);
    \node at (0.65,8.05) {$\tld{\mathcal{T}}_2$};
    \node at (3.3,11.6) {$\tld{\mathcal{T}}_0$};
    \draw[->] (3.6,11.6) to[out=0,in=60] (3.6,11.15);
    \node at (5.7,8.7) {$\tld{\mathcal{T}}_1$};
    \draw[<->] (4.9,9) -- (5.45,8.7) -- (4.9,8.4);
    \node at (-5.7,8.7) {$\hphantom{M}$};
\end{tikzpicture}

%% file: Inserts/Fig1.8.tex
\begin{tikzpicture}
    \node at (-1.5,3.5) {$\tld{T}$};
    \node at (0,4.7) {$\tld{T}_0$};
    \node at (0,2.3) {$\tld{T}_1$};
    \node at (1.5,5.3) {$\tld{T}_{00}$};
    \node at (1.5,4.1) {$\tld{T}_{01}$};
    \node at (1.5,2.9) {$\tld{T}_{10}$};
    \node at (1.5,1.7) {$\tld{T}_{11}$};
    \node[rotate=45] at (-0.75,4.1) {$\longhookleftarrow$};
    \node[rotate=-45] at (-0.75,2.9) {$\longhookleftarrow$};
    \node[rotate=30] at (0.75,5) {$\hookleftarrow$};
    \node[rotate=-30] at (0.75,4.4) {$\hookleftarrow$};
    \node[rotate=30] at (0.75,2.6) {$\hookleftarrow$};
    \node[rotate=-30] at (0.75,2) {$\hookleftarrow$};
    \node at (2.5,3.5) {$\dots$};
    
    \node at (0,-2) {\includegraphics[scale=0.75]{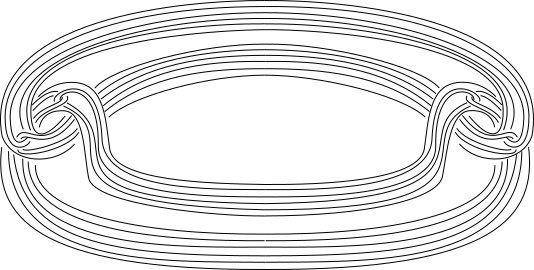}};
    \draw[line width = 0.75pt] (-2.5,1) -- (2.5,1) to[out=0,in=90] (6.3,-2) to[out=-90,in=0] (2.5,-5) -- (-2.5,-5) to[out=180,in=-90] (-6.3,-2) to[out=90,in=180] cycle;
    \draw[line width = 0.75pt] (-2,-1.5) -- (2,-1.5) arc (90:-90:0.5) -- (-2,-2.5) arc (-90:-270:0.5);
    \node at (-5.1,0.6) {$\tld{T}$};
    \node at (-5.55,-1.3) {$\tld{T}_0$};
    \node at (-5.55,-3) {$\tld{T}_1$};
    \node at (-4,0.45) {$\tld{T}_{00}$};
    \draw[->] (-3.75,0.25) -- (-3.6,0.05);
    \node at (-4.7,0) {$\tld{T}_{01}$};
    \draw[->] (-4.45,-0.1) -- (-4.1,-0.47);
    \node at (-3.9,-4.5) {$\tld{T}_{11}$};
    \draw[->] (-3.8,-4.4) -- (-3.7,-4.2);
    \node at (-4.8,-4.05) {$\tld{T}_{10}$};
    \draw[->] (-4.7,-3.95) -- (-4.4,-3.5);
\end{tikzpicture}

%% file: Inserts/Fig1.9.tex
\begin{tikzpicture}[scale=1.143]
    \node at (0,0) {\includegraphics[scale=0.8]{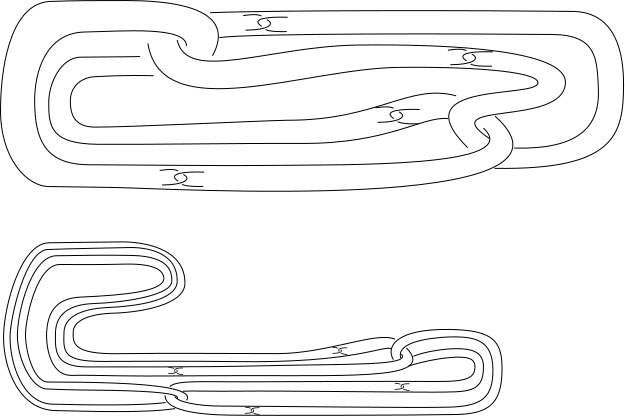}};
    
    \draw (-3.15,4) -- (-3.15,0);
    \draw (-3.15,-0.5) -- (-3.15,-4);
    \draw (-1.7,4) -- (-1.7,0);
    \draw (-1.7,-0.5) -- (-1.7,-4);
    \draw (-0.35,4) -- (-0.35,0);
    \draw (-0.35,-0.5) -- (-0.35,-4);
    \draw (1,4) -- (1,0);
    \draw (1,-0.5) -- (1,-4);
    \draw (2.35,4) -- (2.35,0);
    \draw (2.35,-0.5) -- (2.35,-4);
    \draw (3.75,4) -- (3.75,0);
    \draw (3.75,-0.5) -- (3.75,-4);
    
    \node at (5,-2.3) {etc.};
    \node at (5.6,3.6) {$T_1$};
    \node at (-5.75,3.5) {$T_0$};
    \node at (-2.6,0) {\footnotesize{$T_{00}$}};
    \draw[->] (-2.6,0.15) -- (-2.55,0.42);
    \node at (-2.05,0) {\footnotesize{$T_{01}$}};
    \draw[->] (-2.15,0.15) -- (-2.2,0.39);
    \node at (1.3,2.3) {\footnotesize{$T_{00}$}};
    \draw[->] (1.3,2.15) -- (1.4,1.95);
    \node at (1.9,2.3) {\footnotesize{$T_{01}$}};
    \draw[->] (1.9,2.15) -- (1.8,1.9);
    \node at (-3,0.57) {$\dots$};
    \node at (-1.8,0.55) {$\dots$};
    \node at (-1.4,3.43) {$\dots$};
    \node at (-0.3,3.4) {$\dots$};
    \node at (1.1,1.7) {$\dots$};
    \node at (2.1,1.7) {$\dots$};
    \node at (2.4,2.8) {$\dots$};
    \node at (3.45,2.75) {$\dots$};
    
    \node at (3.1,-3.85) {$T_{01}$};
    \node at (-5.8,-1.1) {$T_{00}$};
    \node at (0.05,-2.25) {\footnotesize{$T_{000}$}};
    \draw[->] (0.1,-2.4) -- (0.35,-2.55);
    \node at (0.7,-2.15) {\footnotesize{$T_{001}$}};
    \draw[->] (0.7,-2.3)-- (0.65,-2.55);
    \node at (-1.3,-4.2) {\footnotesize{$T_{011}$}};
    \draw[->] (-1.3,-4.1) -- (-1.2,-3.87);
    \node at (-0.6,-4.2) {\footnotesize{$T_{010}$}};
    \draw[->] (-0.8,-4) -- (-0.95,-3.8);
    \node at (-2.8,-3.01) {\footnotesize{$\dots$}};
    \node at (-2.2,-3.01) {\footnotesize{$\dots$}};
    \node at (-1.4,-3.74) {\footnotesize{$\dots$}};
    \node at (-0.8,-3.75) {\footnotesize{$\dots$}};
    \node at (0.2,-2.68) {\footnotesize{$\dots$}};
    \node at (0.8,-2.65) {\footnotesize{$\dots$}};
    \node at (1.35,-3.31) {\footnotesize{$\dots$}};
    \node at (2,-3.31) {\footnotesize{$\dots$}};
\end{tikzpicture}

%% file: Inserts/Fig2.1.tex
\begin{tikzpicture}[scale=1.35]
    \node at (0,0) {$\tld{T}$};
    \node at (1,0.85) {$\tld{T}_0$};
    \node at (1,-0.85) {$\tld{T}_1$};
    \node at (2,1.3) {$\tld{T}_{00}$};
    \node at (2,0.4) {$\tld{T}_{01}$};
    \node at (2,-0.4) {$\tld{T}_{10}$};
    \node at (2,-1.3) {$\tld{T}_{11}$};
    \node at (3,1.55) {$\tld{T}_{000}$};
    \node at (3,1.05) {$\tld{T}_{001}$};
    \node at (3,0.6) {$\tld{T}_{010}$};
    \node at (3,0.2) {$\tld{T}_{011}$};
    \node at (3,-0.2) {$\tld{T}_{100}$};
    \node at (3,-0.6) {$\tld{T}_{101}$};
    \node at (3,-1.05) {$\tld{T}_{110}$};
    \node at (3,-1.55) {$\tld{T}_{111}$};
    
    \node[rotate=45] at (0.5,0.43) {$\longhookleftarrow$};
    \node[rotate=-45] at (0.5,-0.43) {$\longhookleftarrow$};
    \node[rotate=30] at (1.45,1.1) {$\hookleftarrow$};
    \node[rotate=-30] at (1.45,0.45) {$\hookleftarrow$};
    \node[rotate=30] at (1.45,-0.55) {$\hookleftarrow$};
    \node[rotate=-30] at (1.45,-1.2) {$\hookleftarrow$};
    \node[rotate=12] at (2.45,1.4) {$\hookleftarrow$};
    \node[rotate=-12] at (2.45,1.05) {$\hookleftarrow$};
    \node[rotate=12] at (2.45,0.5) {$\hookleftarrow$};
    \node[rotate=-12] at (2.45,0.2) {$\hookleftarrow$};
    \node[rotate=12] at (2.45,-0.3) {$\hookleftarrow$};
    \node[rotate=-12] at (2.45,-0.6) {$\hookleftarrow$};
    \node[rotate=12] at (2.45,-1.15) {$\hookleftarrow$};
    \node[rotate=-12] at (2.45,-1.5) {$\hookleftarrow$};
    
    \node[rotate=90] at (0,-1.75) {$\longhookrightarrow$};
    \node[rotate=90] at (1,-1.85) {$\hookrightarrow$};
    \node[rotate=90] at (2,-1.85) {$\hookrightarrow$};
    
    \node at (0,-3.15) {$\tld{D}_0$};
    \node at (1,-2.65) {$\tld{D}_{00}$};
    \node at (1,-3.65) {$\tld{D}_{01}$};
    \node at (2,-2.4) {$\tld{D}_{000}$};
    \node at (2,-2.9) {$\tld{D}_{001}$};
    \node at (2,-3.4) {$\tld{D}_{010}$};
    \node at (2,-3.9) {$\tld{D}_{011}$};
    
    \node[rotate=55] at (-0.45,-3.6) {$\longhookleftarrow$};
    \node[rotate=-55] at (-0.5,-4.85) {$\longhookleftarrow$};
    \node[rotate=30] at (0.45,-2.85) {$\hookleftarrow$};
    \node[rotate=-30] at (0.45,-3.5) {$\hookleftarrow$};
    \node[rotate=30] at (0.45,-4.8) {$\hookleftarrow$};
    \node[rotate=-30] at (0.45,-5.5) {$\hookleftarrow$};
    
    \node at (0,-5.15) {$\tld{D}_1$};
    \node at (1,-4.65) {$\tld{D}_{10}$};
    \node at (1,-5.65) {$\tld{D}_{11}$};
    \node at (2,-4.4) {$\tld{D}_{100}$};
    \node at (2,-4.9) {$\tld{D}_{101}$};
    \node at (2,-5.4) {$\tld{D}_{110}$};
    \node at (2,-5.9) {$\tld{D}_{111}$};
    \node at (-1.6,-4.15) {Fix set $= \tld{S}$};
    \node at (-1.6,-4.5) {\footnotesize{($\lVert \vec{x} \rVert = 1$ and $x_4 = 0$)}};
    
    \node[rotate=12] at (1.4,-2.5) {$\hookleftarrow$};
    \node[rotate=-12] at (1.4,-2.95) {$\hookleftarrow$};
    \node[rotate=12] at (1.4,-3.5) {$\hookleftarrow$};
    \node[rotate=-12] at (1.4,-3.95) {$\hookleftarrow$};
    \node[rotate=12] at (1.4,-4.5) {$\hookleftarrow$};
    \node[rotate=-12] at (1.4,-4.95) {$\hookleftarrow$};
    \node[rotate=12] at (1.4,-5.5) {$\hookleftarrow$};
    \node[rotate=-12] at (1.4,-5.95) {$\hookleftarrow$};
\end{tikzpicture}

%% file: Inserts/Fig2.2.tex
\begin{tikzpicture}[scale=1.55]
    \draw (0,1) arc (90:-90:3 and 0.75);
    \draw (0,1.25) arc (90:-90:3.5 and 1);
    \draw (0,1.5) arc (90:-90:3.75 and 1.25);
    \draw (0,1.75) arc (90:-90:4.25 and 1.5);
    \draw (0,2) arc (90:-90:4.5 and 1.75);
    \draw (0,2.25) arc (90:-90:5 and 2);
    
    \draw (1.7,2.13) arc (90:-90:0.15 and 0.63);
    \draw[dashed] (1.7,2.13) arc (90:270:0.15 and 0.63);
    
    \draw (3.2,2.3) -- (3.2,-1.8);
    \node at (3.2,2.5) {$P$};
    \draw (3.2,1.8) arc (85:-85:0.2 and 1.55);
    \draw[dashed] (3.2,1.8) arc (95:265:0.2 and 1.55);
    \draw (3.2,1.36) ellipse (0.05 and 0.12);
    \draw (3.2,0.78) ellipse (0.05 and 0.12);
    \draw (3.2,-0.28) ellipse (0.05 and 0.12);
    \draw (3.2,-0.85) ellipse (0.05 and 0.12);
    
    \node at (5.7,2.3) {Trivial curve on $\partial T_\sigma$, so};
    \draw[->] (4.3,2.2) -- (3.35,1.55);
    \node at (5.79,1.9) {these are not dwh of $T_{\sigma 0}$.};
    \draw[<->] (3.28,1.3) -- (4.5,1.7) -- (3.28,0.72);
    \draw[<->] (3.28,-0.2) -- (4.5,1.7) -- (3.28,-0.8);
    
    \node at (4.75,0.2) {$T_{\sigma 0}$};
    \draw[->] (4.75,0.35) to[out=90,in=0] (4.5,0.5);
    \node at (5.2,-0.6) {$T_\sigma$};
    \draw[->] (5,-0.55) -- (4.8,-0.4);
    
    \node at (-0.9,0) {$\hphantom{M}$};
\end{tikzpicture}

%% file: Inserts/Fig2.3.tex
\begin{tikzpicture}
    \node at (0.4,-0.1) {\includegraphics[scale=0.57]{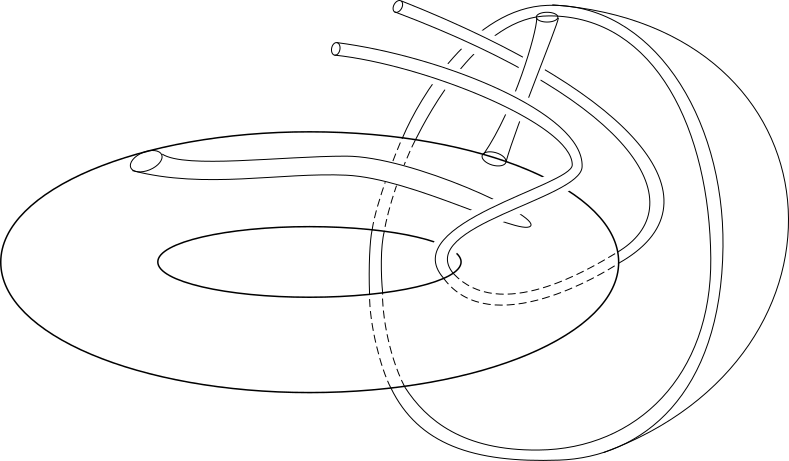}};
    \node at (-3.5,1.4) {feeler};
    \node at (-5,0.7) {$T_\sigma$};
    
    \draw (-3,-6) circle (1.5);
    \draw[pattern = north east lines, pattern color = lightgray] (-2.7,-5.5) circle (0.5);
    \draw[pattern = north east lines, pattern color = lightgray] (-3.3,-6.5) circle (0.25);
    \draw[fill=white] (-2.7,-5.5) circle (0.15);
    \node at (-2.45,-5.6) {$t$};
    \node at (-3.2,-5) {$m$};
    \node at (-3.5,-6.1) {$m$};
    \draw[->] (-2.5,-4.4) -- (-1.8,-2.7);
    
    \draw[pattern = north east lines, pattern color = lightgray] (3,-6) circle (1.5);
    \draw[fill=white] (2.9,-4.77) arc (95:445:1.25) -- (3.1,-5.2) arc  (435:105:0.4) -- cycle;
    \draw[pattern = north east lines, pattern color = lightgray] (3,-6.7) circle (0.25);
    \draw[fill=white] (3,-5.6) circle (0.1);
    \node at (3.8,-4.5) {$t$};
    \node at (3.6,-5.5) {$m$};
    \node at (3.4,-6.5) {$m$};
    \node at (1.5,-4.5) {$\Delta$};
    \draw[->] (1.7,-4.7) -- (1.9,-4.9);
    \draw[->] (2.25,-4.5) -- (2.5,-3.7);
\end{tikzpicture}

%% file: Inserts/Fig2.4.tex
\begin{tikzpicture}
    \node at (0,0) {\includegraphics[scale=0.5]{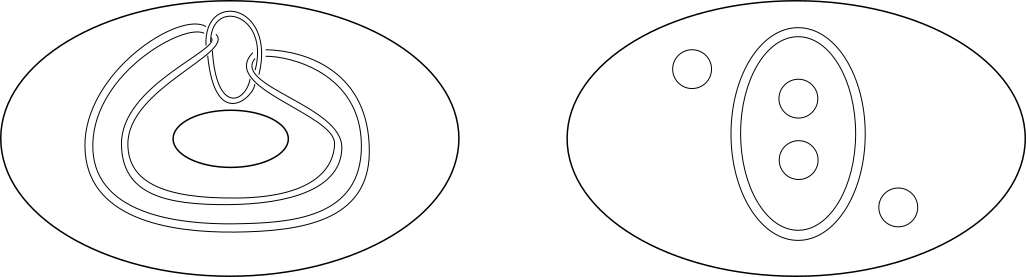}};
    \node at (-6,1.5) {$T$};
    \node at (-5.8,0.5) {$T_0$};
    \node at (-3.68,1.2) {$T_1$};
    \draw[->] (-3.68,1) -- (-3.8,0.7);
    \draw (-3.68,1.82) arc (90:-90:0.45 and 0.72);
    \draw[dashed] (-3.68,1.82) arc (90:270:0.47 and 0.72);
    \node at (-2.9,1.5) {$D$};
    \draw[->] (-3.05,1.5) -- (-3.27,1.5);
    \draw (-4.3,1.8) -- (-4.3,1.9) -- (-3.15,2.1) -- (-3.15,0.2) -- (-4.3,0) -- (-4.3,0.25);
    \draw[dashed] (-4.3,0.25) -- (-4.3,1.8);
    \node at (-3.75,-2.3) {(a)};
    
    \node at (3.8,-0.3) {$T_0$};
    \node at (3.8,0.5) {$T_0$};
    \node at (5.1,-0.93) {$T_0$};
    \node at (2.39,0.9) {$T_0$};
    \node at (4.8,1.25) {$T_1$};
    \node at (5.9,1.6) {$\Delta_m$};
    \draw[->] (4.6,1.1) -- (4.4,0.9);
    \node at (3.75,-2.3) {(b)};
\end{tikzpicture}

%% file: Inserts/Fig2.5.tex
\begin{tikzpicture}
    \node at (-4,0) {\includegraphics[scale=0.25]{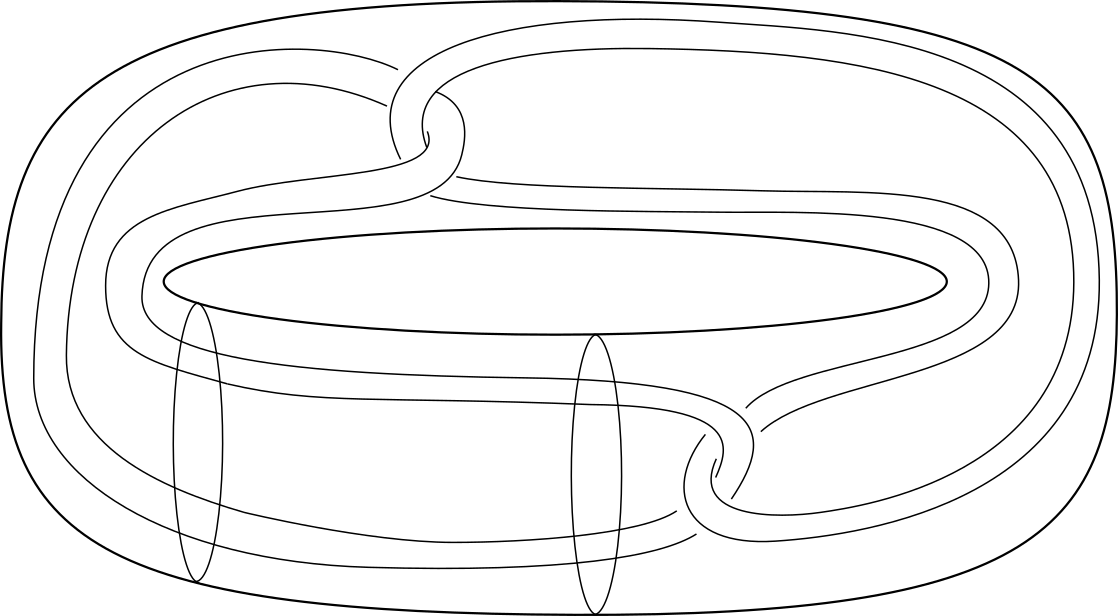}};
    \draw (-6.4,1.8) -- (-6.4,-1.8);
    \draw (-3.75,2.03) -- (-3.75,-2.02);
    \draw (-1.7,1.85) -- (-1.7,-1.85);
    \node at (-3.75,-2.25) {\footnotesize{$P_2$}};
    \node at (-6.4,-2.1) {\footnotesize{$P_1$}};
    \node at (-1.7,-2.1) {\footnotesize{$P_3$}};
    \node at (-4.2,2.2) {\footnotesize{$T_{\sigma}$}};
    \node at (-2.5,1.5) {\footnotesize{$T_{\sigma 1}$}};
    \node at (-5.8,1.2) {\footnotesize{$T_{\sigma 0}$}};
    \node at (-4,-2.7) {(a)};
    
    \node at (4,0) {\includegraphics[scale=0.35]{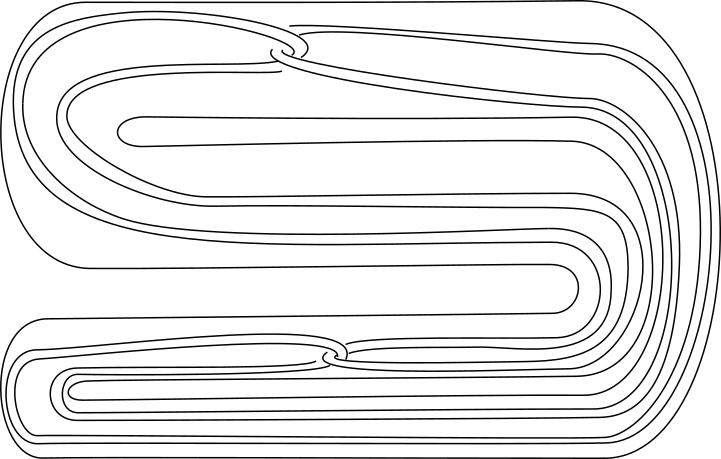}};
    \draw (2.5,2.3) -- (2.5,-2.3);
    \node at (2.5,-2.5) {\footnotesize{$P_1$}};
    \draw (5.5,2.3) -- (5.5,-2.3);
    \node at (5.5,-2.5) {\footnotesize{$P_2$}};
    \node at (4,2.3) {\footnotesize{$T_\sigma$}};
    \node at (1.8,1.7) {\footnotesize{$T_{\sigma 0}$}};
    \node at (4.9,1.55) {\footnotesize{$T_{\sigma 1}$}};
    \node at (4,-2.7) {(b)};
\end{tikzpicture}

%% file: Inserts/Fig2.6.tex
\begin{tikzpicture}
    \node at (0,0) {\includegraphics[scale=0.6]{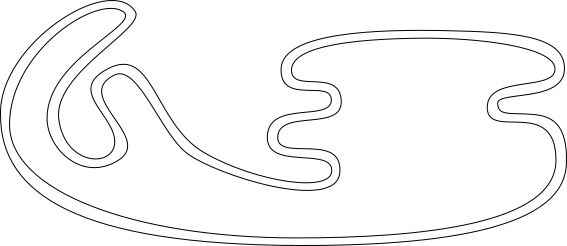}};
    
    \draw (0.5,2.5) -- (0.5,-2.5);
    \draw (-2.75,2.5) -- (-2.75,-2.5);
    \draw (3.7,2.5) -- (3.7,-2.5);
    \node at (0.5,-2.8) {$P_2$};
    \node at (-2.75,-2.8) {$P_1$};
    \node at (3.7,-2.8) {$P_3$};
\end{tikzpicture}

%% file: Inserts/Fig2.7.tex
\begin{tikzpicture}
    \node at (-0.3,0) {\includegraphics[scale=0.35]{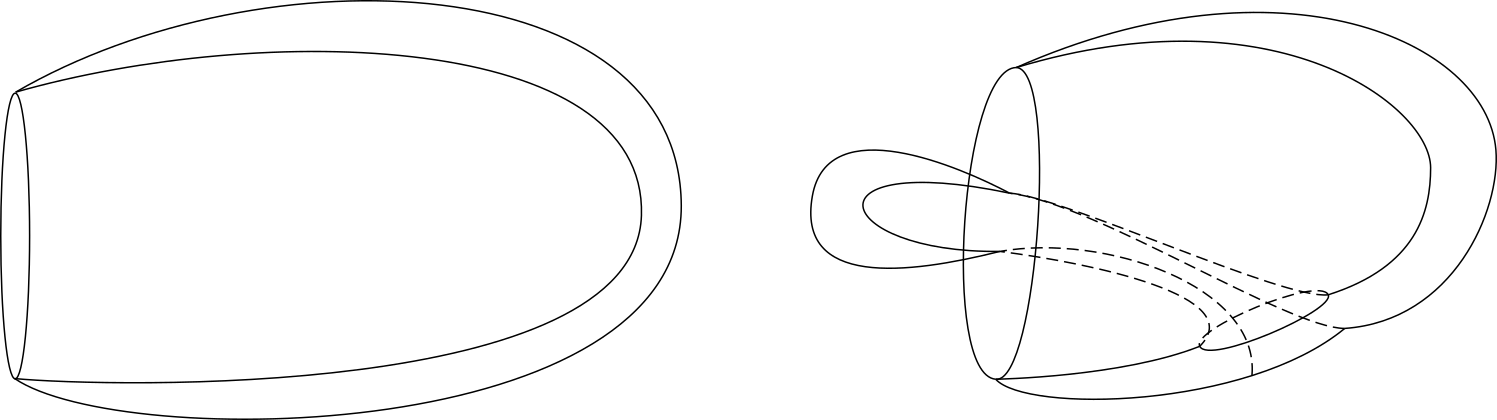}};
    
    \node at (-4.2,-2.7) {(a)};
    \draw[line width = 0.5ex, white] (-7.4,1.7) -- (-7.4,-1.8) -- (-6.6,-2) -- (-6.6,1.5) -- cycle;
    \draw (-7.4,1.7) -- (-7.4,-1.8) -- (-6.6,-2) -- (-6.6,1.5) -- cycle;
    \node at (-7.1,1.3) {$J$};
    \node at (-7.8,0) {$\Delta_J$};
    \node at (-7,-2.3) {$P$};
    \draw[->] (-7.5,0) -- (-7.05,0);
    \draw (-5.5,1.7) arc (90:-90:0.3 and 1.8);
    \draw[dashed] (-5.5,1.7) arc (90:270:0.3 and 1.8);
    \draw (-4.8,1.45) arc (90:-90:0.25 and 1.51);
    \draw[dashed] (-4.8,1.45) arc (90:270:0.25 and 1.51);
    
    \draw[white, line width = 0.5ex] (-3,2.2) -- (-3,-2.2);
    \draw (-3.8,1.93) -- (-3.8,2.4) -- (-3,2.2) -- (-3,-2.2) -- (-3.8,-2) -- (-3.8,-1.87);
    \draw (-3.8,1.45) -- (-3.8,-1.48);
    \node at (-3.4,-2.35) {$P^\pr$};
    \draw (-3.4,1.93) arc (90:-90:0.28 and 1.87);
    \draw[dashed] (-3.4,1.93) arc (90:270:0.28 and 1.87);
    \draw (-3.4,1.43) arc (90:-90:0.17 and 1.42);
    \draw[dashed] (-3.4,1.43) arc (90:270:0.17 and 1.42);
    \node at (-4.4,2.2) {$D_J$};
    \draw[->] (-4.4,2) -- (-4.3,1.6);
    \node at (-1.3,1.9) {$I(\Delta_J)$};
    \draw[->] (-1.3,1.7) -- (-1.4,1.4);

    \node at (3.8,-2.7) {(b)};
    \draw (1,0.54) arc (90:-90:0.12 and 0.54);
    \draw[dashed] (1,0.54) arc (90:270:0.12 and 0.54);
    \draw (1.33,0.46) arc (90:-90:0.1 and 0.48);
    \draw[dashed] (1.33,0.46) arc (90:270:0.1 and 0.48);
    \draw[white, line width = 0.5ex] (2.7,-2.2) -- (2.7,1.7);
    \draw (1.5,-0.6) -- (1.5,-1.9) -- (2.7,-2.2) -- (2.7,1.7) -- (1.5,2) -- (1.5,0.5);
    \node at (2.1,-2.3) {$P$};
    \draw (2.1,0.16) arc (90:-90:0.08 and 0.27);
    \draw[dashed] (2.1,0.16) arc (90:270:0.08 and 0.27);
    
    \draw[rotate=-40] (2.8,1.93) arc (90:-90:0.04 and 0.09);
    \draw[dashed, rotate=-40] (2.8,1.93) arc (90:270:0.04 and 0.09);
    \draw[rotate=-40] (3.1,2.1) arc (90:-90:0.06 and 0.2);
    \draw[dashed, rotate=-40] (3.1,2.1) arc (90:270:0.06 and 0.2);
    \draw[white, line width = 0.5ex] (5.9,-2.2) -- (5.9,1.9);
    \draw (5.1,-1.2) -- (5.1,-2) -- (5.9,-2.2) -- (5.9,1.9) -- (5.1,2.1) -- (5.1,1.77);
    \draw (5.1,1.3) -- (5.1,-0.77);
    \node at (5.5,-2.35) {$P^\pr$};
    \draw (5.5,1.66) arc (90:-90:0.22 and 1.35);
    \draw[dashed] (5.5,1.66) arc (90:270:0.22 and 1.35);
    \draw (5.5,1.1) arc (90:-90:0.12 and 0.84);
    \draw[dashed] (5.5,1.1) arc (90:270:0.12 and 0.84);
    
    \node at (6.9,1.3) {$I(\Delta_J)$};
    \node at (0.7,-1.1) {$I(\Delta_J)$};
    \draw[->] (0.8,-0.9) -- (0.9,-0.63);
    \node at (2.13,0.7) {$\Delta_J$};
    \node at (0.7,0.9) {$D_J$};
    \draw[->] (0.7,0.7) -- (0.9,0.25);
    \node at (3,1) {$J$};
    \draw[->] (2.85,1) -- (2.5,0.8);
    \node at (4.5,1) {$D_J$};
    \draw[->] (4.6,1.1) -- (4.8,1.3);
\end{tikzpicture}

%% file: Inserts/Fig2.9.tex
\begin{tikzpicture}
    \node at (0,-2.8) {\includegraphics[scale=0.6]{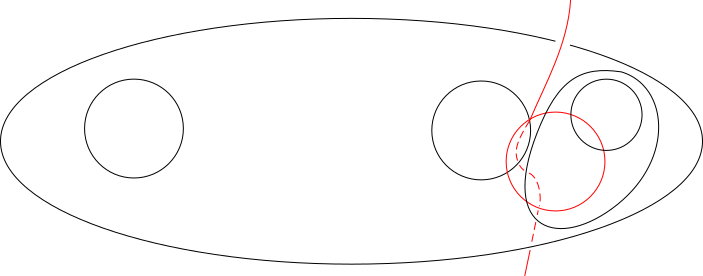}};
    
    \node at (-1.75,-0.6) {$E$ or $\Delta^+$};
    \draw[->] (-1,-0.6) to[out=0,in=90] (-0.8,-0.85);
    \node at (-0.5,-1.6) {$\de T_0$};
    \draw[->] (-0.9,-1.75) to[out=210,in=15] (-2.6,-2.3);
    \draw[->] (-0.2,-1.5) to[out=30,in=130] (2,-1.8);
    \node at (3.4,-3.4) {$h$};
    \draw (3.4,-3.4) circle (0.25);
    \node at (4.1,-2.4) {$h$};
    \node[red] at (3.7,-1) {$\alpha$};
    \node[red] at (4.6,-3.4) {$J$ or $J_0$};
    \node at (5.2,-1.7) {$\de T_0$};
    \draw[->] (5.2,-1.9) -- (4.9,-2.2);
\end{tikzpicture}

%% file: Inserts/Fig2.10.tex
\begin{tikzpicture}
    \node at (0.17,-0.85) {\includegraphics[scale=0.5]{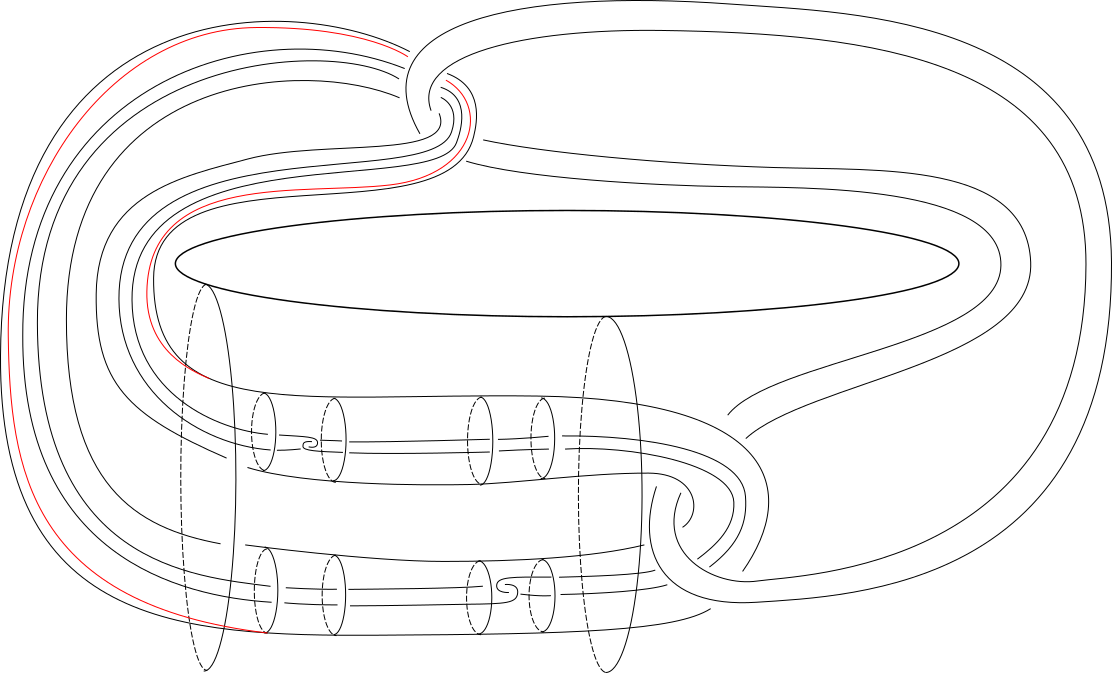}};
    
    \node at (5,3.5) {$T_s$};
    \node at (2,0.5) {$T_m$};
    \node at (-6.4,2.4) {$T$};
    \node at (-6.8,-4) {$T_0$};
    \draw[->] (-6.6,-3.9) -- (-6.2,-3.4);
    \node[red] at (-6.1,-4.6) {$\alpha$};
    \draw[red,->] (-6.05,-4.4) -- (-5.9,-4.05);
    \node at (2.1,-1.5) {$T_1$};
    \draw[->] (2,-1.7) -- (1.8,-2.2);
    
    \draw (-4.45,4) -- (-4.45,-5.5);
    \node at (-4.45,-5.75) {$P$};
    \draw (0.85,4) -- (0.85,-5.5);
    \node at (0.85,-5.75) {$Q$};
    \draw (4.3,4) -- (4.3,-5.5);
    \node at (4.3,-5.75) {$R$};
    
    \draw (-3.65,0) -- (-3.65,-5.5);
    \node at (-3.65,-5.75) {$P^\pr$};
    \draw (-2.75,0) -- (-2.75,-5.5);
    \node at (-2.75,-5.75) {$P^{\pr\pr}$};
    \draw (-0.8,0) -- (-0.8,-5.5);
    \node at (-0.8,-5.75) {$Q^{\pr\pr}$};
    \draw (0,0) -- (0,-5.5);
    \node at (0,-5.75) {$Q^\pr$};
    
    \node at (-1.75,-5.1) {lots of};
    \node at (-1.75,-5.5) {space};
    \draw[->] (-1.2,-5.3) -- (-0.8,-5.3);
    \draw[->] (-2.25,-5.3) -- (-2.75,-5.3);
    
    \node at (-3.3,-1.35) {$\Delta_i^{P^\pr}$};
    \node at (-2.35,-1.4) {$\Delta_i^{P^{\pr\pr}}$};
    \node at (-1.15,-1.4) {$\Delta_f^{Q^{\pr\pr}}$};
    \node at (-0.3,-1.4) {$\Delta_f^{Q^\pr}$};
    \node at (-3.3,-3.5) {$\Delta_f^{P^\pr}$};
    \node at (-2.35,-3.6) {$\Delta_f^{P^{\pr\pr}}$};
    \node at (-1.1,-3.6) {$\Delta_i^{Q^{\pr\pr}}$};
    \node at (-0.3,-3.55) {$\Delta_i^{Q^\pr}$};
    
    \node at (-5.05,-5.1) {$\vphantom{\Delta}_m \Delta_f^P$};
    \node at (1.5,-5.1) {$\vphantom{\Delta}_m \Delta_i^Q$};

    \draw[->] (-3.05,-2.05) -- (-3.35,-2.05);
    \node at (-3.2,-1.85) {\footnotesize{1}};
    \draw[->] (-0.55,-4.45) -- (-0.25,-4.45);
    \node at (-0.4,-4.65) {\footnotesize{2}};
\end{tikzpicture}

%% file: Inserts/Fig2.11.tex
\begin{tikzpicture}
    \node at (0,0.3) {\includegraphics[scale=0.5]{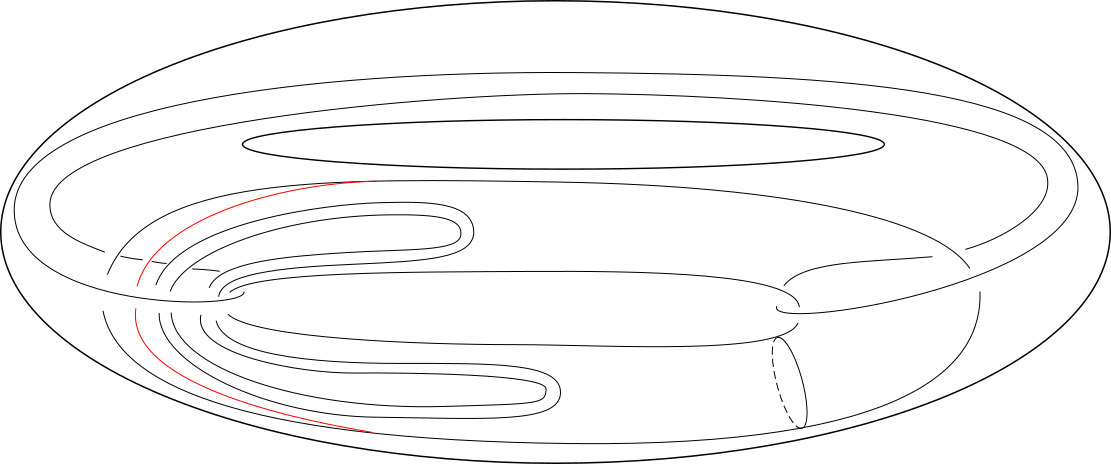}};
    
    \draw (-3.3,1.27) -- (-3.3,-2.42);
    \node at (-3.3,-3) {$P$};
    \draw (-2.4,0.97) -- (-2.4,-2.6);
    \node at (-2.4,-3) {$P^\pr$};
    \draw (-1.5,0.98) -- (-1.5,-2.7);
    \node at (-1.5,-3) {$P^{\pr\pr}$};
    
    \node at (-3.9,-0.7) {\footnotesize{$\vphantom{\Delta}_m \Delta_f^P$}};
    \draw[->] (-3.55,-0.7) -- (-3.35,-0.7);
    \node at (-2.85,-1.8) {\footnotesize{$\Delta_f^{P^\pr}$}};
    \draw[->] (-2.65,-1.8) -- (-2.45,-1.8);
    \node at (-1.95,-1.8) {\footnotesize{$\Delta_f^{P^{\pr\pr}}$}};
    \draw[->] (-1.75,-1.8) -- (-1.55,-1.8);
    \node at (-2.85,0.25) {\footnotesize{$\Delta_i^{P^\pr}$}};
    \draw[->] (-2.65,0.25) -- (-2.45,0.25);
    \node at (-1.95,0.3) {\footnotesize{$\Delta_i^{P^{\pr\pr}}$}};
    \draw[->] (-1.75,0.3) -- (-1.55,0.3);
    
    \node at (4,-1.3) {$T$};
    \node at (6.6,-0.35) {$T_s$};
    \draw[->] (6.4,-0.35) to[out=180,in=-90] (6.2,-0.1);
    \node[red] at (-5.75,-0.7) {$\alpha$};
    
    \node at (-1.13,0.13) {\large{$\ast$}};
    \node at (-0.6,0.13) {BCS};
    \node at (-0.35,-1) {$T_0$};
    \draw[->] (-0.6,-1.1) -- (-1,-1.4);
    
    \draw (0.5,3.35) arc (90:-90:0.25 and 0.78);
    \draw[dashed] (0.5,3.35) arc (90:270:0.25 and 0.78);
    \node at (0.5,3.6) {$T_m$};
    
    \draw[very thick] (-1.5,-2.17) -- (-1.5,-2);
    \draw[very thick] (-1.5,-1.58) -- (-1.5,-1.45);
    
    \draw[very thick] (-1.5,0.65) -- (-1.5,0.5);
    \draw[very thick] (-1.5,0.11) -- (-1.5,-0.03);
\end{tikzpicture}

%% file: Inserts/Fig2.12.tex
\begin{tikzpicture}[scale=1.1]
    \node at (0,6.25) {\includegraphics[scale=0.715]{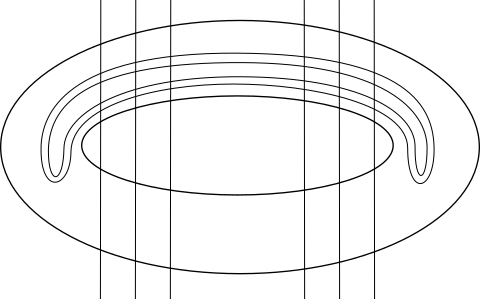}};
    \node at (-2.4,3.5) {$P$};
    \node at (-1.8,3.5) {$P^\pr$};
    \node at (-1.2,3.5) {$P^{\pr\pr}$};
    \node at (1.15,3.48) {$Q^{\pr\pr}$};
    \node at (1.75,3.48) {$Q^\pr$};
    \node at (2.3,3.48) {$Q$};
    \node at (4.2,4.6) {$T$ in monotone};
    \node at (4.2,4.25) {picture};
    \draw[->] (4.1,4.8) -- (3.7,5.2);
    \node at (-3.5,8.45) {$T_0$ in precise};
    \node at (-3.5,8.05) {position};
    \draw[->] (-2.8,8.2) -- (-2.1,7.75);
    \node at (0,3.25) {(a)};
    
    \node at (0,0) {\includegraphics[scale=0.715]{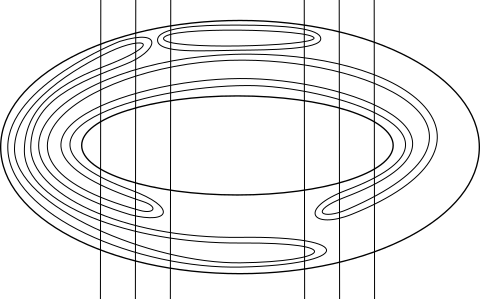}};
    \node at (-2.4,-2.75) {$P$};
    \node at (-1.8,-2.75) {$P^\pr$};
    \node at (-1.2,-2.75) {$P^{\pr\pr}$};
    \node at (1.2,-2.77) {$Q^{\pr\pr}$};
    \node at (1.75,-2.77) {$Q^\pr$};
    \node at (2.3,-2.77) {$Q$};
    \node at (4.2,-1.7) {$T$ in monotone};
    \node at (4.2,-2.05) {picture};
    \draw[->] (4.1,-1.5) -- (3.7,-1.1);
    \node at (0,-3.3) {(b) Possible positions for $T_0$ in rather precise positions.};
    
    \node at (0,-6.6) {\includegraphics[scale=0.715]{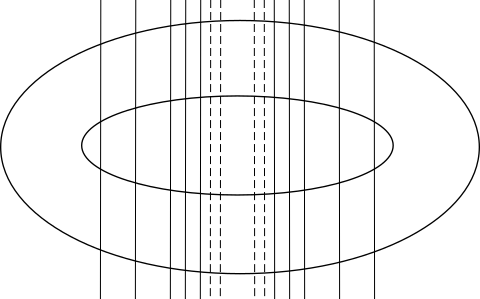}};
    \node at (-2.4,-9.35) {$P$};
    \node at (-1.8,-9.35) {$P^\pr$};
    \node at (-1.2,-9.35) {$P^{\pr\pr}$};
    \node[rotate around = {90:(0,0)}] at (-1.2,-9.65) {$=$};
    \node at (-1.2,-10) {$P_1$};
    \node at (-0.8,-9.7) {$P_1^\pr$};
    \draw[->] (-0.82,-9.48) -- (-0.9,-9.2);
    \node at (-0.3,-9.5) {$P_1^{\pr\pr}$};
    \draw[->] (-0.45,-9.3) -- (-0.65,-9.15);
    \node at (0.3,-9.5) {$Q_1^{\pr\pr}$};
    \draw[->] (0.35,-9.3) -- (0.55,-9.15);
    \node at (0.7,-9.7) {$Q_1^\pr$};
    \draw[->] (0.7,-9.5) -- (0.8,-9.2);
    \node at (1.2,-9.4) {$Q^{\pr\pr}$};
    \node[rotate around = {90:(0,0)}] at (1.15,-9.73) {$=$};
    \node at (1.2,-10.05) {$Q_1$};
    \node at (1.75,-9.4) {$Q^\pr$};
    \node at (2.3,-9.4) {$Q$};

    \node at (4.5,-8.4) {$T$ with the next stage};
    \node at (4.5,-8.75) {of plane-pairs added};
    \node at (0,-10.3) {(c)};

    \node at (-6,0) {\vspace{1em}};
\end{tikzpicture}

%% file: Inserts/Fig2.13.tex
\begin{tikzpicture}
    \node at (0,0) {\includegraphics[scale=0.3]{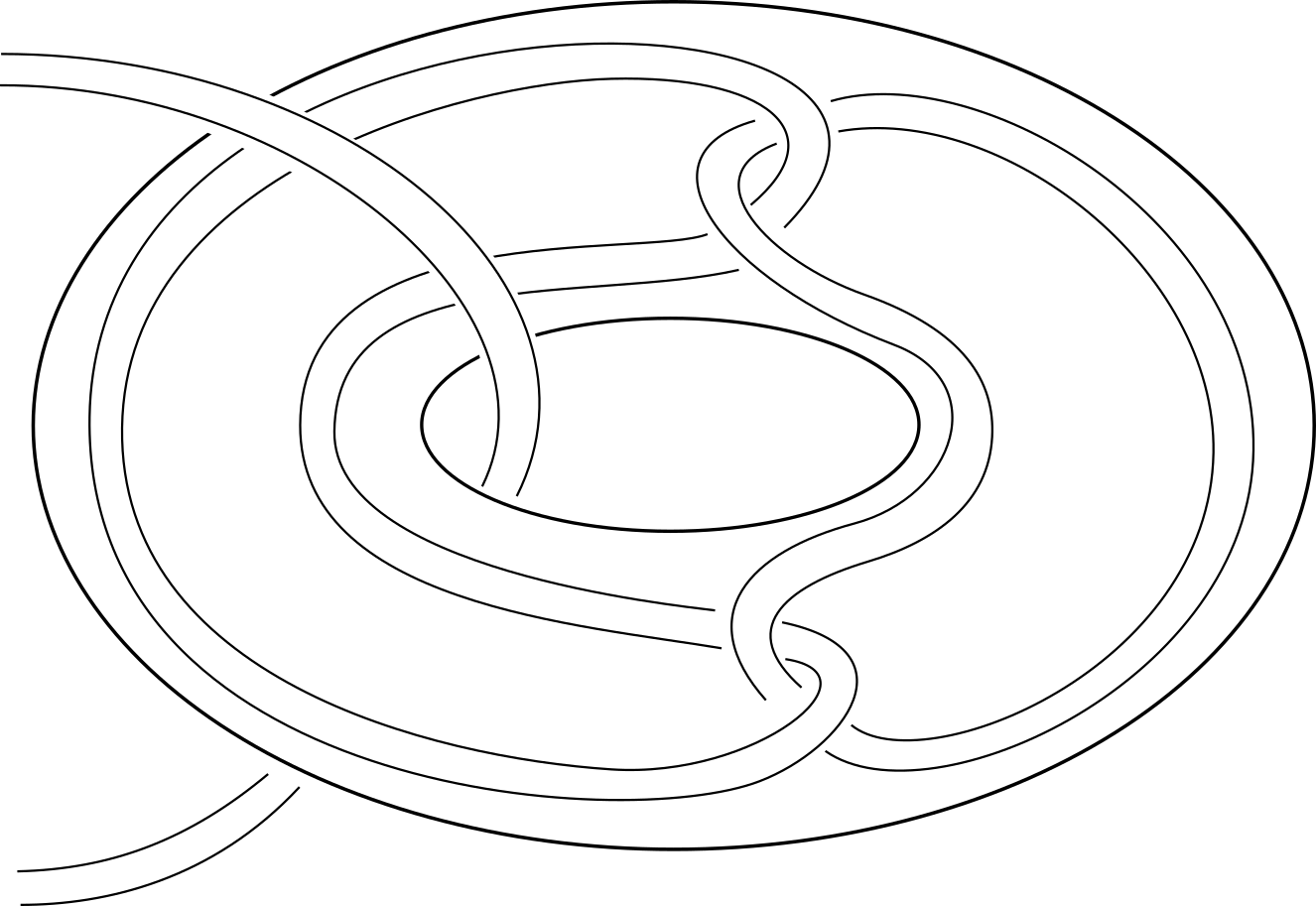}};
    
    \draw (0.25,-3.7) -- (0.25,3.7);
    \node at (0.25,-3.95) {$P^{\pr\pr}$};
    \draw (1.8,-3.7) -- (1.8,3.7);
    \node at (1.8,-3.95) {$Q^{\pr\pr}$};
    \draw[<->] (0.4,-3.35) -- (1.6,-3.35);
    \node at (1,-3.55) {\footnotesize{$>>d$}};
    \draw[<->] (-0.7,-3.35) -- (0.1,-3.35);
    \node at (-0.3,-3.55) {\footnotesize{$d$}};
    \draw (-0.8,-3.7) -- (-0.8,3.7);
    \node at (-0.8,-3.95) {$P^\pr$};
    \draw[<->] (-1.3,3.8) -- (0.8,3.8);
    \node at (-0.25,4.45) {\footnotesize{displacement of daughter's turning}};
    \node at (-0.25,4.1) {\footnotesize{points from mother's turning point}};
    
    \draw[red,thick] (-0.8,1.15) -- (-1.05,1.1);
    \draw[red,fill=red] (-0.8,1.15) circle (0.25ex);
    \draw[red,thick] (-1.5,1) to[out=200,in=90] (-2.2,0.25) to[out=-90,in=165] (-0.8,-0.8);
    \draw[red,fill=red] (-0.8,-0.8) circle (0.25ex);
    \node[red] at (-2.4,0.25) {$\alpha$};
\end{tikzpicture}

%% file: Inserts/Fig2.14.tex
\begin{tikzpicture}
    \node at (0,0) {\includegraphics[scale=0.55]{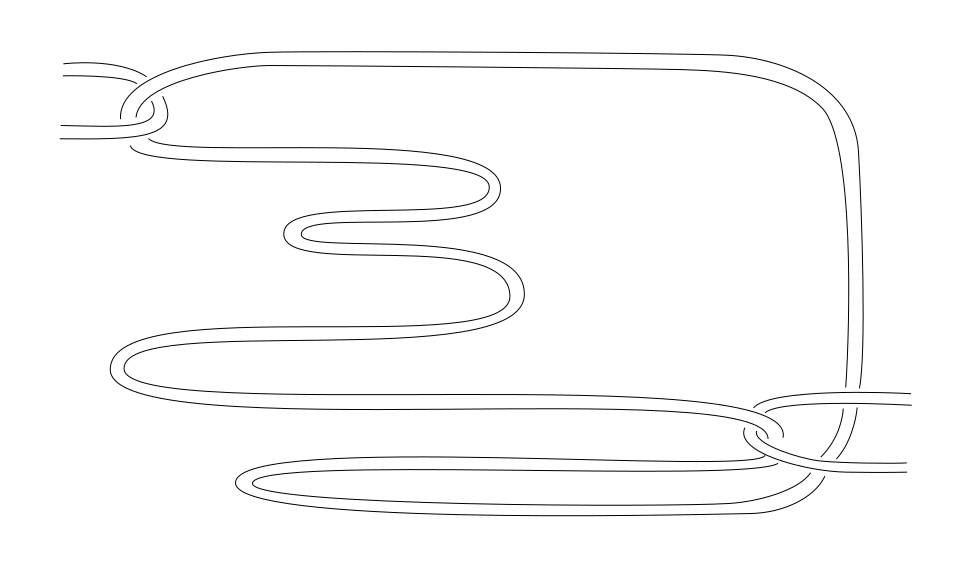}};

    \draw (-4.2,-3.9) -- (-4.2,3.9);
    \node at (-4.2,-4.1) {$P$};
    \draw (-0.5,-3.9) -- (-0.5,3.9);
    \draw (3.3,-3.9) -- (3.3,3.9);
    \node at (3.3,-4.1) {$Q$};

    \node at (-5.5,3.6) {$T_1$};
    \draw (-6.06,3.25) ellipse (0.04 and 0.09);
    \draw (-6.1,2.35) ellipse (0.04 and 0.09);

    \draw (-4.2,3.23) ellipse (0.05 and 0.1);
    \node at (-3.85,3.7) {$\Delta_i^-$};
    \draw (-4.2,2.03) ellipse (0.05 and 0.1);
    \node at (-3.85,1.6) {$\Delta_f^-$};
    \draw (-3.7,3.35) ellipse (0.05 and 0.1);
    \draw (-3.2,3.4) ellipse (0.05 and 0.1);
    \draw (-2.9,3.4) ellipse (0.05 and 0.1);
    \draw (-2.6,3.4) ellipse (0.05 and 0.1);
    \node at (-2.1,3.4) {$\cdots$};
    \draw (-3.7,2.02) ellipse (0.05 and 0.1);
    \draw (-3.2,2) ellipse (0.05 and 0.1);
    \draw (-2.9,2) ellipse (0.05 and 0.1);
    \draw (-2.6,2) ellipse (0.05 and 0.1);
    \node at (-2.1,2) {$\cdots$};

    \draw (-0.5,3.4) ellipse (0.05 and 0.1);
    \draw (0,3.4) ellipse (0.05 and 0.1);
    \draw (0.5,3.4) ellipse (0.05 and 0.1);
    \draw (0.8,3.4) ellipse (0.05 and 0.1);
    \draw (1.1,3.4) ellipse (0.05 and 0.1);
    \node at (1.5,3.37) {$\cdots$};

    \draw (-0.5,-0.52) ellipse (0.04 and 0.08);
    \draw (-1,-0.56) ellipse (0.04 and 0.08);
    \draw (-1.5,-0.57) ellipse (0.045 and 0.09);
    \draw (-1.8,-0.58) ellipse (0.05 and 0.1);
    \draw (-2.1,-0.58) ellipse (0.05 and 0.1);
    \node at (-2.5,-0.6) {$\cdots$};
    \draw (-0.5,-1.59) ellipse (0.05 and 0.1);
    \draw (-1,-1.59) ellipse (0.05 and 0.1);
    \draw (-1.5,-1.59) ellipse (0.05 and 0.1);
    \draw (-1.8,-1.59) ellipse (0.05 and 0.1);
    \draw (-2.1,-1.59) ellipse (0.05 and 0.1);
    \node at (-2.5,-1.6) {$\cdots$};

    \draw (-0.5,-3.13) ellipse (0.04 and 0.08);
    \draw (0,-3.14) ellipse (0.04 and 0.08);
    \draw (0.5,-3.15) ellipse (0.04 and 0.08);
    \draw (0.8,-3.16) ellipse (0.04 and 0.08);
    \draw (1.1,-3.16) ellipse (0.04 and 0.08);
    \node at (1.5,-3.17) {$\cdots$};

    \draw (3.3,-1.68) ellipse (0.04 and 0.08);
    \draw (2.8,-1.63) ellipse (0.04 and 0.08);
    \draw (2.3,-1.61) ellipse (0.045 and 0.09);
    \draw (2.0,-1.6) ellipse (0.045 and 0.09);
    \draw (1.7,-1.59) ellipse (0.045 and 0.09);
    \node at (1.2,-1.61) {$\cdots$};
    \draw (3.3,-2.52) ellipse (0.04 and 0.07);
    \draw (2.8,-2.52) ellipse (0.04 and 0.08);
    \draw (2.3,-2.52) ellipse (0.04 and 0.08);
    \draw (2.0,-2.52) ellipse (0.04 and 0.08);
    \draw (1.7,-2.52) ellipse (0.04 and 0.08);
    \node at (1.2,-2.52) {$\cdots$};
    \node at (3,-1.2) {$\Delta_i^+$};
    \node at (3,-2.85) {$\Delta_f^+$};

    \draw (6.28,-1.54) ellipse (0.04 and 0.08);
    \draw (6.2,-2.54) ellipse (0.04 and 0.07);
    \node at (5.7,-2.9) {$T_1$};
    \node at (5.1,0) {$T_0$};
\end{tikzpicture}

%% file: Inserts/Fig2.15.tex
\begin{tikzpicture}[scale=1.05]
    \draw (-7,1.3) -- (-7,-1.3);
    \draw (-5.5,1.3) -- (-5.5,-1.3);
    \draw (-4,1.3) -- (-4,-1.3);
    \draw (-3,1.3) -- (-3,-1.3);
    \draw (-2,1.3) -- (-2,-1.3);
    \draw (-1.75,1.3) -- (-1.75,-1.3);
    \draw (-1.5,1.3) -- (-1.5,-1.3);
    \node at (-1,0) {$\dots$};
    \draw (7,1.3) -- (7,-1.3);
    \draw (5.5,1.3) -- (5.5,-1.3);
    \draw (4,1.3) -- (4,-1.3);
    \draw (3,1.3) -- (3,-1.3);
    \draw (2,1.3) -- (2,-1.3);
    \draw (1.75,1.3) -- (1.75,-1.3);
    \draw (1.5,1.3) -- (1.5,-1.3);
    \node at (1,0) {$\dots$};
    \draw (-0.5,-1.3) -- (-0.5,1.3);
    \draw (0.5,-1.3) -- (0.5,1.3);
    \node at (0,0) {\footnotesize{$> 0.09$}};
    \draw[<->] (-0.4,-0.4) -- (0.4,-0.4);
    \node at (-0.5,1.6) {$\lbar{P}$};
    \node at (0.5,1.6) {$\lbar{Q}$};
    
    \node at (-5.5,1.8) {$\frac{1}{(n+1)L(n+1)}$};
    \draw[->] (-5.9,1.4) -- (-6.2,1);
    \draw[->] (-5.1,1.4) -- (-4.8,1);
    \draw[<->] (-6.9,0.8) -- (-5.6,0.8);
    \draw[<->] (-5.4,0.8) -- (-4.1,0.8);
    \node at (-3,1.8) {$\frac{1}{(n+2)L(n+2)}$};
    \draw[->] (-3.4,1.4) -- (-3.7,1);
    \draw[->] (-2.6,1.4) -- (-2.3,1);
    \draw[<->] (-3.9,0.8) -- (-3.1,0.8);
    \draw[<->] (-2.9,0.8) -- (-2.1,0.8);
    \node at (5.5,1.8) {$\frac{1}{(n+1)L(n+1)}$};
    \draw[->] (5.9,1.4) -- (6.2,1);
    \draw[->] (5.1,1.4) -- (4.8,1);
    \draw[<->] (3.9,0.8) -- (3.1,0.8);
    \draw[<->] (2.9,0.8) -- (2.1,0.8);
    \node at (3,1.8) {$\frac{1}{(n+2)L(n+2)}$};
    \draw[->] (3.4,1.4) -- (3.7,1);
    \draw[->] (2.6,1.4) -- (2.3,1);
    \draw[<->] (6.9,0.8) -- (5.6,0.8);
    \draw[<->] (5.4,0.8) -- (4.1,0.8);
    
    \draw[decorate,decoration={brace,amplitude=5pt}] (7,-1.5) -- (-7,-1.5) node[midway,yshift=-15pt] {$\frac{1}{2}$};
    \node at (-7,-1.9) {$P$};
    \node at (-5.5,-1.9) {$P^\pr$};
    \node at (-4,-1.9) {$P^{\pr\pr}$};
    \node[rotate=90] at (-4,-2.3) {$=$};
    \node at (-4,-2.75) {$P_1$};
    \node at (-3,-1.95) {$P_1^\pr$};
    \node at (-2.1,-1.95) {$P_1^{\pr\pr}$};
    \node[rotate=90] at (-2.1,-2.35) {$=$};
    \node at (-2.1,-2.75) {$P_2$};
    \node at (-1.7,-1.95) {$P_2^\pr$};
    \node at (-1.3,-1.95) {$P_2^{\pr\pr}$};
    \node at (7,-1.9) {$Q$};
    \node at (5.5,-1.9) {$Q^\pr$};
    \node at (4,-1.9) {$Q^{\pr\pr}$};
    \node[rotate=90] at (4,-2.3) {$=$};
    \node at (4,-2.75) {$Q_1$};
    \node at (3,-1.95) {$Q_1^\pr$};
    \node at (2.1,-1.95) {$Q_1^{\pr\pr}$};
    \node[rotate=90] at (2.1,-2.35) {$=$};
    \node at (2.1,-2.75) {$Q_2$};
    \node at (1.65,-1.95) {$Q_2^\pr$};
    \node at (1.2,-1.95) {$Q_2^{\pr\pr}$};
\end{tikzpicture}

%% file: Inserts/Fig2.16.tex
\begin{tikzpicture}
    \node at (0,0) {\includegraphics[scale=0.8]{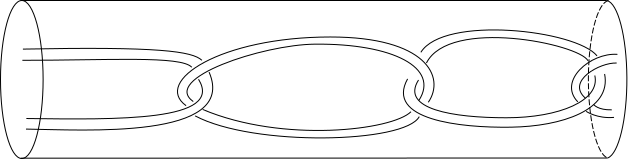}};
    \draw[decorate, decoration={brace,amplitude=10pt}] (-6.9,-1.7) -- (-6.9,1.7) node[midway, xshift=-1.75em] {$\tld{T}$};
    
    \draw[white, line width = 0.5ex] (-5.2,1.6) arc (90:-90:0.25 and 1.6);
    \draw (-5.2,1.67) arc (90:-90:0.25 and 1.67);
    \draw[dashed] (-5.2,1.67) arc (90:270:0.25 and 1.67);
    \node at (-5.1,-2) {$\Delta_f^{Q^\pr}$};
    
    \draw[white, line width = 0.5ex] (-4.4,1.6) arc (90:-90:0.25 and 1.6);
    \draw (-4.4,1.67) arc (90:-90:0.25 and 1.67);
    \draw[dashed] (-4.4,1.67) arc (90:270:0.25 and 1.67);
    \node at (-4.3,-2) {$\Delta_f^{Q^{\pr\pr}}$};
    
    \draw[white, line width = 0.5ex] (-3.6,1.6) arc (90:-90:0.25 and 1.6);
    \draw (-3.6,1.67) arc (90:-90:0.25 and 1.67);
    \draw[dashed] (-3.6,1.67) arc (90:270:0.25 and 1.67);
    \node at (-3.5,-2) {$\Delta_i^{P^{\pr\pr}}$};
    
    \node at (-2.8,0.8) {$\tld{T}_1$};

    \draw[white, line width = 0.5ex] (-1.8,1.6) arc (90:-90:0.25 and 1.6);
    \draw (-1.8,1.67) arc (90:-90:0.25 and 1.67);
    \draw[dashed] (-1.8,1.67) arc (90:270:0.25 and 1.67);
    \node at (-1.7,-2) {$\Delta_i^{P^\pr}$};
    
    \draw[white, line width = 0.5ex] (-1.1,1.6) arc (90:-90:0.25 and 1.6);
    \draw (-1.1,1.67) arc (90:-90:0.25 and 1.67);
    \draw[dashed] (-1.1,1.67) arc (90:270:0.25 and 1.67);
    \draw[white, line width = 0.5ex] (-0.4,1.6) arc (90:-90:0.25 and 1.6);
    \draw (-0.4,1.67) arc (90:-90:0.25 and 1.67);
    \draw[dashed] (-0.4,1.67) arc (90:270:0.25 and 1.67);
    \node at (-0.8,-2.4) {$\subset \vphantom{\Delta}_m\Delta_f^P$};
    \draw[<->] (-0.4,-1.75) -- (-0.75,-2.2) -- (-1.1,-1.75);
    
    \draw[white, line width = 0.5ex] (0.3,1.6) arc (90:-90:0.25 and 1.6);
    \draw (0.3,1.67) arc (90:-90:0.25 and 1.67);
    \draw[dashed] (0.3,1.67) arc (90:270:0.25 and 1.67);
    \node at (0.4,-2) {$\Delta_f^{P^\pr}$};
    
    \draw[white, line width = 0.5ex] (1,1.6) arc (90:-90:0.25 and 1.6);
    \draw (1,1.67) arc (90:-90:0.25 and 1.67);
    \draw[dashed] (1,1.67) arc (90:270:0.25 and 1.67);
    \node at (1.1,-2) {$\Delta_f^{P^{\pr\pr}}$};
    
    \draw[white, line width = 0.5ex] (1.7,1.6) arc (90:-90:0.25 and 1.6);
    \draw (1.7,1.67) arc (90:-90:0.25 and 1.67);
    \draw[dashed] (1.7,1.67) arc (90:270:0.25 and 1.67);
    \node at (1.8,-2) {$\Delta_i^{Q^{\pr\pr}}$};
    
    \node at (1.7,1) {$\tld{T}_0$};

    \draw[white, line width = 0.5ex] (2.85,1.6) arc (90:-90:0.25 and 1.6);
    \draw (2.85,1.67) arc (90:-90:0.25 and 1.67);
    \draw[dashed] (2.85,1.67) arc (90:270:0.25 and 1.67);
    \node at (2.95,-2) {$\Delta_i^{Q^\pr}$};
    
    \draw[white, line width = 0.5ex] (3.5,1.6) arc (90:-90:0.25 and 1.6);
    \draw (3.5,1.67) arc (90:-90:0.25 and 1.67);
    \draw[dashed] (3.5,1.67) arc (90:270:0.25 and 1.67);
    \draw[white, line width = 0.5ex] (4.15,1.6) arc (90:-90:0.25 and 1.6);
    \draw (4.15,1.67) arc (90:-90:0.25 and 1.67);
    \draw[dashed] (4.15,1.67) arc (90:270:0.25 and 1.67);
    \node at (3.82,-2.4) {$\subset \vphantom{\Delta}_m \Delta_i^Q$};
    \draw[<->] (3.5,-1.75) -- (3.82, -2.2) -- (4.15,-1.75);
    
    \draw[white, line width = 0.5ex] (4.8,1.6) arc (90:-90:0.25 and 1.6);
    \draw (4.8,1.67) arc (90:-90:0.25 and 1.67);
    \draw[dashed] (4.8,1.67) arc (90:270:0.25 and 1.67);
    \node at (4.9,-2) {$\Delta_f^{Q^\pr}$};
    
    \draw[white, line width = 0.5ex] (5.45,1.6) arc (90:-90:0.25 and 1.6);
    \draw (5.45,1.67) arc (90:-90:0.25 and 1.67);
    \draw[dashed] (5.45,1.67) arc (90:270:0.25 and 1.67);
    \node at (5.55,-2) {$\Delta_f^{Q^{\pr\pr}}$};
    
    \node at (5.3,1.15) {$\tld{T}_1$};
    
    \node at (6.4,-2) {$\Delta_i^{P^{\pr\pr}}$};
    
    \draw[red] (-1.45,1.3) to[out=5,in=175] (0.45,1.3);
    \node[red] at (-0.5, 1.6) {$\tld{\alpha}$};
    \draw[red,fill=red] (-1.45,1.3) circle (0.15ex);
    \draw[red,fill=red] (0.45,1.3) circle (0.15ex);
    
    \node at (7.2,0) {\hspace{0.25em}};
\end{tikzpicture}

%% file: Inserts/Fig2.17.tex
\begin{tikzpicture}
    \node at (0,0) {\includegraphics[scale=0.8]{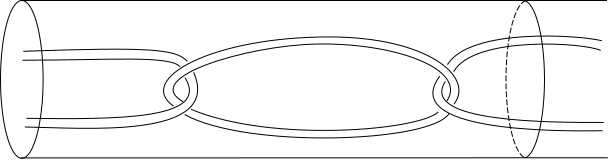}};
    
    \node at (-6.9,0) {$\tld{T}_0$};
    \node at (5.8,1.2) {$\tld{T}_{01}$};
    \node at (-5,0.9) {$\tld{T}_{01}$};
    
    \draw[white, line width=0.5ex] (-4,1.6) arc (90:-90:0.25 and 1.6);
    \draw (-4,1.7) arc (90:-90:0.25 and 1.67);
    \draw[dashed] (-4,1.7) arc (90:270:0.25 and 1.67);
    \node at (-3.9,-2) {$\Delta_i^{P^{\pr\pr}}$};

    \draw[white, line width=0.5ex] (-2.25,1.6) arc (90:-90:0.2 and 1.6);
    \draw (-2.25,1.7) arc (90:-90:0.2 and 1.67);
    \draw[dashed] (-2.25,1.7) arc (90:270:0.2 and 1.67);
    \node at (-2.25,-2) {$P^\pr$};
    
    \draw[white, line width=0.5ex] (-1.7,1.6) arc (90:-90:0.25 and 1.6);
    \draw (-1.7,1.7) arc (90:-90:0.25 and 1.67);
    \draw[dashed] (-1.7,1.7) arc (90:270:0.25 and 1.67);
    \node at (-1.7,-2) {$P_1$};
    \draw (-1.7,0.57) ellipse (0.05 and 0.09);
    \draw (-1.7,-0.93) ellipse (0.05 and 0.09);
    
    \draw[white, line width=0.5ex] (-1,1.6) arc (90:-90:0.25 and 1.6);
    \draw (-1,1.7) arc (90:-90:0.25 and 1.67);
    \draw[dashed] (-1,1.7) arc (90:270:0.25 and 1.67);
    \node at (-1,-2) {$P_1^\pr$};
    \draw (-1,0.72) ellipse (0.05 and 0.09);
    \node at (-0.9,0.3) {\footnotesize{$\vphantom{\Delta}_o \Delta_i^{P_1^\pr}$}};
    \draw (-1,-1.06) ellipse (0.05 and 0.09);
    \node at (-0.9,-1.45) {\footnotesize{$\vphantom{\Delta}_o \Delta_f^{P_1^\pr}$}};
    
    \draw[white, line width=0.5ex] (-0.3,1.6) arc (90:-90:0.25 and 1.6);
    \draw (-0.3,1.7) arc (90:-90:0.25 and 1.67);
    \draw[dashed] (-0.3,1.7) arc (90:270:0.25 and 1.67);
    \node at (-0.3,-2) {$P_1^{\pr\pr}$};
    \draw (-0.3,0.82) ellipse (0.05 and 0.08);
    \node at (-0.2,0.3) {\footnotesize{$\vphantom{\Delta}_o \Delta_i^{P_1^{\pr\pr}}$}};
    \draw (-0.3,-1.12) ellipse (0.05 and 0.08);
    \node at (-0.2,-1.45) {\footnotesize{$\vphantom{\Delta}_o \Delta_f^{P_1^{\pr\pr}}$}};
    
    \draw[white, line width=0.5ex] (0.8,1.6) arc (90:-90:0.25 and 1.6);
    \draw (0.8,1.7) arc (90:-90:0.25 and 1.67);
    \draw[dashed] (0.8,1.7) arc (90:270:0.25 and 1.67);
    \node at (0.8,-2) {$Q_1^{\pr\pr}$};
    \draw (0.8,0.83) ellipse (0.05 and 0.08);
    \node at (0.9,0.3) {\footnotesize{$\vphantom{\Delta}_o \Delta_f^{Q_1^{\pr\pr}}$}};
    \draw (0.8,-1.13) ellipse (0.05 and 0.08);
    \node at (0.9,-1.45) {\footnotesize{$\vphantom{\Delta}_o \Delta_i^{Q_1^{\pr\pr}}$}};
    
    \draw[white, line width=0.5ex] (1.5,1.6) arc (90:-90:0.25 and 1.6);
    \draw (1.5,1.7) arc (90:-90:0.25 and 1.67);
    \draw[dashed] (1.5,1.7) arc (90:270:0.25 and 1.67);
    \node at (1.5,-2) {$Q_1^\pr$};
    \draw (1.5,0.76) ellipse (0.05 and 0.08);
    \node at (1.6,0.3) {\footnotesize{$\vphantom{\Delta}_o \Delta_f^{Q_1^\pr}$}};
    \draw (1.5,-1.08) ellipse (0.05 and 0.08);
    \node at (1.6,-1.45) {\footnotesize{$\vphantom{\Delta}_o \Delta_i^{Q_1^\pr}$}};
    
    \draw[white, line width=0.5ex] (2.2,1.6) arc (90:-90:0.25 and 1.6);
    \draw (2.2,1.7) arc (90:-90:0.25 and 1.67);
    \draw[dashed] (2.2,1.7) arc (90:270:0.25 and 1.67);
    \node at (2.8,-2) {$Q_1 \supset \Delta_i^{Q^{\pr\pr}}$};
    \draw (2.2,0.59) ellipse (0.05 and 0.09);
    \draw (2.2,-0.97) ellipse (0.05 and 0.08);
    
    \draw[red] (-1.85,1.4) to[out=5,in=175] (-0.85,1.4);
    \draw[red,fill=red] (-1.85,1.4) circle (0.15ex);
    \draw[red,fill=red] (-0.85,1.4) circle (0.15ex);
    \node[red] at (-1.35,1.7) {$\tld{\alpha}_0$};
    
    \node at (0.3,1.2) {$\tld{T}_{00}$};
\end{tikzpicture}

%% file: Inserts/Fig2.18.tex
\begin{tikzpicture}
    \draw (-6.5,0.9) -- (-6.5,-1.35);
    \node at (-6.5,-1.65) {\footnotesize{$Q_1$}};
    \draw (-6.05,0.9) -- (-6.05,-1.35);
    \node at (-6.05,-1.65) {\footnotesize{$Q_1^{\prime}$}};
    \draw (-5.6,0.9) -- (-5.6,-1.35);
    \node at (-5.6,-1.65) {\footnotesize{$Q_1^{\prime\prime}$}};
    
    \draw (-4.6,0.9) -- (-4.6,-1.35);
    \node at (-4.6,-1.65) {\footnotesize{$P_1^{\prime\prime}$}};
    \draw (-4.15,0.9) -- (-4.15,-1.35);
    \node at (-4.15,-1.65) {\footnotesize{$P_1^{\prime}$}};
    \draw (-3.7,0.9) -- (-3.7,-1.35);
    \node at (-3.7,-1.65) {\footnotesize{$P_1$}};
    
    \draw (-2.6,0.9) -- (-2.6,-1.35);
    \node at (-2.6,-1.65) {\footnotesize{$P_1$}};
    \draw (-2.15,0.9) -- (-2.15,-1.35);
    \node at (-2.15,-1.65) {\footnotesize{$P_1^{\prime}$}};
    \draw (-1.7,0.9) -- (-1.7,-1.35);
    \node at (-1.7,-1.65) {\footnotesize{$P_1^{\prime\prime}$}};
    
    \draw (-0.6,0.9) -- (-0.6,-1.35);
    \node at (-0.6,-1.65) {\footnotesize{$Q_1^{\prime\prime}$}};
    \draw (-0.15,0.9) -- (-0.15,-1.35);
    \node at (-0.15,-1.65) {\footnotesize{$Q_1^{\prime}$}};
    \draw (0.3,0.9) -- (0.3,-1.35);
    \node at (0.3,-1.65) {\footnotesize{$Q_1$}};
    
    \draw (1.5,0.9) -- (1.5,-1.35);
    \node at (1.5,-1.65) {\footnotesize{$Q_1$}};
    \draw (1.95,0.9) -- (1.95,-1.35);
    \node at (1.95,-1.65) {\footnotesize{$Q_1^{\prime}$}};
    \draw (2.4,0.9) -- (2.4,-1.35);
    \node at (2.4,-1.65) {\footnotesize{$Q_1^{\prime\prime}$}};
    
    \draw (3.6,0.9) -- (3.6,-1.35);
    \node at (3.6,-1.65) {\footnotesize{$P_1^{\prime\prime}$}};
    \draw (4.05,0.9) -- (4.05,-1.35);
    \node at (4.05,-1.65) {\footnotesize{$P_1^{\prime}$}};
    \draw (4.5,0.9) -- (4.5,-1.35);
    \node at (4.5,-1.65) {\footnotesize{$P_1$}};
    
    \draw (5.6,0.9) -- (5.6,-1.35);
    \node at (5.6,-1.65) {\footnotesize{$P_1$}};
    \draw (6.05,0.9) -- (6.05,-1.35);
    \node at (6.05,-1.65) {\footnotesize{$P_1^{\prime}$}};
    \draw (6.5,0.9) -- (6.5,-1.35);
    \node at (6.5,-1.65) {\footnotesize{$P_1^{\prime\prime}$}};
    
    \node at (0,0) {\includegraphics[scale=0.65]{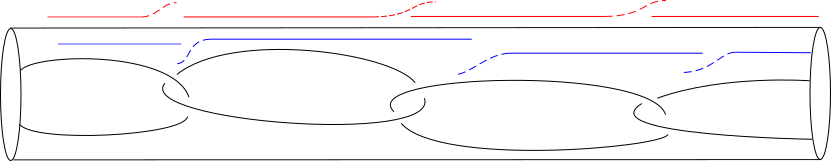}};
    \node[red] at (-5.4,1.2) {\footnotesize{$\alpha_1$}};
    \node[red] at (-2.1,1.2) {\footnotesize{$\alpha_0$}};
    \node[red] at (1.9,1.2) {\footnotesize{$\alpha_1$}};
    \node[red] at (5.5,1.2) {\footnotesize{$\alpha_0$}};
    \node[blue] at (-5,0.73) {\footnotesize{$\alpha_1$}};
    \node[blue] at (-1.7,0.77) {\footnotesize{$\alpha_0$}};
    \node[blue] at (3,0.6) {\footnotesize{$\alpha_1$}};
    \node[blue] at (6.1,0.6) {\footnotesize{$\alpha_0$}};
    \node at (-6.65,1.2) {$\tld{T}_0$};
    
    \node at (-5.5,-0.3) {\footnotesize{$T_{01}$}};
    \node at (-1.9,-0.1) {\footnotesize{$T_{00}$}};
    \node at (2,-0.6) {\footnotesize{$T_{01}$}};
\end{tikzpicture}